\documentclass[12pt, a4paper, usenames]{amsart}

\usepackage{tabularx}
\usepackage[labelsep=period]{caption}
\usepackage{xltabular}
\keepXColumns
\usepackage{booktabs}
\usepackage{longtable}

\usepackage{fullpage}
\usepackage{amsmath,amssymb}
\usepackage[colorlinks=true,citecolor=blue]{hyperref}
\usepackage{xypic}
\usepackage{xfrac}
\usepackage{tikz-cd}
\usepackage{makecell}
\usepackage{color}
\usepackage{flexisym}
\usepackage{float}

\renewcommand{\theenumi}{\alph{enumi}}

\def\ord{\mathop{\mathrm {ord}}\nolimits}

\newcommand{\Z}{\mathbb{Z}} 
\newcommand{\Q}{\mathbb{Q}}

\newcommand{\F}{\mathbb{F}}

\DeclareMathOperator{\Ind}{Ind}

\usepackage[textsize=normal]{todonotes}

\makeatletter
\providecommand\@dotsep{5}
\def\listtodoname{List of Todos}
\def\listoftodos{\@starttoc{tdo}\listtodoname}
\makeatother

\theoremstyle{plain}
\newtheorem{theorem}{Theorem}[section]
\newtheorem{lemma}[theorem]{Lemma}
\newtheorem{corollary}[theorem]{Corollary}
\newtheorem{prop}[theorem]{Proposition}
\newtheorem{prop-def}[theorem]{Proposition / Definition}
\newtheorem{sec:hyp}[theorem]{Hypothesis}
\newtheorem{algorithm}[theorem]{Algorithm}
\newtheorem{conj}[theorem]{Conjecture}

\newtheorem*{problem*}{Problem}
\newtheorem*{theorem*}{Theorem}

\theoremstyle{remark}

\newtheorem{remark}[theorem]{Remark}

\newtheorem{example}[theorem]{Example}

\newtheorem*{note*}{Note}
\newtheorem*{remark*}{Remark}
\newtheorem*{example*}{Example}

\theoremstyle{definition}
\newtheorem*{definition*}{Definition}
\newtheorem*{hypothesis*}{Hypothesis}
\newtheorem*{hypotheses*}{Hypotheses}
\newtheorem*{assumptions*}{Assumptions}
\newtheorem{definition}[theorem]{Definition}

\newcommand{\Gal}{\mathrm{Gal}}

\newcommand{\Span}{\mathrm{Span}}
\newcommand{\rank}{\mathrm{rank}}

\newcommand{\Norm}{\mathrm{Norm}}

\newcommand{\Cl}{\mathrm{Cl}}

\newcommand{\GL}{\mathrm{GL}}

\newcommand{\im}{\mathrm{im}}

\DeclareMathOperator{\poly}{poly}

\title[Algorithms for $p$-rationality and for $p$-saturation of units]{Algorithms for $p$-rationality\\ and for $p$-saturation of units}

\author{Tommy Hofmann}
\address{
Naturwissenschaftlich-Technische Fakult\"at\\
Universit\"at Siegen\\
Walter-Flex-Straße 3\\
57068 Siegen\\
Germany}
\email{tommy.hofmann@uni-siegen.de}

\author{Henri Johnston}
\address{
Department of Mathematics\\
University of Exeter\\
Exeter\\
EX4 4QE\\
United Kingdom
}
\email{H.Johnston@exeter.ac.uk}
\urladdr{https://mathematics.exeter.ac.uk/people/profile/index.php?username=hj241}

\date{Version of 14th September 2026}

\begin{document}

\maketitle

\begin{abstract}
Let $K$ be a number field.
We give new practical algorithms that determine whether
$K$ is $p$-rational. 
For real cyclotomic fields $K=\mathbb{Q}(\zeta_{n})^{+}$,
our method avoids
computing either the class number or the full unit group. 
We also use the notion of $p$-rationality to explain
the practical efficiency of an algorithm
for determining whether a subgroup 
of the unit group $\mathcal{O}_{K}^{\times}$ is $p$-saturated.
This in turn yields a new algorithm for the unconditional verification
of unit groups of number fields that substantially 
outperforms existing unconditional algorithms
in practice.
Finally, by combining 
our algorithms for determining $p$-rationality with work 
of Greenberg, we construct certain Galois representations 
$\rho : \mathrm{Gal}(\overline{\mathbb{Q}}/\mathbb{Q}) 
\rightarrow \mathrm{GL}_{n}(\mathbb{Z}_{p})$
with open image and further prescribed properties.
\end{abstract}

\section{Introduction}

This article is aimed at two main audiences: researchers interested in $p$-rationality of number fields and those interested in the unconditional computation of their unit groups.

Let $K$ be a number field with $c_{K}$ pairs of complex embeddings.
Let $p$ be a prime number and let $S_{p}$ be the set of primes of $K$ lying above $p$.
Let $X_{S_{p}}(K)$ be the Galois group of the maximal abelian pro-$p$-extension
of $K$ unramified outside $S_{p}$. Then by global class field theory,
we have an isomorphism
\[
X_{S_{p}}(K) \simeq \Z_{p}^{c_{K} + 1 + \delta_{K}} \times T_{S_{p}}(K),
\]
where $\delta_{K} \geq 0$ is the Leopoldt defect of $K$ at $p$ and
$T_{S_{p}}(K)$ is a finite torsion $\Z_{p}$-module. 
The field $K$ satisfies Leopoldt's conjecture at $p$ precisely
when $\delta_{K}=0$, and is said to be \emph{$p$-rational}
if both $\delta_{K}=0$ and $T_{S_{p}}(K)=0$. 
Gras \cite{MR3492629} has conjectured that each fixed $K$ is $p$-rational for all but 
finitely many $p$.
A more detailed review and an alternative characterisation of $p$-rationality 
is given in \S \ref{sec:p-rationality}.
For now, let us note that $p$-rational fields have found applications in several
areas of number theory, as illustrated for example by work of
Goren \cite{MR1826497}, Greenberg \cite{MR3512524}, Hajir--Maire \cite{MR3933912}, David--Pries \cite{MR4083092}, Hajir--Maire--Ramakrishna \cite{MR4308183}, and
Bouazzaoui--Lim \cite{MR4973192}.

Now let $U$ be a subgroup of the unit group $\mathcal{O}_{K}^{\times}$. 
Then $U$ is said to be \emph{$p$-saturated in $\mathcal{O}_{K}^{\times}$}
if, whenever $u \in \mathcal{O}_{K}^{\times}$ satisfies $u^{p} \in U$, 
we have $u \in U$; when $U$ has finite index, this is equivalent to 
$p \nmid [\mathcal{O}_{K}^{\times} : U]$.
Algorithms that determine whether $U$ is $p$-saturated in 
$\mathcal{O}_{K}^{\times}$ play a crucial role in more general algorithms that verify 
whether $U$ is equal to the full unit group $\mathcal{O}_{K}^{\times}$. 
Such algorithms can be used to unconditionally verify that a tentative unit group $U$,
obtained by a computation conditional on the generalised Riemann hypothesis (GRH) or by some other heuristic method, is in fact equal to $\mathcal{O}_{K}^{\times}$.

The key insight of the present article is that there is a connection
between the $p$-rationality of $K$
and the $p$-saturation of subgroups $U$ of $\mathcal{O}_{K}^{\times}$. 
This insight relies on a characterisation of $p$-rationality due to 
Nguyen Quang Do \cite{do2025greenbergsgeneralizedconjecturefamilies}
that is closely related to a similar result of Movahhedi \cite{MR1124802}.
This characterisation says that $K$ is $p$-rational if and only if three conditions are
satisfied (see Theorem~\ref{thm:nqd-char-p-rat} for a precise statement). 
The first of these conditions says that the compositum 
of all $\Z_{p}$-extensions of $K$ contains the $p$-Hilbert class field of $K$,
and thus is automatically satisfied if $p$ does not divide the class number $h_{K}$ of $K$.
The second condition involves $p$-th roots of unity in $K$ and its 
$p$-adic completions, and is automatically satisfied if $p$ is odd and unramified in $K/\Q$.
The third condition is that every global unit that is a $p$-th power in the unit group
of every completion of $K$ above $p$ is already a $p$-th power in $\mathcal{O}_{K}^{\times}$.
It is this third condition that provides the aforementioned connection,
which in turn leads to new algorithms for determining whether $K$ is $p$-rational
and for determining whether a subgroup $U$ of $\mathcal{O}_{K}^{\times}$
is $p$-saturated. When $p$ is odd and unramified in $K/\Q$, 
the third condition can be expressed in terms of the 
so-called Schirokauer map \cite{MR1253502} 
(see Definition~\ref{def:Schirokauer}), which is explicit 
and straightforward to compute, and thus
allows for a significant speed-up of these algorithms.

We first discuss previous algorithms for determining the $p$-rationality
of $K$.
Pitoun and Varescon \cite{MR3266966} have given an algorithm that
checks whether $K$ satisfies 
Leopoldt's conjecture at $p$ and then computes $T_{S_{p}}(K)$ as defined above. 
In particular, this can be used
to verify whether $K$ is $p$-rational.
Gras \cite{MR4491965} has given an algorithm to test $K$ for $p$-rationality 
along with an implementation in PARI/GP \cite{PARI2}.
Both of these approaches require computation of certain
ray class groups $\Cl_{K}(p^{k})$, or at least their Sylow-$p$ subgroups, and 
hence in particular 
require information about the Sylow-$p$ subgroup of the class group $\Cl_{K}$.
Barbulescu and Ray \cite{MR4158582} have proposed improvements to 
the method of \cite{MR3266966} for certain families of number fields.

In this article, we give two new algorithms for determining the $p$-rationality of $K$.
The first applies to real cyclotomic fields $K=\Q(\zeta_{n})^{+}=\Q(\zeta_{n}+\zeta_{n}^{-1})$.
Let $E_{n}^{+}$, $C_{n}^{+}$ and $h_{n}^{+}$ denote the unit group,
the subgroup of cyclotomic units, and the class number of $\Q(\zeta_{n})^{+}$, respectively.
Sinnott \cite{MR485778} showed that the index $[E_{n}^{+} : C_{n}^{+}]$ is equal to 
$h_{n}^{+}$ up to an explicit power of $2$.
Combining this result with the above characterisation of $p$-rationality for 
arbitrary number fields,
we give an explicit criterion for $p$-rationality of $\Q(\zeta_{n})^{+}$
that only requires direct knowledge of $C_{n}^{+}$, and not of $h_{n}^{+}$ or $E_{n}^{+}$. 
This leads to an algorithm that avoids computing either $h_{n}^{+}$ or $E_{n}^{+}$, and 
thus vastly outperforms the aforementioned methods that require the computation of certain ray class groups. 
We then use this algorithm to determine the $p$-rationality of $\Q(\zeta_{n})^{+}$ for all $5 \leq n \leq 1000$ with $n \not\equiv 2 \bmod 4$ and all $p < 10^{7}$.
We remark that if $\Q(\zeta_{n})^{+}$ is $p$-rational then $p \nmid h_{n}^{+}$, 
and so in many cases the algorithm can be used to show that $p \nmid h_{n}^{+}$.
In principle, the algorithm also extends to arbitrary real abelian fields, although certain prime numbers must be excluded and the required cyclotomic units may be difficult to compute in practice.

Now let $K$ be an arbitrary number field. We define the notion 
of \emph{quasi-$p$-rationality} for $K$ and show that it is weaker
than the notion of $p$-rationality for $K$ but stronger than 
Leopoldt's conjecture for $K$ at $p$. In fact, $K$
is quasi-$p$-rational if and only if it satisfies the second and third conditions of the characterisation of $p$-rationality discussed above. 
These two conditions do not involve the $p$-Hilbert class field of $K$ or the $p$-part of $h_{K}$.
This enables us to give an algorithm to determine the quasi-$p$-rationality of $K$
that does not require any computation of (ray) class groups.
If $K$ is quasi-$p$-rational, then to determine whether $K$ is in fact $p$-rational,
it usually suffices to determine whether $p$ divides $h_{K}$, though in some edge cases
it is necessary to revert to the aforementioned ray class group algorithm of Gras
\cite{MR4491965}. 
Of course, if $K$ is not quasi-$p$-rational then it is not $p$-rational,
and if it is quasi-$p$-rational then it satisfies Leopoldt's conjecture
at $p$. 

We now discuss algorithms for determining whether a subgroup $U$ of $\mathcal{O}_{K}^{\times}$ is $p$-saturated. For simplicity, we assume here that 
the index of $U$ in $\mathcal{O}_{K}^{\times}$ is finite. 
An impractical approach is to iterate over all non-trivial elements of $U/U^{p}$,
lift them to elements of $U$, and attempt to extract $p$-th roots.
Less naive approaches are based on the following observation.
Assume that $U$ contains the $p$-torsion subgroup of $\mathcal{O}_{K}^{\times}$.
Then $U^{p} \leq U \cap \mathcal{O}_{K}^{\times p}$ with equality if and only 
if $U$ is $p$-saturated in $\mathcal{O}_{K}^{\times}$ (see Proposition~\ref{prop:p-sat}).
Let $X$ be an $\F_{p}$-vector space and let 
$\psi : \mathcal{O}_{K}^{\times} \rightarrow X$ be a group homomorphism. 
Let $\overline{\psi\vert_{U}} : U/U^{p} \rightarrow X$ 
be the map induced by the restriction of $\psi$ to $U$.
Then the image of $U \cap \mathcal{O}_{K}^{\times p}$ under the canonical projection 
$U \rightarrow U/U^{p}$ must lie in $\ker(\overline{\psi\vert_{U}})$.
Therefore if $\ker(\overline{\psi\vert_{U}})$ 
is trivial then $U$ must be $p$-saturated in $\mathcal{O}_{K}^{\times}$.
Based on work of Pohst and Zassenhaus~\cite{MR498486, MR1243639, MR1483321, MR1273458}
and Arenz~\cite{MR1151862},
the idea of taking $\psi$ to be
the canonical map from $\mathcal{O}_{K}^{\times}$ to
a finite direct product of unit groups of residue fields modulo $p$-th powers
was introduced by Biasse and Fieker \cite{MR3207410}, and generalised and analysed by
Biasse, Fieker, Hofmann and Page \cite{MR4440537}. 
Note that the same idea was already used by Adleman~\cite{10.1145/103418.103432} 
in the case $p = 2$, in the context of the number field sieve.
Of course, if the kernel is non-trivial then 
increasing the number of residue fields may make it smaller, and an important 
question is how many residue fields should appear in the direct product.
One disadvantage of this approach is that computing the kernel
requires solving discrete logarithm problems in unit groups
of finite fields modulo $p$-th powers.
Bernard, Fouque and Lesavourey \cite{zbMATH07854947} 
showed that taking $\psi$ to be the aforementioned Schirokauer map
tends to be highly effective in practice, and that in general one can combine the two approaches. In this article, we give a modified version of the approach of 
\cite{zbMATH07854947} that also incorporates some elements of that of \cite{MR4440537},
as well as some new features. 
A key observation is the following: if $K$ is quasi-$p$-rational,
and $U$ is $p$-saturated in $\mathcal{O}_{K}^{\times}$,
then our new algorithm terminates at an early stage; when 
$p$ is odd and unramified in $K/\Q$,
it does so using only the Schirokauer map and thus requires no discrete logarithm computations.
Under the assumption of the aforementioned conjecture of Gras, 
$K$ is quasi-$p$-rational for all but finitely many primes $p$.
Together with the previous observation, 
this helps explain why the algorithm is so effective in practice. 

We apply the algorithm of the previous paragraph to the problem of unconditional verification of tentative unit groups as described above. 
We give a runtime analysis as well as concrete examples with timings that show that this new algorithm substantially outperforms existing unconditional algorithms in practice. 

One reason for the recent interest in the notion of 
$p$-rationality is Greenberg's article \cite{MR3512524},
in which he described an approach to constructing certain continuous representations 
$\Gal(\overline{\Q}/\Q) \longrightarrow \GL_{n}(\Z_{p})$
with open image and further prescribed properties
from totally complex $p$-rational number fields.
He also remarked that an obstacle to applying some of his results is
the practical difficulty of verifying that a number field of large degree is $p$-rational.
We address this obstacle by using our algorithms to establish $p$-rationality for certain explicit number fields, thereby making the corresponding constructions unconditional.

The algorithms in this article have been implemented in \textsc{Hecke}~\cite{MR3703682}; full instructions, examples, and datasets are 
available at \url{github.com/thofma/p-rationality/}.

\subsection*{Organisation}
This article is organised as follows. 
Roughly speaking, the first half develops the theory, and the second half 
presents the resulting algorithms and their applications.
The theory of the $p$-saturation of finitely generated abelian groups is reviewed
in \S \ref{sec:p-sat-fg-ab-groups}.
Section~\ref{sec:p-rationality} reviews basic facts and known results regarding 
$p$-rational fields. 
Section~\ref{sec:real-cyclo} gives explicit criteria for the $p$-rationality
of real cyclotomic fields, and \S \ref{sec:real-abelian-number-fields}
considers the case of arbitrary real abelian fields.
Section~\ref{sec:Schirokauer-map} introduces the Schirokauer map,
and \S \ref{sec:quasi-p-rationality} introduces the notion of
quasi-$p$-rational fields. 
Next, \S \ref{sec:p-rat-p-sat} studies the connection between (quasi)-$p$-rationality
and $p$-saturation of units. Section~\ref{sec:maps-on-S-units-p-rat} considers
the connection between $p$-rationality and $S$-units.
The second half begins with \S \ref{sec:coventions-comp-runtime}, which 
introduces our conventions on complexity and runtime. 
Section~\ref{sec:comp-maps} considers the computation of certain maps, 
including the Schirokauer map. 
Section~\ref{sec:alg-real-cyclo} presents an algorithm for determining the
$p$-rationality of real cyclotomic fields $\Q(\zeta_{n})^{+}$, and gives an 
overview of computational results.
Next, \S \ref{sec:algs-p-sat-units} gives an algorithm for the $p$-saturation of units, and in \S \ref{sec:unconditional-unit-group-comp} this is applied 
to the unconditional computation of unit groups.
Section~\ref{sec:quasi-p-rat-algs} gives algorithms 
for the determination of the (quasi)-$p$-rationality of arbitrary number fields. 
Finally, \S \ref{sec:Galois-reps} constructs certain Galois representations with open image by combining results of Greenberg with an application of the algorithms
of the previous section.

\subsection*{Acknowledgements}
The authors are grateful to Jonathan Bober and Andrew Booker for helpful discussions,
Tim Dokchitser for providing \textsc{Magma} \cite{MR1484478} code to compute Brauer relations, 
Aurel Page for helpful discussions on unconditional effective versions of the
Grunwald--Wang theorem, and Georges Gras and Christian Wuthrich for helpful comments on drafts of this manuscript.
For the purpose of open access, the authors have applied a Creative Commons Attribution (CC BY) license to any author accepted manuscript version arising.

\subsection*{Use of AI in this article}
After an initial draft was written, ChatGPT-5.6 Sol was used to check and polish this article. 
All the new mathematical ideas are due to the authors.

\subsection*{Funding}
Hofmann gratefully acknowledges support from the Deutsche Forschungsgemeinschaft --
Project-ID 286237555 -- TRR 195 and Project-ID 539387714.

\section{$p$-saturation of finitely generated abelian groups}\label{sec:p-sat-fg-ab-groups}
Let $A$ be a finitely generated abelian group, written multiplicatively.
The torsion-free rank of $A$ is $\rank_{\Z}(A) = \dim_\Q(\Q \otimes_{\Z} A)$. 
For a prime number $p$, the $p$-rank of $A$ is 
$\rank_{p}(A) = \dim_{\F_p}(A/A^{p})$ and the $p$-torsion subgroup of $A$ 
is $A[p]=\{ a \in A : a^{p}=1 \}$. Then
\begin{equation}\label{eq:p-rank-vs-Z-rank}
\rank_{p}(A) = \rank_{\Z}(A) + \dim_{\F_{p}}(A[p]).    
\end{equation}

A subgroup $B \leq A$ is called $p$-saturated in $A$ if, whenever $a \in A$ and $a^{p} \in B$, 
we have $a \in B$. 
This is equivalent to $A/B$ having no element of order $p$.
If the index $[A:B]$ is finite, then $B$ is $p$-saturated in $A$
if and only if $p$ does not divide $[A:B]$.

The following results are well known to experts; see~\cite[Lemma~4.3]{MR4440537} for similar results in the context of multiplicative groups in number fields.

\begin{prop}\label{prop:p-sat}
Let $B$ be a subgroup of a finitely generated abelian group $A$.  
Let $p$ be a prime number and assume that $A[p] \leq B$.
Then $B$ is $p$-saturated in $A$ if and only if $B \cap A^{p} = B^{p}$.
\end{prop}

\begin{proof}
Suppose that $B \cap A^{p} = B^{p}$. 
Let $a \in A$ with $a^{p} \in B$. Thus $a^{p} \in B \cap A^{p} = B^{p}$.
Hence there exists $b \in B$ such that $a^{p}=b^{p}$ and so $(a/b)^{p}=1$.
Hence $a/b \in A[p] \leq B$.
Thus $a \in B$. 
Therefore $B$ is $p$-saturated in $A$. 
Suppose conversely that $B$ is $p$-saturated in $A$.
Let $b\in B \cap A^{p}$. 
Then $b=a^{p}$ for some $a \in A$. Thus $a^{p} \in B$ and so $a \in B$. 
Therefore  $B \cap A^{p} \leq B^{p}$. The reverse inclusion is clear. 
\end{proof}

\begin{corollary}\label{cor:p-sat}
Let $B$ be a subgroup of a finitely generated abelian group $A$.  
Let $p$ be a prime number and assume that $A[p] \leq B$.
Let $X$ be an $\F_p$-vector space and let $\psi \colon A \to X$ be a group homomorphism. 
Let $\psi \vert_{B}$ be the restriction of $\psi$ to $B$.
\renewcommand{\labelenumii}{(\roman{enumii})}
\begin{enumerate}
    \item 
    We have $B^{p} \leq \ker(\psi \vert_{B})$, and hence $\psi \vert_{B}$
    induces an $\F_{p}$-linear map $\overline{\psi \vert_{B}} \colon B/B^{p} \to X$.
    \item The following are equivalent:
    \begin{enumerate}
        \item $B^{p} = \ker(\psi \vert_{B})$,
        \item $\dim_{\F_p}(\psi(B)) = \rank_{p}(B)$,
        \item $\dim_{\F_p}(\overline{\psi \vert_{B}}(B/B^{p})) = \rank_{p}(B)$,
        \item $\overline{\psi \vert_{B}}$ is injective.
    \end{enumerate}
    \item If the equivalent conditions of \textup{(b)} 
    hold then $B$ is $p$-saturated in $A$.
    \item If $\dim_{\F_p}(\psi(B)) = \rank_{p}(A)$ then $B$ is $p$-saturated and of finite index in $A$.
    \item If $b \in B$ and $b = a^{p}$ for some $a \in A \setminus B$, 
    then $bB^{p}$ is a non-trivial element of $\ker(\overline{\psi \vert_{B}})$, and 
    $[\langle B, a \rangle : B]=p$.
\end{enumerate}    
\end{corollary}

\begin{proof}
(a) This follows from the fact that $X$ is a vector space over $\F_{p}$.

(b) The rank-nullity theorem gives 
$\dim_{\F_{p}}(B/\ker(\psi|_B)) = \dim_{\F_p}(\psi(B))$
and together with part (a) this implies that 
(i) $\Longleftrightarrow$ (ii).
The definition of $\overline{\psi \vert_{B}}$ gives (ii) 
$\Longleftrightarrow$ (iii) and the rank-nullity theorem gives (iii) 
$\Longleftrightarrow$ (iv).

(c) We have $B^{p} \leq B \cap A^{p} \leq B \cap \ker(\psi) = \ker(\psi\vert_B)$.
Suppose that the equivalent assertions of (b) hold. In particular,
$B^{p} = \ker(\psi \vert_{B})$ and so the above containments become equalities.
Thus $B^{p} = B \cap A^{p}$
and so the desired result follows from Proposition~\ref{prop:p-sat}.

(d) 
We have 
$\rank_{p}(B) \geq \dim_{\F_p}(\psi(B)) 
= \rank_{p}(A) \geq \rank_{p}(B)$, and so each inequality is in fact an equality.
Thus $B$ is $p$-saturated in $A$ by part (c).
Moreover, $A[p]=B[p]$ since $A[p] \leq B \leq A$.
Therefore by \eqref{eq:p-rank-vs-Z-rank} we have 
$\rank_{\Z}(A) =\rank_{\Z}(B)$, and so $B$ has finite index in $A$. 

(e) We have that $\psi(b)=\psi(a^{p})$ and so $b \in \ker(\psi)$ since $X$
is a vector space over $\F_{p}$. 
Thus $bB^{p} \in \ker(\overline{\psi \vert_{B}})$.
Moreover, $bB^{p}$ is non-trivial: if $b=c^{p}$ for some $c \in B$,
then $(ac^{-1})^{p}=1$, whence $ac^{-1} \in A[p] \leq B$, contradicting $a \notin B$.
The quotient $\langle B, a \rangle / B$ is generated by the coset $aB$,
which is non-trivial since $a \notin B$.
But $(aB)^{p}=a^{p}B=bB=B$ and so $aB$ has order exactly $p$, which 
gives the desired result. 
\end{proof}

\begin{remark}
Corollary~\ref{cor:p-sat}(c) is useful for determining whether a subgroup 
$B$ is $p$-saturated in $A$. 
By contrast, Corollary~\ref{cor:p-sat}(e)  is useful if $B$ is not
$p$-saturated in $A$ and one wishes to compute a subgroup $C$ such that $B \leq C \leq A$
and $C$ is $p$-saturated in $A$; the process of enlarging $B$ to such a subgroup $C$
is often called $p$-saturation, or saturating $B$ at $p$.
A key point is that non-trivial elements of $\ker(\overline{\psi \vert_{B}})$ 
identify cosets whose representatives $b \in B$ are candidates for having $p$-th 
roots in $A$; thus it is advantageous to 
choose $\psi$ so that $\dim_{\F_{p}}(\ker(\overline{\psi \vert_{B}}))$ is as small as
possible.
\end{remark}

\section{Review of $p$-rational fields}\label{sec:p-rationality}

In this section, we review basic facts and known 
results regarding $p$-rational fields.

\subsection{$p$-rational fields}\label{subsec:p-rational-fields}
The notion of $p$-rationality was first defined in the PhD thesis
of Movahhedi \cite{MovahhediPhD}, where it was used to find non-abelian extensions of 
$\Q$ satisfying Leopoldt's conjecture at $p$. 
The extensive literature on $p$-rationality also includes 
work of Jaulent--Nguyen Quang Do \cite{MR1265910}, Movahhedi--Nguyen Quang Do \cite{MR1042770}, Movahhedi \cite{MR1124802} and Gras--Jaulent \cite{MR1017575}.
A good general reference is the book of Gras
\cite[Chapter~VI, \S 3]{MR1941965}.
We only give a brief overview here.

Let $K$ be a number field. 
Let $c_{K}$ denote the number of pairs of complex embeddings of $K$.
Let $p$ be a prime number and let $S_{p}=S_{p}(K)$ be the set of primes of $K$ 
lying above $p$. 
Let $K_{S_{p}}(p)$ be the maximal pro-$p$-extension of $K$ unramified outside
$S_{p}$ and let $G_{S_{p}}(K)=\Gal(K_{S_{p}}(p)/K)$.
Let $K_{S_{p}}^{\mathrm{ab}}(p)$ be the 
maximal abelian pro-$p$-extension of $K$ unramified outside $S_{p}$, let 
$X_{S_{p}}(K)=\Gal(K_{S_{p}}^{\mathrm{ab}}(p)/K)$, and let 
$T_{S_{p}}(K)$ denote the $\Z_{p}$-torsion subgroup of $X_{S_{p}}(K)$.
The following result is due to Movahhedi \cite[Proposition~1]{MR1124802} (see also \cite[\S 3]{MR3512524}).

\begin{prop}\label{prop:def-char-p-rational}
For a number field $K$, the following conditions are equivalent:
\begin{enumerate}
    \item $G_{S_{p}}(K)$ is a free pro-$p$ group on $c_{K}+1$ generators.
    \item $X_{S_{p}}(K)$ is a free $\Z_{p}$-module of rank $c_{K}+1$.
    \item $K$ satisfies Leopoldt's conjecture at $p$ and $T_{S_{p}}(K)=0$.
\end{enumerate}
\end{prop}

\begin{definition}\label{def:p-rat}
If a number field $K$ satisfies  
the equivalent conditions of Proposition~\ref{prop:def-char-p-rational}
then it is said to be \textit{$p$-rational}.
\end{definition}

\begin{example}
By the Kronecker--Weber theorem, $\Q$ is $p$-rational for all $p$.    
\end{example}

The following useful result is well known (see \cite[IV, 3.4.5]{MR1941965}, for example).

\begin{lemma}\label{lem:p-rat-subfields}
If a number field is $p$-rational, then so are all of its subfields.    
\end{lemma}

Gras \cite[Conjecture~8.11]{MR3492629} has proposed the following conjecture.

\begin{conj}\label{conj:Gras}
A number field is $p$-rational for all but finitely many primes $p$.
\end{conj}

\begin{remark}\label{rmk:Dirichlet-density-0}
Even if Gras's Conjecture~\ref{conj:Gras} is false, it seems plausible that
a number field should be $p$-rational for `most' primes $p$. For example,
Greenberg \cite[\S 4.2]{MR3512524} writes
\begin{quote}
\textit{\ldots it seems reasonable to expect that any fixed number field $K$ should
be $p$-rational for all primes $p$ except for a set of Dirichlet density $0$.}  
\end{quote}
\end{remark}

\begin{remark}\label{rmk:going-up}
Suppose that $K$ is $p$-rational and that $L/K$ is a finite Galois $p$-extension. 
Then
$L$ is $p$-rational if and only if the set of primes ramified in $L/K$
satisfies certain conditions; see
\cite[Th\'eor\`eme $(\beta)$, p.\ 344]{MR1017575}, \cite[Th\'eor\`eme~3]{MR1124802} or \cite[Th\'eor\`eme~3.5]{MR1265910}.
A special case of this ``going-up'' result is given in Corollary~\ref{cor:p-rat-going-up} below.
\end{remark}

\subsection{Necessary and sufficient conditions for $p$-rationality}\label{subsec:nec-suff-p-rat}
We first set up some further notation.
For a finite prime $\mathfrak{p}$ of $K$, let $K_{\mathfrak{p}}$ denote the completion of 
$K$ at $\mathfrak{p}$ and 
let $U_{\mathfrak{p}}$ denote the unit group of the valuation ring
$\mathcal{O}_{K_{\mathfrak{p}}}$.
For a field $\mathbb{F}$, let $\mu_{p}(\mathbb{F})$ and $\mu_{p^\infty}(\mathbb{F})$
denote the groups of $p$-th and $p$-power roots of unity contained in $\mathbb{F}^{\times}$, respectively.
Let $\overline{\mathbb{F}}$ denote an algebraic closure of $\mathbb{F}$ and let 
$\mu_{p} = \mu_{p}(\overline{\mathbb{F}})$.
Let $H_{p}(K)$ denote the $p$-Hilbert class field of $K$.
Let $h_{K}$ and $d_{K}$ denote the class number and absolute discriminant of $K$, 
respectively. 

\begin{definition}\label{def:various-maps}
We define the following diagonal maps:
\begin{enumerate}
\item 
$\lambda_{K,p} : \mu_{p}(K) \longrightarrow \prod_{\mathfrak{p} \in S_{p}(K)} 
\mu_{p}(K_{\mathfrak{p}})$,
\item 
$\eta_{K,p} : \mathcal{O}_{K}^{\times} \longrightarrow 
\textstyle{\prod_{\mathfrak{p} \in S_{p}(K)}} U_{\mathfrak{p}}/ U_{\mathfrak{p}}^{p}$, and
\item
$\overline{\eta}_{K,p} : \mathcal{O}_{K}^{\times} / \mathcal{O}_{K}^{\times p } \longrightarrow 
\textstyle{\prod_{\mathfrak{p} \in S_{p}(K)}} U_{\mathfrak{p}}/ U_{\mathfrak{p}}^{p}$.
\end{enumerate}
\end{definition}

We recall the following necessary and sufficient condition for $p$-rationality due to 
Nguyen Quang Do \cite[Proposition~6.1]{do2025greenbergsgeneralizedconjecturefamilies};
see also \cite[Proposition~2.3]{MR4192837}.
This is similar to a result of Movahhedi \cite[Proposition~2]{MR1124802}.

\begin{theorem}\label{thm:nqd-char-p-rat}
A number field $K$ is $p$-rational precisely when all three of the following conditions are
satisfied:
\begin{enumerate}
    \item The compositum of all $\Z_{p}$-extensions of $K$ contains $H_{p}(K)$.
    \item The diagonal map $\lambda_{K,p} : \mu_{p}(K) \longrightarrow \prod_{\mathfrak{p} \in S_{p}(K)} 
    \mu_{p}(K_{\mathfrak{p}})$ is an isomorphism.
    \item The diagonal map 
    $\overline{\eta}_{K,p} : \mathcal{O}_{K}^{\times}/\mathcal{O}_{K}^{\times p} \longrightarrow 
    \prod_{\mathfrak{p} \in S_{p}(K)} U_{\mathfrak{p}}/ U_{\mathfrak{p}}^{p}$ is injective.
\end{enumerate}
\end{theorem}

\begin{remark}
Clearly, condition (a) holds if $p \nmid h_{K}$.
By the rank-nullity theorem, condition (c) is equivalent to the assertion
that 
$\dim_{\F_{p}}(\eta_{K,p}(\mathcal{O}_{K}^{\times}))
=\rank_{p}(\mathcal{O}_{K}^{\times})$.
\end{remark}

\begin{remark}\label{rmk:roots-of-unity-condition}
Condition (b) is equivalent to the following pair of assertions:
\[
\begin{cases}
\lvert S_{p}(K) \rvert =1, & \text{ if } \mu_{p} \subset K,  \\
\mu_{p} \not \subset K_{\mathfrak{p}} \text{ for every } \mathfrak{p} \in S_{p}(K),
& \text{ if } \mu_{p} \not \subset K.
\end{cases}
\]
For $\mathfrak{p} \in S_{p}(K)$, let $e_{\mathfrak{p}}$ denote the
ramification index of $\mathfrak{p}$ in $K/\Q$.
If $(p-1) \nmid e_{\mathfrak{p}}$ then $\mu_{p} \not \subset K_{\mathfrak{p}}$ 
(and so $\mu_{p} \not \subset K$) because 
$\Q_{p}(\zeta_{p})/\Q_{p}$ is totally ramified and of degree $p-1$.
In particular, if $p>[K:\Q]+1$ or $p \nmid 2d_{K}$ then condition (b) holds.

If $\mu_{p} \not \subset K$ then condition (b) is equivalent to the following:
\begin{itemize}
    \item no prime $\mathfrak{p} \in S_{p}(K)$ splits completely in $K(\zeta_{p})/K$.
\end{itemize}
In fact, by the above discussion, it is only necessary to check this
condition for $\mathfrak{p} \in S_{p}(K)$ such that $(p-1) \mid e_{\mathfrak{p}}$.
This observation is useful for computational purposes.
\end{remark}

Let $\Cl_{K}$ denote the class group of $K$.
For $k \geq 0$, let $\Cl_{K}(p^{k})$ denote the ray class group
of $K$ of modulus $(p^{k})$.
We recall the following criterion for $p$-rationality due to Gras \cite[\S 2.1]{MR4491965}.

\begin{theorem}\label{thm:Gras-ray-p-rat}
Let $K$ be a number field and let $p$ be a prime number. 
If $p=2$, let $k=3$, and if $p>2$, let $k=2$.
Then $\rank_{p}(\Cl_{K}(p^{k})) \geq c_{K}+1$
and this inequality is an equality if and only if $K$ is $p$-rational.
\end{theorem}

\subsection{$p$-rationality for totally real fields}\label{subsec:p-rat-tot-real}
The following well-known result is an easy consequence of 
Proposition~\ref{prop:def-char-p-rational} and Definition~\ref{def:p-rat}
(see also \cite[\S 3]{MR3512524}).

\begin{prop}\label{prop:tot-real-p-rat-equivs}
Let $F$ be a totally real number field and let $F^{\mathrm{cyc},p}$
denote the cyclotomic $\Z_{p}$-extension of $F$.
Then the following conditions are equivalent:
\begin{enumerate}
    \item $F$ is $p$-rational. 
    \item $G_{S_{p}}(F) \simeq X_{S_{p}}(F) \simeq \Z_{p}$.
    \item $F_{S_{p}}(p)  = F_{S_{p}}^{\mathrm{ab}}(p) = F^{\mathrm{cyc},p}$.
\end{enumerate}
\end{prop}

\begin{corollary}\label{cor:p-rat-going-up}
Let $F$ be a totally real number field and 
let $K/F$ be a finite Galois $p$-extension that is unramified outside
$S_{p}(F)$. Then $F$ is $p$-rational if and only if $K$ is $p$-rational.
\end{corollary}

\begin{proof}
By hypothesis, $K$ is contained in $F_{S_{p}}(p)$ and so $K_{S_{p}}(p) = F_{S_{p}}(p)$.
If $F$ is $p$-rational then $F_{S_{p}}(p) = F^{\mathrm{cyc},p}$ by 
Proposition~\ref{prop:tot-real-p-rat-equivs}, and therefore 
$K_{S_{p}}(p) = K^{\mathrm{cyc},p}$; thus $K$ is $p$-rational, again by Proposition~\ref{prop:tot-real-p-rat-equivs}.
The converse is Lemma~\ref{lem:p-rat-subfields}.
\end{proof}

Let $v_{p}$ denote the $p$-adic valuation function normalised by $v_{p}(p)=1$.

\begin{corollary}\label{cor:tot-real-p-rat-class-number}
Let $F$ be a totally real $p$-rational number field.
Let $m \geq 0$ be the largest integer such that $F_{m}/F$ is unramified,
where $F_{m}$ denotes the $m$-th layer of the cyclotomic $\Z_{p}$-extension $F^{\mathrm{cyc},p}$ of $F$. Then $v_{p}(h_{F})=m$.    
\end{corollary}

\begin{proof}
We have $H_{p}(F) \subseteq F_{S_{p}}(p)$ by definition and   
$F_{S_{p}}(p) = F^{\mathrm{cyc},p}$ by Proposition~\ref{prop:tot-real-p-rat-equivs}.
Therefore $H_{p}(F) \subseteq F^{\mathrm{cyc},p}$, 
from which the desired result follows easily.
\end{proof}

\begin{corollary}\label{cor:tot-real-p-rat-tame}
Let $F$ be a totally real $p$-rational number field. 
If each prime in $S_{p}(F)$ is at most tamely ramified in $F/\Q$ then 
$p \nmid h_{F}$.
\end{corollary}

\begin{proof}
Suppose for a contradiction that $p \mid h_{F}$.
Then $m > 0$ and so $H_{p}(F)$ must contain the first layer of the cyclotomic $\Z_{p}$-extension of $\Q$.
Hence some prime of $H_{p}(F)$ above $p$ must be wildly ramified in $H_{p}(F)/\Q$.
But $H_{p}(F)/F$ is unramified and so some prime of $F$ above $p$ must be wildly 
ramified in $F/\Q$, contradicting the tameness hypotheses.
\end{proof}

\subsection{$p$-rationality and $p$-adic regulators}
For a totally real number field $F$ and a prime number $p$, Leopoldt's conjecture
for $F$ at $p$ is equivalent to the non-vanishing of 
the $p$-adic regulator $R_{F,p}$ of $F$ (see \cite[\S 5.5]{MR1421575} or \cite[(10.3.5)]{MR2392026}).

We now recall the following result of Coates \cite[Appendix, Lemma~8]{MR460282}.

\begin{theorem}\label{thm:Coates}
Let $F$ be a totally real number field and let $p$ be an odd prime number. 
If $R_{F,p} \neq 0$ then 
\[
v_{p}\left( | T_{S_{p}(F)} | \right) 
= v_{p}\left( 
|\mu_{p^{\infty}}(F(\zeta_{p}))| \frac{h_{F}R_{F,p}}{\sqrt{d_{F}}} 
\prod_{\mathfrak{p} \in S_{p}(F)} \left( 1 - \Norm_{F/\Q}(\mathfrak{p})^{-1} \right)
\right). 
\]
\end{theorem}

This leads to the following condition for $p$-rationality of totally real number fields.

\begin{corollary}\label{cor:p-rat-p-reg}
Let $F$ be a totally real number field and let $p$ be a prime number such that $p \nmid 2d_{F}$.
Then $F$ is $p$-rational if and only if $p \nmid h_{F}$ and $v_{p}(R_{F,p})=[F:\Q]-1$.
\end{corollary}

\begin{proof}
Since $p \nmid d_{F}$, we have 
$v_{p}(\sqrt{d_{F}})=0$, $|\mu_{p^{\infty}}(F(\zeta_{p}))|=p$
and
\[
\textstyle{\prod_{\mathfrak{p} \in S_{p}(F)} \left( 1 - \Norm_{F/\Q}(\mathfrak{p})^{-1} \right)} = p^{-[F:\Q]}.
\]
Hence if $R_{F,p}\ne0$ then Theorem~\ref{thm:Coates} gives
\[
v_{p}(|T_{S_p}(F)|)
=
v_{p}(h_{F}) + v_{p}(R_{F,p})-([F:\Q]-1),
\]
and the definition of $R_{F,p}$ together with 
the assumption $p\nmid d_{F}$ give $v_{p}(R_{F,p})\geq [F:\Q]-1$.

We now use the characterisation of $p$-rationality given by 
Proposition~\ref{prop:def-char-p-rational}(c).
If $F$ is $p$-rational, then $T_{S_p}(F)=0$ and Leopoldt’s conjecture gives $R_{F,p}\ne0$; 
the preceding inequality therefore forces $p \nmid h_{F}$ and 
$v_{p}(R_{F,p})=[F:\Q]-1$. Conversely, these two conditions imply $R_{F,p}\ne0$ 
and $T_{S_p}(F)=0$, and so $F$ is $p$-rational.
\end{proof}

\subsection{$p$-rationality for CM-fields}
The following result of Benmerieme and Movahhedi \cite[Proposition~2.14]{MR4192837}
relates the $p$-rationality of a CM-field $K$ to that of its maximal totally real 
subfield $K^{+}$.

\begin{prop}\label{prop:CM-tot-real-p-rat}
Let $K$ be a CM-field and let $p$ be an odd prime number. 
Suppose that $p \nmid h_{K}$ and that 
$\mu_{p} \not \subset K_{\mathfrak{p}}$ for each $\mathfrak{p} \in S_{p}(K)$.
Then $K$ is $p$-rational if and only if $K^{+}$ is $p$-rational.
\end{prop}

Recall that if $K$ is a CM-field, then $h_{K^{+}}$ divides $h_{K}$ and the 
quotient $h_{K}^{-} = h_{K} / h_{K^{+}}$ is called the relative class number of $K$
(see \cite[Theorem~4.10]{MR1421575}).

\begin{corollary}\label{cor:CM-tot-real-p-rat}
Let $K$ be a CM-field and let $p$ be a prime number such that $p \nmid 2d_{K}h_{K}^{-}$.
Then $K$ is $p$-rational if and only if $K^{+}$ is $p$-rational.
\end{corollary}

\begin{proof}
If $K$ is $p$-rational then $K^{+}$ is $p$-rational by Lemma~\ref{lem:p-rat-subfields}.
Suppose that $K^{+}$ is $p$-rational.
Then $p \nmid h_{K^{+}}$ by Corollary~\ref{cor:tot-real-p-rat-tame} since 
$p \nmid d_{K}$. Hence $p \nmid h_{K} =  h_{K^{+}} h_{K}^{-}$.
Moreover, $\mu_{p} \not \subset K_{\mathfrak{p}}$ 
for each $\mathfrak{p} \in S_{p}(K)$ since $p \nmid 2d_{K}$.
Therefore $K$ is $p$-rational by Proposition~\ref{prop:CM-tot-real-p-rat}.
\end{proof}

\section{$p$-rationality of real cyclotomic fields}\label{sec:real-cyclo}

\subsection{Cyclotomic units}\label{subsec:cyclotomic-units}
Here we recall some basic facts about cyclotomic units (also known as circular units)
and their relation to class numbers of real cyclotomic fields. 

\begin{definition}\label{def:cyclo-units-n}
For $n \in \Z_{\geq 1}$, let $\zeta_{n}=\exp(2\pi i / n)$, 
let $E_{n} = \Z[\zeta_{n}]^{\times}$, let 
$D_{n}$ be the subgroup of the multiplicative group $\Q(\zeta_{n})^{\times}$
generated by $\{ 1-\zeta_{n}^{a} : 1 \leq a < n \}$,
and let $C_{n} = D_{n} \cap E_{n}$.
Let $\Q(\zeta_{n})^{+}=\Q(\zeta_{n}+\zeta_{n}^{-1})$ 
be the maximal totally real subfield of $\Q(\zeta_{n})$ and
let $h_{n}^{+}$ denote its class number.
Let $E_{n}^{+} = E_{n} \cap \Q(\zeta_{n})^{+}$ and let 
$C_{n}^{+} = C_{n} \cap E_{n}^{+}$.
\end{definition}

\begin{remark}\label{rmk:n-not-2-mod-4}
For $n$ odd, we have $\Q(\zeta_{n})=\Q(\zeta_{2n})$.
Thus we can and do assume without loss of generality that $n \not \equiv 2 \bmod{4}$.
This assumption is 
important when considering the number of distinct prime factors of~$n$.
\end{remark}

The following result is due to Sinnott 
\cite[Theorem and (4.2)]{MR485778}.

\begin{theorem}\label{thm:sinnott-cyclo}
Let $n \in \Z_{\geq 3}$ with $n \not \equiv 2 \bmod 4$ and
let $g$ be the number of distinct prime factors of $n$.
Let $b=0$ if $g=1$ and $b=2^{g-2}+1-g$ if $g \geq 2$.
Then
\[
[E_{n} : C_{n}] 
=[E_{n}^{+} : C_{n}^{+}] 
= 2^{b}h_{n}^{+}.
\]
\end{theorem}

\subsection{Class numbers and $p$-rationality}\label{subsec:class-no-p-rat}
The following lemma is well known to experts, but we include the short proof for the convenience of the reader.

\begin{lemma}\label{lem:real-cyclo-p-rat-class-number}
Let $n \in \Z_{\geq 3}$ with $n \not \equiv 2 \bmod 4$.
If $\Q(\zeta_{n})^{+}$ is $p$-rational then $p \nmid h_{n}^{+}$.
\end{lemma}

\begin{proof}
Let $F=\Q(\zeta_{n})^{+}$.
If every prime of $F$ above $p$ is at most tamely ramified in $F/\Q$, the result follows immediately from Corollary~\ref{cor:tot-real-p-rat-tame}. 
We therefore assume that some prime above $p$ is wildly ramified in $F/\Q$.
Hence $p^{2} \mid n$ and so for every $k \geq 0$ we have that $F_{k} = \Q(\zeta_{np^{k}})^{+}$,
where $F_{k}$ denotes the $k$-th layer of the cyclotomic $\Z_{p}$-extension of $F$.
By considering ramification indices, we see that $F_{k}/F$ is wildly ramified at all primes above $p$
whenever $k \geq 1$. Therefore the desired result now follows from 
Corollary~\ref{cor:tot-real-p-rat-class-number}.
\end{proof}

\subsection{$p$-rationality for odd $p$}\label{subsec:real-cyclo-p-odd-rat}
We now combine the results of \S \ref{subsec:p-rat-tot-real}, 
\S \ref{subsec:cyclotomic-units}, and \S \ref{subsec:class-no-p-rat}
to give an explicit necessary and sufficient condition for the $p$-rationality of a real cyclotomic field when $p$ is odd. An important advantage of the following result
is that it only requires direct knowledge of $C_{n}^{+}$, 
not of $E_{n}^{+}$ or $h_{n}^{+}$.
Let $\varphi$ denote Euler's totient function
and recall that the relevant maps are defined in Definition~\ref{def:various-maps}.

\begin{theorem}\label{thm:p-odd-rat-cyclo}
Let $n \in \Z_{\geq 3}$ with $n \not \equiv 2 \bmod 4$ and let $p$ be an odd prime number.
Then $\Q(\zeta_{n})^{+}$ is $p$-rational if and only if
\begin{enumerate}
    \item either $p \nmid n$ or no prime of $\Q(\zeta_{n})^{+}$ above $p$ splits in $\Q(\zeta_{n})/\Q(\zeta_{n})^{+}$, and
    \item $\dim_{\F_{p}}(\eta_{\Q(\zeta_{n})^{+},p}(C_{n}^{+}))=\frac{1}{2}\varphi(n)-1$.
\end{enumerate}
\end{theorem}

\begin{proof}
Henceforth abbreviate $\eta_{\Q(\zeta_{n})^{+},p}$, $\overline{\eta}_{\Q(\zeta_{n})^{+},p}$
and $\lambda_{\Q(\zeta_{n})^{+},p}$ to 
$\eta$, $\overline{\eta}$ and $\lambda$,
respectively.
Since $p$ is odd and $\Q(\zeta_{n})^{+}$ is totally real, we have
\begin{equation}\label{eq:p-odd-unit-rank-real-cyclo}
\rank_{p}(E_{n}^{+}) 
= \rank_{\Z}(E_{n}^{+})
= [\Q(\zeta_{n})^{+}:\Q] -1 
= \textstyle{\frac{1}{2}}\varphi(n)-1.
\end{equation}
If $p \mid n$ then $\Q(\zeta_{n})^{+}(\zeta_{p})=\Q(\zeta_{n})$ and if $p \nmid n$ then
$p \nmid d_{\Q(\zeta_{n})^{+}}$. Hence by Remark~\ref{rmk:roots-of-unity-condition}, 
condition 
(a) is equivalent to the bijectivity of $\lambda$.

Suppose that $\Q(\zeta_{n})^{+}$ is $p$-rational. 
By Theorem~\ref{thm:nqd-char-p-rat} 
the map $\lambda$ is an isomorphism and so condition (a) holds
by the discussion above.
Lemma~\ref{lem:real-cyclo-p-rat-class-number} implies that $p \nmid h_{n}^{+}$.
Thus $p \nmid [E_{n}^{+} : C_{n}^{+}]$ by Theorem~\ref{thm:sinnott-cyclo} and
so $\eta(C_{n}^{+}) = \eta(E_{n}^{+})$.
Again by Theorem~\ref{thm:nqd-char-p-rat}, $\overline{\eta}$ is injective
and thus by the rank-nullity theorem, the previous sentence and \eqref{eq:p-odd-unit-rank-real-cyclo} we have 
\[
\dim_{\F_{p}}(\eta(C_{n}^{+}))
=
\dim_{\F_{p}}(\eta(E_{n}^{+}))
=
\dim_{\F_{p}}(\overline{\eta}(E_{n}^{+}/(E_{n}^{+})^{p}))
=
\rank_{p}(E_{n}^{+})
= 
\textstyle{\frac{1}{2}}\varphi(n)-1.
\]
Therefore condition (b) also holds.

Suppose conversely that (a) and (b) hold. 
Then $\lambda$ is an isomorphism by (a) and the discussion above.
By (b) and \eqref{eq:p-odd-unit-rank-real-cyclo} we have
\[
\rank_{p}(E_{n}^{+})
=
\textstyle{\frac{1}{2}}\varphi(n)-1
= 
\dim_{\F_{p}}(\eta(C_{n}^{+}))
\leq 
\dim_{\F_{p}}(\eta(E_{n}^{+})) 
\leq 
\rank_{p}(E_{n}^{+}),
\]
and so these inequalities are in fact equalities. 
In particular, 
\[
\dim_{\F_{p}}(\overline{\eta}(E_{n}^{+}/(E_{n}^{+})^{p})) 
= \dim_{\F_{p}}(\eta(E_{n}^{+}))
=\rank_{p}(E_{n}^{+}),
\]
and so $\overline{\eta}$ is injective.
Moreover, 
$\dim_{\F_{p}}(\eta(C_{n}^{+}))=\rank_{p}(E_{n}^{+})$
and $E_{n}^{+}$ has no non-trivial $p$-torsion since $p$ is odd and $\Q(\zeta_{n})^{+}$
is totally real. Hence by Corollary~\ref{cor:p-sat}(d), 
$C_{n}^{+}$ is $p$-saturated and of finite index in $E_{n}^{+}$, 
or equivalently, $p \nmid [E_{n}^{+}:C_{n}^{+}]$.
Thus $p \nmid h_{n}^{+}$ by Theorem~\ref{thm:sinnott-cyclo},
and so $H_{p}(\Q(\zeta_{n})^{+})=\Q(\zeta_{n})^{+}$.
Therefore $\Q(\zeta_{n})^{+}$ is $p$-rational by Theorem~\ref{thm:nqd-char-p-rat}.
\end{proof}

\begin{remark}
By Corollary~\ref{cor:p-sat}(b), condition (b) of Theorem~\ref{thm:p-odd-rat-cyclo}
is equivalent to the injectivity of the map 
$C_{n}^{+}/(C_{n}^{+})^{p} \rightarrow \prod_{\mathfrak{p} \in S_{p}(\Q(\zeta_{n})^{+})} U_{\mathfrak{p}}/ U_{\mathfrak{p}}^{p}$ induced by 
$\eta_{\Q(\zeta_{n})^{+},p}$.
\end{remark}

\subsection{$2$-rationality}\label{subsec:real-cyclo-2-rat}
We shall need the following result to rule out certain situations.

\begin{lemma}\label{lem:2-rat}
Let $n \in \Z_{\geq 3}$ with $n \not \equiv 2 \bmod 4$.
If at least one of the following conditions holds then $\Q(\zeta_{n})^{+}$ is not $2$-rational:
\begin{enumerate}
    \item $\ell \mid n$ where $\ell$ is a prime number such that $\ell \equiv 1 \bmod{8}$,
    \item $4\ell \mid n$ where $\ell$ is a prime number such that $\ell \equiv -1 \bmod{8}$,
    \item $\ell_{1}\ell_{2} \mid n$ where $\ell_{1},\ell_{2}$ are distinct odd prime numbers such that 
    $\ell_{1} \equiv \ell_{2} \bmod{4}$,
    \item the number of distinct prime factors of $n$ is at least $3$.
\end{enumerate}
\end{lemma}

\begin{proof}
The $2$-rational real quadratic fields are precisely $\Q(\sqrt{2})$ and those of the form 
$\Q(\sqrt{\ell})$ or $\Q(\sqrt{2\ell})$ where $\ell$ is a prime number such that $\ell \equiv \pm 3 \bmod{8}$; see \cite[Example~1.3]{MR4491965} or \cite[Proposition~2.12]{MR4192837}.
In particular, fields of the form 
$\Q(\sqrt{\ell})$ for a prime number $\ell$ such that $\ell \equiv \pm 1 \bmod{8}$, or the form
$\Q(\sqrt{\ell_{1}\ell_{2}})$ for distinct odd prime numbers $\ell_{1}$ and $\ell_{2}$, are not $2$-rational.

We shall now repeatedly use the fact that for an odd prime number $\ell$, 
the (unique) quadratic subfield
of $\Q(\zeta_{\ell})$ is $\Q(\sqrt{\ell^{*}})$ where $\ell^{*}=(-1)^{(\ell-1)/2}\ell$.
In case (a), we have
$\Q(\sqrt{\ell^{*}})=\Q(\sqrt{\ell}) \subseteq \Q(\zeta_{\ell}) \subseteq \Q(\zeta_{n})$
and so $\Q(\sqrt{\ell}) \subseteq \Q(\zeta_{n})^{+}$.
In case (b), we have $\Q(\sqrt{-1}) \subseteq \Q(\zeta_{n})$ and
$\Q(\sqrt{\ell^{*}})=\Q(\sqrt{-\ell}) \subseteq \Q(\zeta_{\ell}) \subseteq \Q(\zeta_{n})$;
hence $\Q(\sqrt{\ell}) \subseteq \Q(\zeta_{n})^{+}$.
In case (c), for $i=1,2$ we have 
$\Q(\sqrt{\ell_{i}^{*}}) \subseteq \Q(\zeta_{\ell_{i}}) \subseteq \Q(\zeta_{n})$ 
and so $\Q(\sqrt{\ell_{1}\ell_{2}}) \subseteq \Q(\zeta_{n})^{+}$ since $\ell_{1}\ell_{2} = \ell_{1}^{*}\ell_{2}^{*}$. In case (d), there are two subcases. 
If $n$ is odd then there are at least three distinct odd prime divisors of $n$, two of which 
must be congruent modulo $4$, which is case (c). If $4 \mid n$ then there are at least 
two distinct odd prime divisors of $n$, say $\ell_{1}$ and $\ell_{2}$. 
Then $\Q(\sqrt{-1}) \subseteq \Q(\zeta_{n})$ and 
for $i=1,2$ we have 
$\Q(\sqrt{\ell_{i}^{*}}) \subseteq \Q(\zeta_{\ell_{i}}) \subseteq \Q(\zeta_{n})$, 
and therefore $\Q(\sqrt{\ell_{1}\ell_{2}}) \subseteq \Q(\zeta_{n})^{+}$.
In each case (a)--(d), we have shown that $\Q(\zeta_{n})^{+}$
contains a real quadratic field that is not $2$-rational; 
thus $\Q(\zeta_{n})^{+}$ itself is not $2$-rational by Lemma~\ref{lem:p-rat-subfields}.
\end{proof}

We now give an explicit necessary and sufficient condition for the $2$-rationality of 
a real cyclotomic field. It
only requires direct knowledge of $C_{n}^{+}$, 
not of $E_{n}^{+}$ or $h_{n}^{+}$.

\begin{theorem}\label{thm:2-rat-cyclo}
Let $n \in \Z_{\geq 3}$ with $n \not \equiv 2 \bmod 4$.
Then $\Q(\zeta_{n})^{+}$ is $2$-rational if and only if
\begin{enumerate}
    \item there is precisely one prime of $\Q(\zeta_{n})^{+}$ above $2$, and
    \item $\dim_{\F_{2}}(\eta_{\Q(\zeta_{n})^{+},2}(C_{n}^{+}))=\frac{1}{2}\varphi(n)$.
\end{enumerate}
\end{theorem}

\begin{proof}
Henceforth abbreviate $\eta_{\Q(\zeta_{n})^{+},2}$, $\overline{\eta}_{\Q(\zeta_{n})^{+},2}$
and $\lambda_{\Q(\zeta_{n})^{+},2}$ to $\eta$ and $\overline{\eta}$ and
$\lambda$, respectively.
Let $g$ be the number of distinct prime factors of $n$.
Since $\{ \pm 1 \} \subset E_{n}^{+}$, we have
\begin{equation}\label{eq:2-unit-rank-real-cyclo}
\rank_{2}(E_{n}^{+}) 
= \rank_{\Z}(E_{n}^{+}) + 1
= [\Q(\zeta_{n})^{+}:\Q]
= \textstyle{\frac{1}{2}}\varphi(n).
\end{equation}

Suppose that $\Q(\zeta_{n})^{+}$ is $2$-rational.
Then by Theorem~\ref{thm:nqd-char-p-rat} and 
Remark~\ref{rmk:roots-of-unity-condition}, condition (a) holds. 
Lemma~\ref{lem:real-cyclo-p-rat-class-number} implies that $2 \nmid h_{n}^{+}$
and Lemma~\ref{lem:2-rat}(d) implies that $g \leq 2$.
Thus $2 \nmid [E_{n}^{+} : C_{n}^{+}]$ by Theorem~\ref{thm:sinnott-cyclo} and
so $\eta(C_{n}^{+}) = \eta(E_{n}^{+})$.
Again by Theorem~\ref{thm:nqd-char-p-rat}, $\overline{\eta}$ is injective
and thus by the rank-nullity theorem, the previous sentence and \eqref{eq:2-unit-rank-real-cyclo} we have 
\[
\dim_{\F_{2}}(\eta(C_{n}^{+}))
=
\dim_{\F_{2}}(\eta(E_{n}^{+}))
=
\dim_{\F_{2}}(\overline{\eta}(E_{n}^{+}/(E_{n}^{+})^{2}))
=
\rank_{2}(E_{n}^{+})
= 
\textstyle{\frac{1}{2}}\varphi(n).
\]
Therefore condition (b) also holds.

Suppose conversely that (a) and (b) hold. 
Then $\lambda$ is an isomorphism by (a) and Remark~\ref{rmk:roots-of-unity-condition}.
By (b) and \eqref{eq:2-unit-rank-real-cyclo} we have
\[
\rank_{2}(E_{n}^{+})
=
\textstyle{\frac{1}{2}}\varphi(n)
= 
\dim_{\F_{2}}(\eta(C_{n}^{+}))
\leq 
\dim_{\F_{2}}(\eta(E_{n}^{+})) 
\leq 
\rank_{2}(E_{n}^{+}),
\]
and so these inequalities are in fact equalities. 
In particular, 
\[
\dim_{\F_{2}}(\overline{\eta}(E_{n}^{+}/(E_{n}^{+})^{2})) 
= \dim_{\F_{2}}(\eta(E_{n}^{+}))
=\rank_{2}(E_{n}^{+}),
\]
and so $\overline{\eta}$ is injective.
Moreover, 
$\dim_{\F_{2}}(\eta(C_{n}^{+}))=\rank_{2}(E_{n}^{+})$
and $\{ \pm 1 \} \subset  C_{n}^{+}$.
Hence by Corollary~\ref{cor:p-sat}(d),
$C_{n}^{+}$ is $2$-saturated and of finite index in $E_{n}^{+}$,
or equivalently, $[E_{n}^{+}:C_{n}^{+}]$ is odd.
Thus $h_{n}^{+}$ is also odd by Theorem~\ref{thm:sinnott-cyclo},
and so $H_{2}(\Q(\zeta_{n})^{+})=\Q(\zeta_{n})^{+}$.
Therefore $\Q(\zeta_{n})^{+}$ is $2$-rational by Theorem~\ref{thm:nqd-char-p-rat}.
\end{proof}

\subsection{Auxiliary results}\label{subsec:real-cyclo-aux}
We give a selection of results that are not strictly necessary given the results 
of \S \ref{subsec:real-cyclo-p-odd-rat} and \S \ref{subsec:real-cyclo-2-rat}, but 
will nonetheless be useful for speeding up certain computations.

\begin{lemma}
If $n=2^{k}$, $2^{k} \cdot 3$ or $2^{k} \cdot 5$
for some $k \in \Z_{\geq 0}$ then $\Q(\zeta_{n})^{+}$ is $2$-rational.    
\end{lemma}

\begin{proof}
The claim follows easily from \cite[IV, 3.5.1]{MR1941965}, once one observes that 
\[
\Q(\sqrt{-3}) \cong \Q(\zeta_{3}) \quad \text{ and } \quad
\textstyle{\Q\left(\sqrt{\sqrt{5}  \left(\frac{1-\sqrt{5}}{2}\right) }\right)} \cong \Q(\zeta_{5}). \qedhere
\]    
\end{proof}

\begin{lemma}\label{lem:cyclo-going-up}
Let $m \in \Z_{\geq 1}$ with $m \not \equiv 2 \bmod{4}$.
Let $p$ be an odd prime number and let $k \in \Z_{\geq 1}$.
Then $\Q(\zeta_{p^{k}m})^{+}$ is $p$-rational if and only if $\Q(\zeta_{pm})^{+}$ is $p$-rational.     
\end{lemma}

\begin{proof}
Since $\Q(\zeta_{p^{k}m})^{+}/\Q(\zeta_{pm})^{+}$ is a finite Galois $p$-extension that
is unramified outside $S_{p}(\Q(\zeta_{pm})^{+})$, this is a special case of
Corollary~\ref{cor:p-rat-going-up}.
\end{proof}

\begin{lemma}\label{lem:p-regular}
Let $p$ be an odd prime number and let $k \in \Z_{\geq 1}$.
Then $\Q(\zeta_{p^{k}})^{+}$ is $p$-rational if and only if $p$ is a regular prime,
that is, $p \nmid h_{\Q(\zeta_{p})}$.
\end{lemma}

\begin{proof}
See \cite[Example (ii)(a), p.\ 503]{MR1324685} or 
\cite[discussion after Corollary~2.2]{MR4192837} for the case $k=1$.
The general case then follows from Lemma~\ref{lem:cyclo-going-up}.
\end{proof}

\section{Real abelian number fields}\label{sec:real-abelian-number-fields}

\subsection{Cyclotomic units}
By the Kronecker--Weber theorem, 
for an abelian number field $K$,
there exists a minimal $n \in \Z_{\geq 1}$
such that $K \subseteq \Q(\zeta_{n})$; this is called the \emph{conductor}
of $K$. By Remark~\ref{rmk:n-not-2-mod-4}, we have $n \not \equiv 2 \bmod 4$.
For an arbitrary abelian number field $K$, 
there are several definitions of cyclotomic units that give distinct subgroups in general;
we refer the reader to \cite{MR2183100} for a helpful overview. 
By \cite[Proposition~1]{MR1057325}, the following definition is equivalent 
to that of Sinnott in \cite[\S 4]{MR595586}.

\begin{definition}\label{def:cyclo-units-abelian}
For an abelian number field $K$ of conductor $n \geq 3$, let $D_{K}$
be the subgroup of the multiplicative group $K^{\times}$ generated by 
\[
\{ \pm\Norm_{\Q(\zeta_{m})/\Q(\zeta_{m}) \cap K}(1-\zeta_{m}^a) 
: m \in \Z_{>1}, \, m \mid n, \, (a,m)= 1 \},
\]
and let $C_{K} = D_{K} \cap \mathcal{O}_{K}^{\times}$.
\end{definition}

The following result is due to Sinnott
\cite[Theorem~4.1 and Proposition~5.1]{MR595586}.

\begin{theorem}\label{thm:sinnott-abelian}
Let $F$ be a real abelian number field of conductor $n \geq 3$.
Then 
\[
[\mathcal{O}_{F}^{\times} : C_{F}] 
= 2^{[F:\Q]-1} 
\left(\frac{\prod_{i=1}^{r}[F \cap \Q(\zeta_{p_{i}^{e_{i}}}):\Q]}{[F:\Q]}\right) u_{F}h_{F},
\]
where $u_{F} \in \Z_{\geq 1}$ is divisible only by primes dividing $[F:\Q]$
and $n=p_{1}^{e_{1}} \cdots p_{r}^{e_{r}}$ is the prime factorisation of $n$.
\end{theorem}

\begin{remark}\label{rmk:cF=1}
In \cite{MR595586}, $u_{F}$ is a certain generalised module index $(R:U)$.
In fact, $u_{F}=1$ in the following situations:
\begin{enumerate}
    \item At most two finite primes ramify in $F/\Q$ \cite[Theorem~5.1]{MR595586};
    \item The Galois group $\Gal(F/\Q)$ is cyclic \cite[Theorem~5.3]{MR595586};
    \item The Galois group $\Gal(F/\Q)$ is the direct product of its inertia
    subgroups \cite[Theorem~5.4]{MR595586}. 
    Note that in this case we have 
    $[\mathcal{O}_{F}^{\times} : C_{F}] = 2^{[F:\Q]-1}h_{F}$.  
\end{enumerate}
\end{remark}

\subsection{$p$-rationality for $p$ satisfying certain restrictions}
We now give a generalisation of Theorem~\ref{thm:p-odd-rat-cyclo} to 
arbitrary real abelian number fields, but at the cost of excluding certain prime numbers. 
Again, the advantage of the following result is that it only requires direct knowledge
of $C_{F}$, and not of $\mathcal{O}_{F}^{\times}$ or $h_{F}$.

\begin{theorem}\label{thm:real-abelian-p-rat}
Let $F$ be a real abelian number field of conductor $n \geq 3$ and
let $p$ be an odd prime number such that $p^{2} \nmid n$.
If $\Gal(F/\Q)$ is not the direct product of its inertia subgroups,
assume in addition that $p \nmid [F:\Q]$.
Then $F$ is $p$-rational if and only if 
\begin{enumerate}
    \item no prime of $F$ above $p$ splits completely in $F(\zeta_{p})/F$, and 
    \item $\dim_{\F_{p}}(\eta_{F,p}(C_{F}))=[F:\Q]-1$.
\end{enumerate}
\end{theorem}

\begin{proof}
Since $p^{2} \nmid n$, each prime of $F$ above $p$ is at most tamely ramified in $F/\Q$.
Hence if $F$ is $p$-rational then $p \nmid h_{F}$ by Corollary~\ref{cor:tot-real-p-rat-tame}.

By Theorem~\ref{thm:sinnott-abelian} we have 
$[\mathcal{O}_{F}^{\times} : C_{F}] = 2^{[F:\Q]-1} k_{F}h_{F}$ for some 
$k_{F} \in \Q_{>0}$ such that, when written in lowest terms, every prime occurring 
in the numerator or denominator of $k_{F}$ divides $[F:\Q]$.
In fact, if $\Gal(F/\Q)$ is the direct product of its inertia subgroups
then $k_{F}=1$ by Remark~\ref{rmk:cF=1}. Therefore the hypotheses imply that 
$v_{p}([\mathcal{O}_{F}^{\times} : C_{F}])=v_{p}(h_{F})$.

Given these two observations, the proof proceeds as in Theorem~\ref{thm:p-odd-rat-cyclo}.
\end{proof}

\begin{remark}
If $p \nmid n$ then condition (a) of Theorem~\ref{thm:real-abelian-p-rat} holds automatically.
\end{remark}

\subsection{Reduction to the cyclic case}
Thanks to the following result of Greenberg \cite[Proposition~3.6]{MR3512524}, 
in many situations the problem of determining $p$-rationality for an abelian number field 
may be reduced to determining $p$-rationality for its cyclic subfields.

\begin{prop}\label{prop:Greenberg-abelian-reduction}
Let $K$ be a finite abelian extension of $\Q$ and let $p$ be a 
prime number such that $p \nmid [K:\Q]$. 
Then $K$ is $p$-rational if and only if every cyclic extension of $\Q$ contained
in $K$ is $p$-rational. 
\end{prop}

\subsection{Real cyclic fields of odd prime degree}\label{subsec:Lim-criterion}
Lim \cite[Proposition~4.1]{MR4330938} has given a criterion for the $p$-rationality 
of real cyclic number fields of odd prime degree in terms of generalised 
Bernoulli numbers,
under the hypothesis that $p$ does not divide the conductor. 
This is then used to show the following result \cite[Theorem~4.6]{MR4330938}.

\begin{theorem}\label{thm:Lim-safe}
Let $\ell$ be an odd Sophie Germain prime and let $q=2\ell + 1$. 
If $p$ is an odd prime number that is a primitive root modulo $\ell$ and 
$p < 4\ell$ then $\Q(\zeta_{q})^{+}$ is $p$-rational.
\end{theorem}

\section{The Schirokauer map}\label{sec:Schirokauer-map}

\subsection{Definition and basic properties of the Schirokauer map}\label{subsec:Schirokauer-def-basic}

Let $K$ be a number field and let $p$ be a prime number.
The following map underpins the computational results of this article.

\begin{definition}\label{def:Schirokauer}
Let $\epsilon_{p}=\epsilon_{K,p}$ be the exponent of the finite group $(\mathcal{O}_{K}/p\mathcal{O}_{K})^{\times}$.
The \textit{Schirokauer map} on $\mathcal{O}_{K}^{\times}$ at $p$
is defined to be the function
\[
\psi_{K,p} \colon \mathcal{O}_{K}^\times \longrightarrow 
\frac{\mathcal{O}_{K}}{p\mathcal{O}_{K}}, 
\quad
x \longmapsto
\frac{x^{\epsilon_{p}} - 1}{p} \bmod p\mathcal{O}_{K}.
\]
\end{definition}

\begin{remark}\label{rmk:comp-of-maps-makes-Schirokauer}
The Schirokauer map $\psi_{K,p}$ is the composition of group
homomorphisms
\begin{equation}\label{eq:comp-of-maps-makes-Schirokauer}
\mathcal{O}_{K}^\times 
\longrightarrow 
\left( \frac{\mathcal{O}_{K}}{p^{2}\mathcal{O}_{K}} \right)^{\times}
\xrightarrow{\, \uparrow \epsilon_{p} \,}
\left( \frac{1+p\mathcal{O}_{K}}{1+p^{2}\mathcal{O}_{K}} \right)
\xrightarrow{\, \cong \, } 
\frac{p\mathcal{O}_{K}}{p^{2}\mathcal{O}_{K}}
\xrightarrow{\, \cong \, }
\frac{\mathcal{O}_{K}}{p\mathcal{O}_{K}},
\end{equation} 
where the first arrow is induced by the canonical projection,
the second arrow raises elements to the power of $\epsilon_{p}$,
the third arrow is
$1+px \bmod{p^{2}\mathcal{O}_{K}} \mapsto px \bmod{p^{2}\mathcal{O}_{K}}$, 
and the fourth arrow is
$px \bmod{p^{2}\mathcal{O}_{K}} \mapsto x \bmod{p\mathcal{O}_{K}}$.
In particular, $\psi_{K,p}$ is itself a group homomorphism.
\end{remark}

\begin{remark}\label{rmk:ep-ram-or-unram}
If $p$ is unramified in $K/\Q$ then $\epsilon_{p}$ is the least common multiple
of the quantities $|(\mathcal{O}_{K}/\mathfrak{p})^{\times}| = \Norm(\mathfrak{p})-1$ as $\mathfrak{p}$ ranges over all primes of $K$ above $p$, and so, in particular,
$p \nmid \epsilon_{p}$.
If some prime above $p$ is ramified in $K/\mathbb{Q}$ then $p \mid \epsilon_{p}$. 
\end{remark}

\begin{remark}
If $p$ is unramified in $K/\Q$ then the map $\lambda_{1}$ defined by Schirokauer 
\cite[p.~412]{MR1253502} restricted to $\mathcal{O}_{K}^{\times}$ is equal 
to the composition of the first three arrows of \eqref{eq:comp-of-maps-makes-Schirokauer}.  
For the relation of $\lambda_{1}$ to the $p$-adic logarithm, see \cite[p.~413]{MR1253502}
(see also \cite[\S 4.3]{MR4158582}).
\end{remark}

\subsection{The relation between $\psi_{K,p}$ and $\eta_{K,p}$}
The following lemma underpins many of the results of this article.

\begin{lemma}\label{lem:p-ranks-eta-psi-equal}
Let $K$ be a number field and let $p$ be a prime number such that $p \nmid 2d_{K}$.
Then $\ker(\eta_{K,p})=\ker(\psi_{K,p})$ and $\dim_{\F_{p}} \eta_{K,p}(B) = \dim_{\F_{p}} \psi_{K,p}(B)$ for every subgroup $B \leq \mathcal{O}_{K}^{\times}$.
\end{lemma}

\begin{proof}
Since $p \nmid d_{K}$, the prime $p$ is unramified in $K/\Q$ and hence
$p\mathcal{O}_{K} = \prod_{\mathfrak{p} \in S_{p}(K)} \mathfrak{p}$.
The Chinese remainder theorem therefore gives the canonical isomorphism
\begin{equation}\label{eq:CRT-1-unit-decomp}
\left( \frac{1+p\mathcal{O}_{K}}{1+p^{2}\mathcal{O}_{K}} \right)
\xrightarrow{\, \cong \, } 
\prod_{\mathfrak{p} \in S_{p}(K)}
\left( \frac{1 + \mathfrak{p}}{1 + \mathfrak{p}^{2}} \right).    
\end{equation}

Let $\mathfrak{p} \in S_{p}(K)$.
For $i \geq 1$, set $U_{\mathfrak{p},i} = 1 + \widehat{\mathfrak{p}}^{i}$ where $\widehat{\mathfrak{p}} = \mathfrak{p}\mathcal{O}_{K_{\mathfrak{p}}}$.
Since $K_{\mathfrak{p}}/\Q_{p}$ is unramified and $p$ is odd, 
\cite[Chapter~XIV, \S 4, Proposition~9]{MR554237} gives
$U_{\mathfrak{p},2} = U_{\mathfrak{p},1}^{p}$.
Therefore completion at $\mathfrak{p}$ and 
the inclusion $U_{\mathfrak{p},1} \hookrightarrow U_{\mathfrak{p}}$ 
consequently induce isomorphisms
\begin{equation}\label{eq:global-to-local-1-units}
\left( \frac{1 + \mathfrak{p}}{1 + \mathfrak{p}^{2}} \right)
\xrightarrow{\, \cong \, } 
\frac{U_{\mathfrak{p},1}}{U_{\mathfrak{p},2}}
=
\frac{U_{\mathfrak{p},1}}{U_{\mathfrak{p},1}^{p}}
\xrightarrow{\, \cong \, } 
\frac{U_{\mathfrak{p}}}{U_{\mathfrak{p}}^{p}}.
\end{equation}
For the last isomorphism of \eqref{eq:global-to-local-1-units}, 
we have used the fact that raising elements to the $p$-th power 
induces an automorphism of $U_{\mathfrak{p}}/U_{\mathfrak{p},1} \simeq (\mathcal{O}_{K}/\mathfrak{p})^{\times}$.

Combining \eqref{eq:CRT-1-unit-decomp} and \eqref{eq:global-to-local-1-units}, 
we obtain an isomorphism 
\[
\rho_{K,p} : 
\left( \frac{1+p\mathcal{O}_{K}}{1+p^{2}\mathcal{O}_{K}} \right)
\xrightarrow{\, \cong \, } 
\prod_{\mathfrak{p} \in S_{p}(K)}
\frac{U_{\mathfrak{p}}}{U_{\mathfrak{p}}^{p}}.
\]
We also have the canonical isomorphism
\[
\delta_{K,p} : 
\left( \frac{1+p\mathcal{O}_{K}}{1+p^{2}\mathcal{O}_{K}} \right)
\xrightarrow{\, \cong \, } 
\frac{\mathcal{O}_{K}}{p\mathcal{O}_{K}}, \quad 1+px \mapsto x \bmod p\mathcal{O}_{K}.
\]
Set $\Phi_{K,p} = \rho_{K,p} \circ \delta_{K,p}^{-1}$.
Then Remark~\ref{rmk:comp-of-maps-makes-Schirokauer} and the definitions give
\[
\Phi_{K,p}(\psi_{K,p}(x)) = \eta_{K,p}(x)^{\epsilon_{p}}
\]
for every $x \in \mathcal{O}_{K}^{\times}$.
But $p \nmid \epsilon_{p}$ by Remark~\ref{rmk:ep-ram-or-unram},
and so raising to the $\epsilon_{p}$-th power is an automorphism on 
$\prod_{\mathfrak{p} \in S_{p}(K)} U_{\mathfrak{p}}/U_{\mathfrak{p}}^{p}$.
The desired results now follow immediately.
\end{proof}

\subsection{Criteria for $p$-rationality in terms of the Schirokauer map}\label{subsec:crit-p-rat-Schirokauer}
The following result may be viewed as a version of 
Theorem~\ref{thm:nqd-char-p-rat} (which is due to Nguyen Quang Do) in terms
of the Schirokauer map; it is also related to
a result of Barbulescu and Ray \cite[Lemma~4.2]{MR4158582}.
The key advantage is that the Schirokauer map $\psi_{K,p}$
is easier to compute than the maps $\eta_{K,p}$ or $\overline{\eta}_{K,p}$ of
Definition~\ref{def:various-maps}.

\begin{theorem}\label{thm:new-p-rat-criterion}
Let $K$ be a number field and let $p$ be a prime number such that $p \nmid 2d_{K}$. 
Then $K$ is $p$-rational precisely when both of the following conditions are
satisfied:
\begin{enumerate}
    \item The compositum of all $\Z_{p}$-extensions of $K$ contains $H_{p}(K)$.
    \item $\dim_{\mathbb{F}_{p}}(\psi_{K,p}(\mathcal{O}_{K}^{\times}))=\rank_{\Z}(\mathcal{O}_{K}^{\times})$, or equivalently, $\ker(\psi_{K,p})=\mathcal{O}_{K}^{\times p}$.
\end{enumerate}
\end{theorem}

\begin{proof}
All of the following assertions use the hypothesis that $p \nmid 2d_{K}$.
The map $\lambda_{K,p}$ is an isomorphism by Remark~\ref{rmk:roots-of-unity-condition}. 
Moreover, $\mu_{p} \not \subset K$ and so 
$\rank_{\Z}(\mathcal{O}_{K}^{\times})=\rank_{p}(\mathcal{O}_{K}^{\times})$.
Corollary~\ref{cor:p-sat}(b) 
applied to $\psi_{K,p}$ 
therefore shows that 
$\dim_{\mathbb{F}_{p}}(\psi_{K,p}(\mathcal{O}_{K}^{\times}))=\rank_{\Z}(\mathcal{O}_{K}^{\times})$ is equivalent to $\ker(\psi_{K,p})=\mathcal{O}_{K}^{\times p}$.
The desired result now follows from 
Theorem~\ref{thm:nqd-char-p-rat} and Lemma~\ref{lem:p-ranks-eta-psi-equal}.
\end{proof}

We obtain the following useful variant for totally real number fields.

\begin{corollary}\label{cor:tot-real-p-rat-iff}
Let $F$ be a totally real number field and  
let $p$ be a prime number such that $p \nmid 2d_{F}$.
Then the following assertions are equivalent:
\begin{enumerate}
\item $F$ is $p$-rational.
\item $\dim_{\mathbb{F}_{p}}(\psi_{F,p}(\mathcal{O}_{F}^{\times}))=[F:\Q]-1$ and $p \nmid h_{F}$.
\item $\ker(\psi_{F,p})=\mathcal{O}_{F}^{\times p}$ and $p \nmid h_{F}$.
\item $v_{p}(R_{F,p})=[F:\Q]-1$ and $p \nmid h_{F}$.
\end{enumerate}
\end{corollary}

\begin{proof}
The equivalence of (a) and (d) is Corollary~\ref{cor:p-rat-p-reg}.
Note that $\rank_{\Z}(\mathcal{O}_{F}^{\times}) = [F:\Q]-1$
and that if (a) holds then $p \nmid h_{F}$ by 
Corollary~\ref{cor:tot-real-p-rat-tame}. 
Therefore the remaining equivalences now follow from 
Theorem~\ref{thm:new-p-rat-criterion}.
\end{proof}

We also observe that Lemma~\ref{lem:p-ranks-eta-psi-equal} easily gives the 
following versions of Theorems~\ref{thm:p-odd-rat-cyclo} and \ref{thm:real-abelian-p-rat} in terms of the Schirokauer map, subject to the additional condition that $p \nmid 2n$.

\begin{theorem}\label{thm:p-rat-cyclo-Schirokauer} 
Let $n \in \Z_{\geq 3}$ with $n \not \equiv 2 \bmod 4$
and let $p$ be a prime number such that $p \nmid 2n$.
Then $\Q(\zeta_{n})^{+}$ is $p$-rational if and only if
$\dim_{\mathbb{F}_{p}}(\psi_{\Q(\zeta_{n})^{+},p}(C_{n}^{+}))=\frac{1}{2}\varphi(n) - 1$.
\end{theorem}

\begin{theorem}\label{thm:real-abelian-p-rat-Schirokauer} 
Let $F$ be a real abelian number field of conductor $n \geq 3$
and let $p$ be a prime number such that $p \nmid 2n$.
If $\Gal(F/\Q)$ is not the direct product of its inertia subgroups,
assume in addition that $p \nmid [F:\Q]$.
Then $F$ is $p$-rational if and only if 
$\dim_{\F_{p}}(\psi_{F,p}(C_{F}))=[F:\Q]-1$.
\end{theorem}

\begin{remark}
Let $K$ be an abelian CM-field of conductor $n$ and let $p$ be a prime number
such that $p \nmid 2nh_{K}^{-}[K:\Q]$. 
In this situation, Theorem~\ref{thm:real-abelian-p-rat-Schirokauer} applied
to $K^{+}$ can be combined with Corollary~\ref{cor:CM-tot-real-p-rat} to determine the 
$p$-rationality of $K$. In the special case $K=\Q(\zeta_{n})$, one may instead
assume that $p \nmid 2nh_{K}^{-}$ and use Theorem~\ref{thm:p-rat-cyclo-Schirokauer}.
\end{remark}

\begin{remark}
Theorem \ref{thm:real-abelian-p-rat-Schirokauer} is related to
\cite[Theorem 3]{MR5078208}, which implicitly concerns $p$-rationality
and uses a map similar to but different from the Schirokauer map.
\end{remark}

\section{Quasi-$p$-rational fields}\label{sec:quasi-p-rationality}

Let $K$ be a number field.
Let $c_{K}$ denote the number of pairs of complex embeddings of $K$.
Let $p$ be a prime number and let $S_{p}=S_{p}(K)$ be the set of primes of $K$ 
lying above $p$. 
Let $K_{S_{p}}^{\mathrm{ab}}(p)$ be the 
maximal abelian pro-$p$-extension of $K$ unramified outside $S_{p}$, 
and let $H_{p}(K)$ be the $p$-Hilbert class field of $K$. 

\begin{definition}\label{def:weak-p-rat}
If $\Gal(K_{S_{p}}^{\mathrm{ab}}(p)/H_{p}(K)) \simeq \Z_{p}^{c_{K}+1}$
then $K$ is said to be \textit{quasi-$p$-rational}.
\end{definition}

\begin{remark}
We avoid the term \emph{weakly $p$-rational}, which has already been used by Gras for a different weakening of the notion of $p$-rationality \cite[\S 3.2]{MR3320496}.
\end{remark}

\begin{lemma}\label{lem:rat-implies-quasi-rat-implies-leopoldt}
Let $K$ be a number field.
\begin{enumerate}
    \item If $K$ is $p$-rational, then $K$ is quasi-$p$-rational.
    \item If $K$ is quasi-$p$-rational, then $K$ satisfies Leopoldt's conjecture at $p$. 
    \item If $p \nmid h_{K}$ then $K$ is $p$-rational if and only if $K$ is quasi-$p$-rational.
\end{enumerate}
\end{lemma}

\begin{proof}
(a) Suppose that $K$ is $p$-rational. Then 
$X_{S_{p}}(K) = \Gal(K_{S_{p}}^{\mathrm{ab}}(p)/K) \simeq \Z_{p}^{c_{K}+1}$
by Proposition~\ref{prop:def-char-p-rational}.
Since $\Gal(K_{S_{p}}^{\mathrm{ab}}(p)/H_{p}(K))$
is a closed subgroup of $X_{S_{p}}(K) \simeq \Z_{p}^{c_{K}+1}$ of finite index,
it is itself isomorphic to $\Z_{p}^{c_{K}+1}$.
Thus $K$ is also quasi-$p$-rational.

(b) Suppose that $K$ is quasi-$p$-rational. 
Then $X_{S_{p}}(K) \simeq T_{S_p}(K) \oplus \Z_{p}^{c_{K}+1}$ 
since $X_{S_{p}}(K)$ is a finitely generated $\Z_{p}$-module
and $\Gal(K_{S_{p}}^{\mathrm{ab}}(p)/H_{p}(K))$
is a closed subgroup of finite index.
Thus $K$ satisfies Leopoldt's conjecture at $p$ by \cite[(10.3.20)(ii)]{MR2392026}.

(c) If $p \nmid h_{K}$ then $H_{p}(K)=K$ and so $X_{S_{p}}(K) = \Gal(K_{S_{p}}^{\mathrm{ab}}(p)/H_{p}(K))$. Thus the claim follows from 
Proposition~\ref{prop:def-char-p-rational} and the definitions.
\end{proof}

\begin{remark}\label{rmk:quasi-p-rat-Gras-conj}
Since class numbers have finitely many prime factors, a number field satisfies Gras's
Conjecture~\ref{conj:Gras} if and only if it is quasi-$p$-rational 
for all but finitely many $p$.
\end{remark}

\begin{prop}\label{prop:quasi-p-rat-criteria}
A number field $K$ is quasi-$p$-rational precisely 
when both of the following conditions are satisfied:
\begin{enumerate}
    \item The diagonal map $\lambda_{K,p} : \mu_{p}(K) \longrightarrow \prod_{\mathfrak{p} \in S_{p}(K)} 
    \mu_{p}(K_{\mathfrak{p}})$ is an isomorphism.
    \item The diagonal map 
    $\overline{\eta}_{K,p} \colon \mathcal{O}_{K}^{\times}/\mathcal{O}_{K}^{\times p} \longrightarrow 
    \prod_{\mathfrak{p} \in S_{p}(K)} U_{\mathfrak{p}}/ U_{\mathfrak{p}}^{p}$ is injective.
\end{enumerate}
\end{prop}

\begin{proof}
We adapt the proof of \cite[Proposition~2.3]{MR4192837}, 
whose statement was recalled in Theorem~\ref{thm:nqd-char-p-rat}.
For an abelian group $A$ written multiplicatively, define the $p$-adic completion of $A$ to be
$\widehat{A} := \textstyle{\varprojlim_{n}} A/A^{p^{n}}$.
By definition, the Leopoldt kernel $D_{K}$ is the kernel of the
diagonal map $\mathcal{O}_{K}^{\times} \longrightarrow \prod_{\mathfrak{p} \in S_{p}(K)} U_{\mathfrak{p}}$ after passing to $p$-adic completions 
and we have an exact sequence
provided by global class field theory
\begin{equation}\label{eqn:four-term-exact-gcft}
1 
\longrightarrow
D_{K}
\longrightarrow
\widehat{\mathcal{O}_{K}^{\times}}
\xrightarrow{\widehat{\eta_{K,p}}}
\textstyle{\prod_{\mathfrak{p} \in S_{p}(K)}} \widehat{U_{\mathfrak{p}}}
\longrightarrow
\Gal(K_{S_{p}}^{\mathrm{ab}}(p)/H_{p}(K))
\longrightarrow
1
\end{equation}
(see \cite[Corollary~13.6]{MR1421575}).
Since $\widehat{\mathcal{O}_{K}^{\times}} = \Z_{p} \otimes_{\Z} \mathcal{O}_{K}^{\times}$,
we have that $D_{K}=0$ if and only if Leopoldt's conjecture holds for $K$ at $p$
by \cite[(10.3.6)(iii)]{MR2392026}.

Suppose that $K$ is quasi-$p$-rational. 
Then Leopoldt's conjecture holds for $K$ at $p$ by Lemma~\ref{lem:rat-implies-quasi-rat-implies-leopoldt}(b) and so $D_{K}=0$. Thus we obtain a commutative diagram
\begin{equation}\label{eqn:snake-lem-diag} 
\begin{gathered}
\xymatrix{
1 \ar[r] & \widehat{\mathcal{O}_{K}^{\times}} \ar[r] \ar[d] & 
\textstyle{\prod_{\mathfrak{p} \in S_{p}(K)}} \widehat{U_{\mathfrak{p}}}
\ar[r] \ar[d] &
\Gal(K_{S_{p}}^{\mathrm{ab}}(p)/H_{p}(K)) \ar[d] \ar[r]  & 1  \\
1 \ar[r] & \widehat{\mathcal{O}_{K}^{\times}} \ar[r] & 
\textstyle{\prod_{\mathfrak{p} \in S_{p}(K)}} \widehat{U_{\mathfrak{p}}} \ar[r] &
\Gal(K_{S_{p}}^{\mathrm{ab}}(p)/H_{p}(K)) \ar[r]  & 1, 
}
\end{gathered}
\end{equation}
where the vertical maps raise elements to the $p$-th power.
Conditions (a) and (b) hold by observing that the right vertical map is injective and
applying the snake lemma.

Suppose conversely that conditions (a) and (b) hold. 
Consider the following commutative diagram in which the rows
are obtained from \eqref{eqn:four-term-exact-gcft} and the vertical maps
again raise elements to the $p$-th power
\begin{equation}\label{eqn:show-Dmod-pth-powers-trivial} 
\begin{gathered}
\xymatrix{
1 \ar[r] & D_{K} \ar[r] \ar[d] & \widehat{\mathcal{O}_{K}^{\times}} \ar[r] \ar[d] & 
\im(\widehat{\eta_{K,p}})
\ar[r] \ar[d]   & 1  \\
1 \ar[r] & D_{K} \ar[r] & \widehat{\mathcal{O}_{K}^{\times}} \ar[r] & 
\im(\widehat{\eta_{K,p}}) \ar[r]  & 1.
}
\end{gathered}
\end{equation}
Then using condition (a) and applying the snake lemma gives an injection 
$D_{K}/D_{K}^{p} \hookrightarrow \ker(\overline{\eta}_{K,p})$.
The latter kernel is trivial by condition (b), and so $D_{K}/D_{K}^{p}$ is also trivial. 
Since $D_{K}$ is a finitely generated $\Z_{p}$-module, Nakayama's lemma implies that $D_{K}=0$.
Hence the exact sequence \eqref{eqn:four-term-exact-gcft} now gives the diagram 
\eqref{eqn:snake-lem-diag}. Applying the snake lemma to \eqref{eqn:snake-lem-diag} 
and using conditions (a) and (b) gives that $\Gal(K_{S_{p}}^{\mathrm{ab}}(p)/H_{p}(K))$ is torsion-free as a $\Z_{p}$-module. 
But Leopoldt's conjecture holds for $K$ at $p$ since $D_{K}=0$
and so $\rank_{\Z_{p}} X_{S_{p}}(K) = c_{K} + 1$ by \cite[(10.3.20)(ii)]{MR2392026}.
Since $\Gal(K_{S_{p}}^{\mathrm{ab}}(p)/H_{p}(K))$ is 
a closed subgroup of
$X_{S_{p}}(K)$ of finite index, we conclude that
$\Gal(K_{S_{p}}^{\mathrm{ab}}(p)/H_{p}(K)) \simeq \Z_{p}^{c_{K}+1}$,
that is, $K$ is quasi-$p$-rational.
\end{proof}

\begin{corollary}\label{cor:char-quasi-p-rat-Schirokauer}
Let $K$ be a number field and let $p$ be a prime number such that $p \nmid 2d_{K}$.
Then $K$ is quasi-$p$-rational if and only if
$\dim_{\F_p}(\psi_{K,p}(\mathcal{O}_{K}^{\times}))=\rank_{\Z}(\mathcal{O}_{K}^{\times})$.
\end{corollary}

\begin{proof}
Since $p \nmid 2d_{K}$, condition (a) of Proposition~\ref{prop:quasi-p-rat-criteria}
is satisfied by Remark~\ref{rmk:roots-of-unity-condition}.    
Thus $K$ is quasi-$p$-rational if and only if $\overline{\eta}_{K,p}$ is injective.
But
$\overline{\eta}_{K,p}$ is injective if and only if 
$\dim_{\F_{p}}(\eta_{K,p}(\mathcal{O}_{K}^{\times})) 
= \rank_{p}(\mathcal{O}_{K}^{\times})$ by the rank-nullity theorem.
Moreover, $\mu_{p} \not \subset K$ and so 
$\rank_{\Z}(\mathcal{O}_{K}^{\times})=\rank_{p}(\mathcal{O}_{K}^{\times})$.
Therefore the desired result now follows from Lemma~\ref{lem:p-ranks-eta-psi-equal}.
\end{proof}

\section{Quasi-$p$-rationality and $p$-saturation of units}\label{sec:p-rat-p-sat}

\subsection{Showing that a subgroup of units is $p$-saturated}\label{subsec:showing-units-p-saturated}
The following result can be used to show that a subgroup of the unit group
$\mathcal{O}_{K}^{\times}$ of a number field
$K$ is $p$-saturated; it does not require the notions of $p$-rationality or quasi-$p$-rationality.

\begin{prop}\label{prop:psi-eta-give-p-sat}
Let $K$ be a number field and let $U$ be a subgroup of $\mathcal{O}_{K}^{\times}$.
Let $p$ be a prime number. Suppose that either of the following conditions holds:
\begin{enumerate}
    \item $p \nmid 2d_{K}$ and either $\dim_{\F_p}(\psi_{K,p}(U))=\rank_{\Z}(\mathcal{O}_{K}^{\times})$
    or $\overline{\psi_{K,p} \vert_{U}}$ is injective;
    \item $\mu_{p}(K) \leq U$ and either $\dim_{\F_p}(\eta_{K,p}(U))=\rank_{p}(\mathcal{O}_{K}^{\times})$ or $\overline{\eta_{K,p} \vert_{U}}$ is injective.
\end{enumerate}
Then $U$ is $p$-saturated in $\mathcal{O}_{K}^{\times}$.
\end{prop}

\begin{proof}
The desired result follows from applications of 
Corollary~\ref{cor:p-sat}(c)(d) to $\psi_{K,p}$ and $\eta_{K,p}$
with $A=\mathcal{O}_{K}^{\times}$ and $B=U$.
Note that in part (a), the hypothesis $p \nmid 2d_{K}$ implies that $\mu_{p} \not\subset K$
and so $\rank_{\Z}(\mathcal{O}_{K}^{\times})=\rank_{p}(\mathcal{O}_{K}^{\times})$.
\end{proof}

Assuming Gras's Conjecture~\ref{conj:Gras}, the following two results
show that, for a fixed subgroup $U$ of $\mathcal{O}_{K}^{\times}$ of finite index,
Proposition~\ref{prop:psi-eta-give-p-sat} detects the $p$-saturation of 
$U$ in $\mathcal{O}_{K}^{\times}$ except for at most 
finitely many prime numbers $p$ (see Remark~\ref{rmk:quasi-p-rat-Gras-conj}).
Even if Gras's conjecture is false, it is expected that exceptions will be rare
(see Remark~\ref{rmk:Dirichlet-density-0}).

\begin{prop}\label{prop:quasi-p-rat-subgroup-of-units-eta}
Let $K$ be a number field and let $U$ be a subgroup of $\mathcal{O}_{K}^{\times}$ of finite index.
Let $p$ be a prime number and suppose that $U$ is $p$-saturated in $\mathcal{O}_{K}^{\times}$.
Then $K$ is quasi-$p$-rational precisely when
both of the following conditions are satisfied:
\begin{enumerate}
    \item The diagonal map 
$\lambda_{K,p} : \mu_{p}(K) \longrightarrow \prod_{\mathfrak{p} \in S_{p}(K)} \mu_{p}(K_{\mathfrak{p}})$
is an isomorphism.
    \item $\dim_{\F_p}(\eta_{K,p}(U))=\rank_{p}(\mathcal{O}_{K}^{\times})$, or equivalently, 
    $\overline{\eta_{K,p} \vert_{U}}$ is injective.
\end{enumerate}
\end{prop}

\begin{proof}
Since $U$ is $p$-saturated and of finite index in $\mathcal{O}_{K}^{\times}$,
we have (i) $\rank_{p}(U)=\rank_{p}(\mathcal{O}_{K}^{\times})$
and (ii) $\eta_{K,p}(U)=\eta_{K,p}(\mathcal{O}_{K}^{\times})$.
The equivalence of the conditions within part (b) follows from (i) 
and Corollary~\ref{cor:p-sat}(b).
That $K$ is quasi-$p$-rational if and only if both (a) and (b) hold
now follows from (ii) and Proposition~\ref{prop:quasi-p-rat-criteria}. 
\end{proof}

\begin{corollary}\label{cor:quasi-p-rat-subgroup-of-units-schirokauer}
Let $K$ be a number field and let $U$ be a subgroup of $\mathcal{O}_{K}^{\times}$ of finite index.
Let $p$ be a prime number such that $p \nmid 2d_{K}$ and suppose that $U$ is $p$-saturated in $\mathcal{O}_{K}^{\times}$.
Then the following assertions are equivalent:
\begin{enumerate}
    \item $K$ is quasi-$p$-rational.
    \item $\dim_{\F_p}(\psi_{K,p}(U))=\rank_{\Z}(\mathcal{O}_{K}^{\times})$. 
    \item $\overline{\psi_{K,p} \vert_{U}}$ is injective.
\end{enumerate}
\end{corollary}

\begin{proof}
The hypothesis $p \nmid 2d_{K}$ implies that $\mu_{p} \not\subset K$
and so (i) $\rank_{\Z}(\mathcal{O}_{K}^{\times})=\rank_{p}(\mathcal{O}_{K}^{\times})$.
Since $U$ is $p$-saturated and of finite index in $\mathcal{O}_{K}^{\times}$,
we have (ii) $\rank_{p}(U)=\rank_{p}(\mathcal{O}_{K}^{\times})$
and (iii) $\psi_{K,p}(U)=\psi_{K,p}(\mathcal{O}_{K}^{\times})$.
The equivalence of (b) and (c) follows from (i), (ii) and Corollary~\ref{cor:p-sat}(b).
The equivalence of (a) and (b) follows from (iii) and Corollary~\ref{cor:char-quasi-p-rat-Schirokauer}.
\end{proof}

\subsection{Finding the $p$-th powers of units in a subgroup}
We now consider the situation in which a given subgroup of the unit group $\mathcal{O}_{K}^{\times}$ is not necessarily $p$-saturated. 
Under the assumption that $K$ is quasi-$p$-rational, 
the following result identifies precisely those elements of $U$ that are $p$-th powers in $\mathcal{O}_{K}^{\times}$, and hence indicates where one should attempt to extract $p$-th roots.

\begin{prop}\label{prop:ker-subgroup-of-units}
Let $K$ be a number field and let $U$ be a subgroup of $\mathcal{O}_{K}^{\times}$.
\begin{enumerate}
    \item If $K$ is quasi-$p$-rational then
    $\ker(\eta_{K,p} \vert_{U}) = U \cap \mathcal{O}_{K}^{\times p}$.
    \item If $K$ is quasi-$p$-rational and $p \nmid 2d_{K}$ then $\ker(\psi_{K,p} \vert_{U}) = U \cap \mathcal{O}_{K}^{\times p}$. 
\end{enumerate}
\end{prop}

\begin{proof}
Suppose that $K$ is quasi-$p$-rational. 
Then by Proposition~\ref{prop:quasi-p-rat-criteria} the map $\overline{\eta}_{K,p}$ 
is injective and so $\ker(\eta_{K,p})=\mathcal{O}_{K}^{\times p}$.
Thus if $p \nmid 2d_{K}$ then we also have
$\ker(\psi_{K,p})=\mathcal{O}_{K}^{\times p}$ by Lemma~\ref{lem:p-ranks-eta-psi-equal}.
The desired results now follow easily. 
\end{proof}

\section{Maps on $S$-units and $p$-rationality}\label{sec:maps-on-S-units-p-rat}

\subsection{Maps on $S$-units}\label{subsec:maps-on-S-units}
Let $K$ be a number field and let $S$ be a finite set of finite primes of $K$.
Let $p$ be a prime number. 
We extend the Schirokauer map to $S$-units as follows.

\begin{definition}\label{def:Schirokauer-S-units}
Let  $\epsilon_{p}=\epsilon_{K,p}$ be the exponent of $(\mathcal{O}_{K}/p\mathcal{O}_{K})^{\times}$.
Let $S_{p}(K)$ denote the set of primes of $K$ above $p$.
Assume that $S \cap S_{p}(K) = \emptyset$.
The \textit{Schirokauer map} on $\mathcal{O}_{K,S}^{\times}$ at $p$
is defined to be the group homomorphism
\[
\psi_{K,S,p} \colon \mathcal{O}_{K,S}^\times \longrightarrow 
\frac{\mathcal{O}_{K,S}}{p\mathcal{O}_{K,S}} \cong \frac{\mathcal{O}_{K}}{p\mathcal{O}_{K}}, 
\quad
x \mapsto
\frac{x^{\epsilon_{p}} - 1}{p} \bmod p\mathcal{O}_{K,S}.
\]
If $S=\emptyset$ then we abbreviate $\psi_{K,S,p}$ to $\psi_{K,p}$.
\end{definition}

We now package the above map together with the valuation maps at primes in $S$.

\begin{definition}
Assume that $S \cap S_{p}(K) = \emptyset$.
The \textit{augmented Schirokauer map} on $\mathcal{O}_{K,S}^{\times}$ at $p$
is defined to be the group homomorphism
\[
\theta_{K,S,p} \colon \mathcal{O}_{K,S}^{\times} \longrightarrow 
\frac{\mathcal{O}_{K}}{p\mathcal{O}_{K}} \times \prod_{\mathfrak{p} \in S} \frac{\Z}{p\Z},
\quad x \mapsto \left( \psi_{K,S,p}(x), (\ord_{\mathfrak{p}}(x) \bmod{p})_{\mathfrak{p} \in S} \right),
\]
where $\ord_{\mathfrak{p}} \colon K^{\times} \rightarrow \Z$ is the normalised valuation map at $\mathfrak{p}$.
If $S=\emptyset$ then $\theta_{K,S,p}=\psi_{K,p}$.
\end{definition}

It is also possible to extend the definition of $\eta_{K,p}$ to $S$-units
in an entirely analogous manner, but we omit the details for reasons of brevity.

\subsection{$p$-rationality and $S$-units}
We extend the results of \S \ref{subsec:crit-p-rat-Schirokauer} and 
\S \ref{sec:p-rat-p-sat} to $S$-units.

\begin{prop}\label{prop:S-units-map-gives-p-sat} 
Let $K$ be a number field. 
Let $S$ be a finite set of finite primes of $K$ and 
let $p$ be a prime number such that $S \cap S_{p}(K)=\emptyset$.
Let $U$ be a subgroup of $\mathcal{O}_{K,S}^{\times}$
and assume that $\mu_{p}(K) \leq U$.
If $\dim_{\mathbb{F}_{p}}(\theta_{K,S,p}(U)) = \rank_{p}(\mathcal{O}_{K,S}^{\times})$
then $U$ is $p$-saturated and of finite index in $\mathcal{O}_{K,S}^{\times}$.
\end{prop}

\begin{proof}
Apply Corollary~\ref{cor:p-sat}(d) to $\theta_{K,S,p}$ 
with $A=\mathcal{O}_{K,S}^{\times}$ and $B=U$.    
\end{proof}

We now give a characterisation of $p$-rationality in terms of $S$-units.
Recall from Lemma~\ref{lem:rat-implies-quasi-rat-implies-leopoldt}(c) that under the assumption that $p \nmid h_{K}$, a number field $K$ is $p$-rational if and only 
if it is quasi-$p$-rational. 

\begin{prop}\label{prop:p-rat-criterion-S-units} 
Let $K$ be a number field and let $S$ be a finite set of finite primes of $K$.
Let $p$ be a prime number such that $p \nmid 2d_{K}h_{K}$ and $S \cap S_{p}(K)=\emptyset$.
Then the following assertions are equivalent:
\begin{enumerate}
\item $K$ is $p$-rational.
\item $\dim_{\mathbb{F}_{p}}(\theta_{K,S,p}(\mathcal{O}_{K,S}^{\times})) = \rank_{\Z}(\mathcal{O}_{K,S}^{\times})$.
\item $\ker(\theta_{K,S,p}) = \mathcal{O}_{K,S}^{\times p}$.
\end{enumerate}    
\end{prop}

\begin{proof}
Since $p \nmid 2d_{K}$ we have $\mu_{p} \not \subset K$ and so 
by considering the map on $\mathcal{O}_{K,S}^{\times}/\mathcal{O}_{K,S}^{\times p}$
induced by $\theta_{K,S,p}$ and applying the rank-nullity theorem, we have 
(b) $\Longleftrightarrow$ (c).

Assume that (a) holds. Note that $\mathcal{O}_{K,S}^{\times p} \leq \ker(\theta_{K,S,p})$
since the codomain of $\theta_{K,S,p}$ is a vector space over $\F_{p}$.
Thus, to show (c), it suffices to show the reverse 
containment.
We adapt part of the proof of \cite[Proposition~3.8]{MR1253502}.
Let $v \in \ker(\theta_{K,S,p})$. 
Then $\ord_{\mathfrak{p}}(v) \equiv 0 \bmod{p}$ for all $\mathfrak{p} \in S$.
Moreover,  
$\ord_{\mathfrak{p}}(v) = 0$ for all $\mathfrak{p} \notin S$ since $v \in \mathcal{O}_{K,S}^{\times}$.
Hence $v$ generates the $p$-th power of a fractional ideal of $\mathcal{O}_{K}$.
Since $p \nmid h_{K}$, we conclude that $v$ generates the $p$-th power of a principal fractional ideal 
of $\mathcal{O}_{K}$. Let $x$ be a generator of this ideal. 
Then $x \in \mathcal{O}_{K,S}^{\times}$ and $v = x^{p} u$ for some $u \in \mathcal{O}_{K}^{\times}$.
Thus $\theta_{K,S,p}(u) = \theta_{K,S,p}(v)=0$. 
Hence $\psi_{K,p}(u)=0$ and so $u \in \mathcal{O}_{K}^{\times p}$ by 
Theorem~\ref{thm:new-p-rat-criterion}.
Therefore (c) holds. 

Suppose conversely that (c) holds. Let $u \in \ker(\psi_{K,p})$.
Then $u \in \ker(\theta_{K,S,p})=\mathcal{O}_{K,S}^{\times p}$ and so
$u=v^{p}$ for some $v \in \mathcal{O}_{K,S}^{\times}$. 
But $u \in \mathcal{O}_{K}^{\times}$ and so $v \in \mathcal{O}_{K}^{\times}$.
Thus $u \in \mathcal{O}_{K}^{\times p}$. 
Therefore $\ker(\psi_{K,p}) \leq \mathcal{O}_{K}^{\times p}$.
In fact, this containment is an equality because the codomain
of $\psi_{K,p}$ is a vector space over $\F_{p}$.
Hence (a) holds by Theorem~\ref{thm:new-p-rat-criterion}.
\end{proof}

\begin{corollary}
    Let $K$ be a number field and let $S$ be a finite set of finite primes of $K$.
Let $U$ be a subgroup of $\mathcal{O}_{K,S}^{\times}$.
Let $p$ be a prime number and suppose that $U$ is $p$-saturated and of finite index in $\mathcal{O}_{K,S}^{\times}$.
If $K$ is $p$-rational, $S \cap S_{p}(K)=\emptyset$ and $p \nmid 2d_{K}h_{K}$
then $\dim_{\F_p}(\theta_{K,S,p}(U))=\rank_{\Z}(\mathcal{O}_{K,S}^{\times})$.
\end{corollary}

\begin{proof}
By Proposition~\ref{prop:p-rat-criterion-S-units} we have
$\dim_{\mathbb{F}_{p}}(\theta_{K,S,p}(\mathcal{O}_{K,S}^{\times})) = \rank_{\Z}(\mathcal{O}_{K,S}^{\times})$. 
Since $U$ is $p$-saturated and of finite index in $\mathcal{O}_{K,S}^{\times}$
we have $\theta_{K,S,p}(U)=\theta_{K,S,p}(\mathcal{O}_{K,S}^{\times})$.
Combining these two equalities gives the desired result. 
\end{proof}

\begin{corollary}\label{cor:ker-subgroup-of-S-units}
Let $K$ be a number field and let $S$ be a finite set of finite primes of $K$.
Let $U$ be a subgroup of $\mathcal{O}_{K,S}^{\times}$ and let $p$ be a prime number.
If $K$ is $p$-rational, $p \nmid 2d_{K}h_{K}$ and $S \cap S_{p}(K)=\emptyset$
then $\ker(\theta_{K,S,p} \vert_{U}) = U \cap \mathcal{O}_{K,S}^{\times p}$.
\end{corollary}

\begin{proof}
This follows immediately from the implication (a) $\Longrightarrow$ (c)
of Proposition~\ref{prop:p-rat-criterion-S-units}.
\end{proof}

\section{Conventions on complexity and runtime}\label{sec:coventions-comp-runtime}

We briefly recall the conventions used for the complexity analysis of our algorithms;
see Lenstra~\cite{MR1129315} or Cohen~\cite[\S{}1.1]{MR1228206} for further details.
We use bit complexity throughout.
Thus the size of an input is the total bit length of its chosen encoding, using the encoding conventions specified below.
We write $\poly(x_{1}, \ldots, x_{r})$ for an 
unspecified polynomial with non-negative coefficients in the indicated variables;
the polynomial may vary between occurrences.

We follow the encoding conventions of \cite{MR1129315}.
In particular, the ring of integers $\mathcal{O}_K$ of a number field $K$ of degree $d$ is represented by an 
integral basis $\omega_1,\dotsc,\omega_d \in \mathcal{O}_K$
and the corresponding structure constants $a_{ijk} \in \Z$ satisfying $\omega_i \omega_j = \sum_{k=1}^{d} a_{ijk} \omega_k$.
By~\cite[\S 2.10]{MR1129315}, the basis may be chosen 
so that this encoding has size 
$O(\poly(d, \log\lvert d_K\rvert))$.
An element 
$x = x_{1} \omega_{1} + \dotsb + x_{d} \omega_{d} \in \mathcal{O}_{K}$ is encoded by the list 
$x_1,\dotsc,x_d \in \Z$. 
For convenience, we define $l(x) = d \max_i \log (2 + \lvert x_i\rvert)$, which is an upper bound,
up to a constant factor, for the size of $x$.
A subgroup $U \leq \mathcal{O}_K^{\times}$ will be represented by a chosen non-empty finite generating set. 
If $U = \langle u_1,\dots,u_s\rangle$, we set $l(U) = s \max_{i} l(u_i)$, which is an upper bound on the size of $U$, again up to a constant factor. 
Thus $l(U)$ refers to the size of the supplied generating set, rather than an intrinsic invariant of $U$.

All probabilistic algorithms in this article are understood to be Las Vegas algorithms: they always return a correct answer and runtime bounds are expected runtime bounds. In particular, whenever we describe an algorithm as probabilistic polynomial-time, its expected runtime is polynomial in the size of the input.

\section{Computations involving $\psi_{K,p}$ and $\eta_{K,p}$}\label{sec:comp-maps}

\subsection{Computations involving $\psi_{K,p}$}\label{subsec:compute-psi}
For the convenience of the reader, 
we state and prove the following folklore complexity result 
(see \cite[\S 6.2.3]{zbMATH07854947}, for example).

\begin{lemma}\label{lem:schirokauer-element-runtime}
There exists a probabilistic algorithm that, given a number field $K$ of degree $d$, 
its ring of integers $\mathcal{O}_{K}$, 
a unit $u \in \mathcal{O}_{K}^{\times}$, and a prime number $p$,
determines the image $\psi_{K,p}(u)$ with runtime in 
\[ O(\poly(d, \log\lvert  d_K \rvert, l(u), \log(p))). \] 
\end{lemma}

\begin{proof}
The exponent $\epsilon_{p}$ can be determined from the decomposition $p\mathcal{O}_{K} = \mathfrak{p}_{1}^{e_1} \dotsm \mathfrak{p}_{g}^{e_g}$ of $p$ into coprime prime ideal powers 
as well as the inertia degree of each $\mathfrak{p}_{i}$.
This can be computed in probabilistic polynomial time by~\cite[Theorem~4.9]{MR1129315}.
Using binary exponentiation, $v \in \mathcal{O}_{K}$ with 
$v \equiv u^{\epsilon_{p}} \bmod p^{2} \mathcal{O}_{K}$ can be determined
with the claimed runtime since $\log(\epsilon_{p}) \leq [K:\Q] \cdot \log(p)$.
We then return $(v - 1)/p \bmod p\mathcal{O}_{K}$, which is computable coordinatewise in polynomial time.
\end{proof}

\begin{prop}\label{prop:alg-schiro-image}
There exists a probabilistic algorithm that, given a number field $K$ of degree $d$, its ring of integers $\mathcal{O}_{K}$, 
a finitely generated subgroup $U$ of $\mathcal{O}_{K}^{\times}$, and a prime number $p$,
determines an $\F_{p}$-basis of $\psi_{K,p}(U)$ with runtime in
\[ 
O(\poly(d, \log \lvert  d_K \rvert, l(U), \log(p))). 
\] 
\end{prop}

\begin{proof}
Let $\omega_1,\dotsc,\omega_d$ be the given $\Z$-basis of $\mathcal{O}_{K}$
and note that the images $\overline \omega_1,\dotsc,\overline \omega_d$ 
form an $\F_{p}$-basis of $\mathcal{O}_{K}/p\mathcal{O}_{K}$.
Let $u_{1},\dotsc,u_{s} \in U$ be the given generators of $U$.
It follows from Lemma~\ref{lem:schirokauer-element-runtime} that we can find a matrix $M = (m_{ij})_{ij} \in \operatorname{M}_{s \times d}(\F_{p})$ such that $\psi_{K,p}(u_{i}) = \sum_{j=1}^{d} m_{ij} \overline{\omega}_j$ for $1 \leq i \leq s$ in probabilistic polynomial time in $\log(p)$ and the size of $U$.
Note that a basis of the row space of $M$ yields a basis of $\psi_{K, p}(U)$.
This can be determined using Gaussian elimination, which can be performed with runtime 
in $\poly(s,d,\log(p))$ (see~\cite[Chapter~2]{storjohann}).
\end{proof}

\begin{corollary}\label{cor:alg-schiro-kernel}
There exists a probabilistic algorithm that, given a number field $K$ of degree $d$, its ring of integers $\mathcal{O}_{K}$, 
a finitely generated subgroup $U$ of $\mathcal{O}_{K}^{\times}$, and a prime number $p$,
determines an $\F_p$-basis of $\ker(\overline{\psi_{K,p}|_U})$ with runtime in
\[ 
O(\poly(d, \log \lvert  d_K \rvert, l(U), \log(p))). 
\]     
\end{corollary}

\begin{proof}
Suppose that $v_1,\dotsc,v_r\in U$ are such that
$\overline{v}_1,\dotsc,\overline{v}_r$ form an $\mathbb{F}_p$-basis
of $U/U^p$. Computing the images $\psi_{K,p}(v_i)$ as in the proof of
Proposition~11.2 gives a matrix whose left kernel is naturally
isomorphic to
  $\ker(\overline{\psi_{K,p}|_U})$.
A basis of this kernel can therefore be determined by Gaussian
elimination.

It remains to compute such elements $v_1,\dotsc,v_r$ from the supplied
generators $u_1,\dotsc,u_s$ of $U$.
Let $L$ be the kernel of the map $\mu \colon \Z^s \to U$, $e_i \mapsto u_i$.
A basis of $L$ can be computed in polynomial time by
\cite[Theorem~1.2]{ge}. Let $\overline{L}$ denote the image of $L$
in $\mathbb{F}_p^s$. Since
  $\mu^{-1}(U^p)=L+p\mathbb{Z}^s$,
the map $\mu$ induces an isomorphism
  $\mathbb{F}_p^s/\overline{L} \rightarrow U/U^p$.
Thus, an $\F_p$-basis of this quotient can again be obtained using Gaussian elimination.
All the computations involved have
runtime polynomial in
$d$, $\log|d_K|$, $l(U)$ and $\log(p)$, as required.
\end{proof}

\begin{remark}
In our computational model, elements of $\mathcal{O}_K$, and thus of $\mathcal{O}_K^\times$, are represented via coordinates with respect to a fixed integral basis. In the context of unit group computations, it is necessary to represent elements of $\mathcal{O}_K^\times$ in the form $v_1^{e_1} \dotsm v_n^{e_n}$ with $v_i \in K^\times$ and $e_i \in \Z$; see \cite[\S 5.4]{MR3207410}. It is straightforward to show that the runtime statements of Proposition~\ref{prop:alg-schiro-image} and Corollary~\ref{cor:alg-schiro-kernel} remain true using this alternative representation (see~\cite[Remark~4.19]{MR4440537}).
\end{remark}

\subsection{Improvements in the Galois case}\label{subsec:Schirokauer-improvement-Galois}
Suppose that $K/\Q$ is a finite Galois extension and let $G=\Gal(K/\Q)$.
Let $r=\rank_{\Z}(\mathcal{O}_{K}^{\times})$.
If $K$ is real then let $\sigma_{1}, \ldots, \sigma_{r+1}$ be the elements of $G$. 
If $K$ is complex then let 
$\sigma_{1}, \ldots, \sigma_{r+1}, \tau \circ \sigma_{1}, \ldots, 
\tau \circ \sigma_{r+1}$ be the elements of $G$, where $\tau$ is some fixed
choice of complex conjugation restricted to $K$. 
Let $U$ be a fixed $G$-stable subgroup of $\mathcal{O}_{K}^{\times}$ of finite index.

Assume that we are given an explicit element $u \in U$ such that 
the subgroup $V$ generated by $(\sigma_{i}(u))_{1 \leq i \leq r}$ has finite index in $U$,
or equivalently, finite index in $\mathcal{O}_{K}^{\times}$.
Such an element $u$ is called a Minkowski unit and always exists
(see \cite[Lemma~5.27]{MR1421575}, for example). 
In practice, such an element can usually be found quickly by testing random combinations of a generating set of $U$.
In certain circumstances, such an element $u$ and the index $[U:V]$ are already known explicitly
(see Remark~\ref{rmk:cyclo-Galois-action}).

For this paragraph, fix a prime number $p$ and henceforth abbreviate $\psi_{K,p}$ to $\psi$.
Assume that $p \nmid [U:V]$. 
Then $\psi(U)=\psi(V) = \Span_{\F_{p}}\{\psi(\sigma_{1}(u)), \ldots, \psi(\sigma_{r}(u)) \}$. 
The following improvements can be made to the algorithm of 
Proposition~\ref{prop:alg-schiro-image} in this situation.
\begin{enumerate}
    \item 
    In practice, the most time-consuming step is the determination of $x^{\epsilon} \bmod p^{2} \mathcal{O}_{K}$ during the evaluation of $\psi(x)$.
    Any element $\sigma \in \Gal(K/\Q)$ induces an automorphism of 
    $\mathcal{O}_{K}/p\mathcal{O}_{K}$ which we denote by $\overline{\sigma}$.
    Together with the definition of $\psi$ it follows that $\psi \circ \sigma = \overline{\sigma} \circ \psi$.
    Thus in the present situation it suffices to determine $\psi(u)$
    and $\overline{\sigma_{1}}, \ldots, \overline{\sigma_{r}}$,
    and then use that $\psi(\sigma_{i}(u)) 
    = \overline{\sigma}_{i}(\psi(u))$ for $1 \leq i \leq r$.
    \item
    In the case that $p$ is inert in $K/\Q$, the quotient $\mathcal{O}_{K}/p\mathcal{O}_{K}$ is isomorphic to $\F_{q}$ where $q=p^{[K:\Q]}$. 
    Moreover, $\Gal(K/\Q)$ is isomorphic to $\Gal(\F_{q}/\F_{p})$,
    which is generated by the Frobenius automorphism $x \mapsto x^{p}$.
    Hence, for each $i$, the automorphism $\overline{\sigma}_{i}$ is a power of Frobenius.
    Thus, after computing $\psi(u)$, each 
    $\psi(\sigma_{i}(u)) = \overline{\sigma}_{i}(\psi(u))$ can be obtained
    by iterated application of Frobenius.
\end{enumerate}

Now suppose that we wish to compute 
$\dim_{\mathbb{F}_{p}}(\psi_{K,p}(U))$ for many prime numbers $p$.
Let $u \in U$ be an explicit element such that 
the subgroup $V$ generated by $(\sigma_{i}(u))_{1 \leq i \leq r}$ has finite index in $U$.
Note that $u$ may already be known or computed as described above.
If the index $[U:V]$ is already known, then for 
$p \nmid [U:V]$ proceed as above, and for $p \mid [U:V]$, 
revert to Proposition~\ref{prop:alg-schiro-image}. 
If the index $[U:V]$ is not already known, then it may be expensive to compute.
Under the additional assumption that $p \nmid 2d_{K}$, we may proceed as follows.
Compute $\dim_{\F_{p}}(\psi_{K,p}(V))$ as above. 
We have 
\[
\dim_{\mathbb{F}_{p}}(\psi_{K,p}(V))
\leq
\dim_{\mathbb{F}_{p}}(\psi_{K,p}(U))
\leq 
\rank_{\Z}(\mathcal{O}_{K}^{\times}),
\]
so if the outer two terms are equal then we have determined the middle term;
note that in this case both $V$ and $U$ must be $p$-saturated in $\mathcal{O}_{K}^{\times}$
by Proposition~\ref{prop:psi-eta-give-p-sat}(a) and so $p \nmid [U:V]$.
If $\dim_{\F_{p}}(\psi_{K,p}(V)) < \rank_{\Z}(\mathcal{O}_{K}^{\times})$ then we may revert to 
Proposition~\ref{prop:alg-schiro-image}. 
By Corollary~\ref{cor:quasi-p-rat-subgroup-of-units-schirokauer}, this situation 
can only occur if either $V$ is not $p$-saturated in $\mathcal{O}_{K}^{\times}$ 
or $K$ is not quasi-$p$-rational.

\subsection{Computations involving $\eta_{K, p}$}\label{subsec:compute-eta}
Recall from \S \ref{subsec:nec-suff-p-rat} that $\eta_{K, p}$ is the diagonal map
\[ 
\textstyle{\mathcal{O}_{K}^{\times} 
\longrightarrow \prod_{\mathfrak{p} \in S_{p}(K)} 
U_{\mathfrak{p}}/U_{\mathfrak{p}}^{p}}. 
\]
Fix $\mathfrak{p} \in S_p(K)$ and let $m  = \dim_{\F_{p}} (U_{\mathfrak{p}}/U_{\mathfrak{p}}^{p})$.
Algorithms for computing an $\F_{p}$-basis of $U_{\mathfrak{p}}/U_{\mathfrak{p}}^p$ and images under the induced map $U_{\mathfrak{p}}/U_{\mathfrak{p}}^p \to \F_{p}^{m}$ have been described in~\cite{MR5022766}.
Although the runtime of these algorithms was not analysed in loc.\ cit., 
it seems plausible that these tasks can be 
performed in polynomial time based on the following observation.
The naive algorithm~\cite[Algorithm~1]{MR5022766} exploits the fact that $U_{\mathfrak{p}}/U_{\mathfrak{p}}^{p} \simeq Q^{\times}/Q^{\times p}$, where $Q = \mathcal{O}_{K}/\mathfrak{p}^{l}$ is an appropriate residue ring.
The techniques from~\cite[\S 3]{MR3952015} show that the computation of $Q^{\times}/Q^{\times p }$ can be performed without finding primitive elements of the residue field $\mathcal{O}_{K}/\mathfrak{p}$ or solving discrete logarithms in it.

\section{A $p$-rationality algorithm for real cyclotomic fields}\label{sec:alg-real-cyclo}

\subsection{Determining a basis for $C_{n}^{+}$}\label{subsec:basis-for-cyclo-units}
Gold and Kim \cite{MR1008802} have given an explicit basis (a minimal set of generators) of $C_{n}$, and this can easily be turned into an algorithm.
Define a map $\phi_{n} \colon C_{n} \to \langle \pm \zeta_{n} \rangle$ by 
$u \mapsto u/\overline{u}$, where $\overline{u}$ denotes the complex conjugate of $u$. 
This is a well-defined group homomorphism. 
Note that $C_{n}^{+}$ is the subset of $C_{n}$ consisting of elements fixed by complex conjugation.
Hence $C_{n}^{+} = \ker(\phi_{n})$. 
Therefore we can determine a basis of $C_{n}^{+}$ using \cite[Algorithm~4.1.11]{MR1728313}.

\subsection{An algorithm for testing the $p$-rationality of $\Q(\zeta_{n})^{+}$}
\label{subesc:p-rat-Qzeta_n+}
The key advantage of the following algorithm is that 
it does not require the (very expensive) computation of the 
class number $h_{n}^{+}$ or full unit group $E_{n}^{+}$.
Recall that $\eta_{\Q(\zeta_{n})^{+},p}$ denotes the map of 
Definition~\ref{def:various-maps}(b)
and $\psi_{\Q(\zeta_{n})^{+},p}$ denotes the Schirokauer map of 
Definition~\ref{def:Schirokauer}.

\begin{algorithm}\label{alg:p-rat-cyclo}
Let $n \in \Z_{\geq 3}$ with $n \not \equiv 2 \bmod 4$.
Given $n$ and a prime number $p$, the following algorithm returns 
\ensuremath{\mathsf{true}} if $\Q(\zeta_{n})^{+}$ is $p$-rational 
and returns \ensuremath{\mathsf{false}} otherwise. 
\renewcommand{\labelenumi}{(\arabic{enumi})}
\begin{enumerate}
\item Compute a basis for the cyclotomic units $C_{n}^{+}$.
\item\label{step:p-nmid-n} If $p \nmid 2n$ then
\begin{enumerate}
    \item Compute $\dim_{\mathbb{F}_{p}}(\psi_{\Q(\zeta_{n})^{+},p}(C_{n}^{+}))$. 
    \item If $\dim_{\mathbb{F}_{p}}(\psi_{\Q(\zeta_{n})^{+},p}(C_{n}^{+}))=\frac{1}{2}\varphi(n)-1$ 
    then return \ensuremath{\mathsf{true}};
    otherwise return \ensuremath{\mathsf{false}}.
\end{enumerate}
\item\label{step:p=2}  If $p=2$ then
\begin{enumerate}
    \item If there is more than one prime of $\Q(\zeta_{n})^{+}$ above $2$
    then return \ensuremath{\mathsf{false}}.
    \item Compute $\dim_{\mathbb{F}_{2}}(\eta_{\Q(\zeta_{n})^{+},2}(C_{n}^{+}))$. 
    \item If $\dim_{\mathbb{F}_{2}}(\eta_{\Q(\zeta_{n})^{+},2}(C_{n}^{+}))=\frac{1}{2}\varphi(n)$ 
    then return \ensuremath{\mathsf{true}};
    otherwise return \ensuremath{\mathsf{false}}.
\end{enumerate}
\item\label{step:p-odd-mid-n}  If $p>2$ and $p \mid n$ then
\begin{enumerate}
    \item If there exists a prime of $\Q(\zeta_{n})^{+}$ above $p$ that splits in $\Q(\zeta_{n})/\Q(\zeta_{n})^{+}$ then return \ensuremath{\mathsf{false}}.
    \item Compute $\dim_{\mathbb{F}_{p}}(\eta_{\Q(\zeta_{n})^{+},p}(C_{n}^{+}))$. 
    \item If $\dim_{\mathbb{F}_{p}}(\eta_{\Q(\zeta_{n})^{+},p}(C_{n}^{+}))=\frac{1}{2}\varphi(n)-1$ 
    then return \ensuremath{\mathsf{true}};
    otherwise return \ensuremath{\mathsf{false}}.
\end{enumerate}
\end{enumerate}
\end{algorithm}

\begin{proof}[Proof of correctness]
The correctness of the output in Steps (2), (3)
and (4) follows from Theorems~\ref{thm:p-rat-cyclo-Schirokauer},
\ref{thm:2-rat-cyclo} and \ref{thm:p-odd-rat-cyclo}, respectively.
\end{proof}

\begin{proof}[Further details on each step]
Step~(1) can be performed as described in \S \ref{subsec:basis-for-cyclo-units}.
In Step~(2), $\psi_{\Q(\zeta_{n})^{+},p}(C_{n}^{+})$ can be computed 
as described in \S \ref{subsec:compute-psi}.
In Steps (3) and (4), $\eta_{\Q(\zeta_{n})^{+},p}(C_{n}^{+})$ can be computed 
as described in \S \ref{subsec:compute-eta}.
Since  the ring of integers of $\Q(\zeta_{n})^{+}$ is $\Z[\zeta_{n}+\zeta_{n}^{-1}]$,
the prime decompositions required in Steps (3)(a) and (4)(a)
can be computed using \cite[Theorem~4.8.13]{MR1228206}. 
\end{proof}

\begin{remark}\label{rmk:alg-cyclo-p-rat-complexity}
In the case $p \nmid 2n$, that is, in Step (2), 
$\psi_{\Q(\zeta_{n})^{+},p}(C_{n}^{+})$ can be computed in time $O(\poly(\log(p)))$
(for fixed $n$) by Proposition~\ref{prop:alg-schiro-image}.
In principle, one can instead compute
$\eta_{\Q(\zeta_{n})^{+},p}(C_{n}^{+})$ (see Theorem 
\ref{thm:p-odd-rat-cyclo}), but this is significantly slower.

In the case $p \mid 2n$, that is, in Steps (3) and (4), 
a more efficient but more complicated procedure 
can be given that uses Lemma~\ref{lem:2-rat} and the results of 
\S \ref{subsec:real-cyclo-aux}. 
This modified version has been implemented in \textsc{Hecke}~\cite{MR3703682}.
\end{remark}

\begin{remark}
Recall that Lemma~\ref{lem:real-cyclo-p-rat-class-number} says that if $\Q(\zeta_{n})^{+}$
is $p$-rational then $p \nmid h_{n}^{+}$ (the converse is false).
Thus in many cases Algorithm~\ref{alg:p-rat-cyclo} can be used to show that 
$p \nmid h_{n}^{+}$.
\end{remark}

\begin{remark}\label{rmk:cyclo-Galois-action}
Step (2)(a) may be sped up using the method described in 
\S \ref{subsec:Schirokauer-improvement-Galois}.
In the case that $n$ is an odd prime power, we may use the fact that 
there is an explicit element of $C_{n}^{+}$ that generates $C_{n}^{+} / \{\pm 1\}$ 
as a module over $\Z[\Gal(\Q(\zeta_{n})^{+}/\Q)]$ (see \cite[Proposition~8.11]{MR1421575}).
\end{remark}

\begin{remark}
There is an algorithm analogous to 
Algorithm~\ref{alg:p-rat-cyclo} to test the $p$-rationality 
of arbitrary real abelian number fields that uses the results of \S \ref{sec:real-abelian-number-fields}, though certain prime numbers must be excluded. 
However, this has not been implemented, partly because the degree of $\Q(\zeta_{n})^{+}$
can be much larger than that of a real abelian field $K$ of conductor $n$, making the computation of $C_{K}$ difficult in practice. 
Of course, in one direction one can use the fact that 
the $p$-rationality of $\Q(\zeta_{n})^{+}$
implies the $p$-rationality of $K$ by Lemma~\ref{lem:p-rat-subfields}, 
but the same issue may still be an obstacle.
\end{remark}

\begin{remark}\label{rmk:Lim-alg}
Based on work of Lim 
\cite[Proposition~4.1]{MR4330938},
for a real cyclic field $F$ of odd prime degree, there exists an algorithm 
to determine whether $F$ is $p$-rational, provided that $p$ does not divide
the conductor of $F$ (see also \S\ref{subsec:Lim-criterion}).
However, this requires the computation of a sum involving at least $pf$ terms, 
where $f$ is the conductor of $F$, and so is asymptotically slower
than Algorithm~\ref{alg:p-rat-cyclo}, at least in the case that $p \nmid 2n$
(see Remark~\ref{rmk:alg-cyclo-p-rat-complexity}).
\end{remark}

\begin{remark}\label{rmk:p-rat-full-cyclo}
The relative class number $h_{n}^{-}=h_{\Q(\zeta_{n})}^{-}$ can be determined algorithmically;
see \cite[Tables, \S 3]{MR1421575} for these values for $1 \leq \varphi(n) \leq 256$ and $n \not \equiv 2 \bmod{4}$, and the references given therein
for details of the algorithms. 
By Corollary~\ref{cor:CM-tot-real-p-rat} we have that if
$p \nmid 2nh_{n}^{-}$ then $\Q(\zeta_{n})$ is $p$-rational if and only if
$\Q(\zeta_{n})^{+}$ is $p$-rational. Thus by combining 
Algorithm~\ref{alg:p-rat-cyclo} with
the algorithms for determining $h_{n}^{-}$ 
(or using the aforementioned table), the $p$-rationality of $\Q(\zeta_{n})$
can be determined in many cases. See Example~\ref{ex:cyclic-p-rat} for an illustration.
\end{remark}

\subsection{Computational results}\label{subsec:comp-results-real-cyclo}

We have applied Algorithm~\ref{alg:p-rat-cyclo} to determine 
the $p$-rationality of $\Q(\zeta_{n})^{+}$ for all $5 \leq n \leq 1000$, $n \not\equiv 2 \bmod 4$ and all primes $p < 10^{7}$. 
Since the full data set is too large to reproduce here, we provide the following condensed version.\footnote{The full data set is available at \url{github.com/thofma/p-rationality/}.}
For a conductor $n \in \Z_{\geq 5}$ with $n \not \equiv 2 \bmod{4}$, let
\[ 
c(n) = \lvert \{ p \text{ prime} \mid p < 10^7 \text{ and } \Q(\zeta_{n})^{+} \text{ not $p$-rational} \}\vert 
\]
and for $c \in \Z_{\geq 0}$ let
\[ 
N(c) = \{ 5 \leq n \leq 1000 \mid n \not\equiv 2\bmod 4 \text{ and } c(n) = c\} 
\]
be the set of conductors $n$ 
such that $\Q(\zeta_n)^{+}$ is not $p$-rational for exactly $c$ primes $p < 10^{7}$.
The distribution of $|N(c)|$ is plotted in Figure~\ref{fig:cyclo-result}.

\begin{figure}[ht!]
    \includegraphics[scale=0.88]{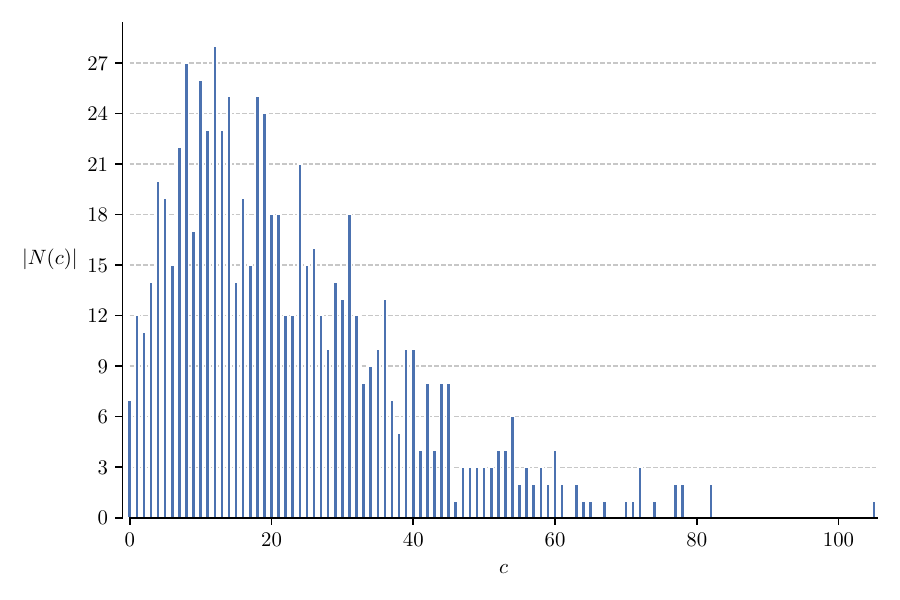}
    \vspace{-1.5em}
    \caption{Number $|N(c)|$ of conductors $n\not\equiv 2 \bmod{4}$ with $5 \leq n \leq 1000$ such that $\Q(\zeta_{n})^{+}$
    is not $p$-rational for exactly $c$ primes $p < 10^7$.}
    \label{fig:cyclo-result}
\end{figure}

\begin{remark}
By \cite[discussion after (3.9)]{MR4129146}, the field
$\Q(\zeta_{5})^{+}=\Q(\sqrt{5})$ is $p$-rational for all primes $p < 6.7 \times 10^{15}$
(see also \cite[\S 4.4]{MR3512524}).
\end{remark}

\begin{remark}
The maximum value of $c(n)$ for $n \leq 1000$ is $105$, which occurs for $n=819$.
Since $\pi(10^{7})=664579$, even in this worst case, $\Q(\zeta_{n})^{+}$
fails to be $p$-rational for fewer than one in every $6329$ prime numbers $p < 10^{7}$.
Thus, whether Gras's Conjecture~\ref{conj:Gras} is true or not, 
this supports the notion that failure of $p$-rationality is `rare' in practice.
\end{remark}

\begin{remark}
For $n \in \{ 5, 179, 227, 479, 503, 719, 983 \}$, 
the field $\Q(\zeta_{n})^{+}$ is $p$-rational for every prime number $p < 10^{7}$.
An interesting observation is that all these values of $n$ are safe primes.
This is perhaps partially explained by the fact that
in this case $[\Q(\zeta_{n})^{+}:\Q]$ is prime and so failure of $p$-rationality
cannot be inherited from proper subfields. Indeed, whenever $n<1000$ is a safe
prime, there are at most four prime numbers $p<10^{7}$ such that $\Q(\zeta_{n})^{+}$
is not $p$-rational. Of course, the full data set is also consistent with 
Theorem~\ref{thm:Lim-safe} due to Lim \cite[Theorem~4.6]{MR4330938}.
\end{remark}

\section{An algorithm for the $p$-saturation of units}\label{sec:algs-p-sat-units}

The following algorithm incorporates a modification of the algorithm of
\cite[\S 6]{zbMATH07854947}, as well as \cite[Algorithm~4.9]{MR4440537}, 
which itself builds on the results of \cite{MR498486, MR1273458,MR3207410}.
It can be generalised to deal with the case of $S$-units thanks to the results
of \S \ref{sec:maps-on-S-units-p-rat}, but we choose not to give the details here
for ease of exposition.

\begin{algorithm}\label{alg:check-saturated-or-extract-roots}
Let $K$ be a number field and let $U$ be a subgroup of $\mathcal{O}_{K}^{\times}$
of finite index containing $\mu(K)$.
Let $p$ be a prime number.
\begin{itemize}
    \item Let $\overline{\psi_{K,p}\vert_{U}} : U/U^{p} \longrightarrow \mathcal{O}_{K}/p\mathcal{O}_{K}$ be the map induced by $\psi_{K,p}\vert_{U}$,
    \item Let $\overline{\eta_{K,p}\vert_{U}} : U/U^{p} \longrightarrow \prod_{\mathfrak{p} \in S_{p}(K)} U_{\mathfrak{p}}/ U_{\mathfrak{p}}^{p}$ be the map induced by $\eta_{K,p}\vert_{U}$.
    \item For a prime ideal $\mathfrak{q}$ of $\mathcal{O}_{K}$ with $\mathfrak{q} \notin S_{p}(K)$, let $k_{\mathfrak{q}} = \mathcal{O}_{K}/\mathfrak{q}$ be the residue field, and 
    let $\overline{\chi_{\mathfrak{q}} \vert_{U}} : U/U^{p} \longrightarrow k_{\mathfrak{q}}^{\times} / k_{\mathfrak{q}}^{\times p}$ be the canonical map.
\end{itemize}
Given $\mathcal{O}_{K}$, $U$, $d_{K}$ and $p$, the following algorithm either returns that 
$U$ is $p$-saturated in $\mathcal{O}_{K}^{\times}$ or returns 
an element $\alpha \in \mathcal{O}_{K}^{\times}$ such that 
$[\langle U, \alpha \rangle : U] = p$.
\renewcommand{\theenumi}{\arabic{enumi}}
\renewcommand{\labelenumi}{(\theenumi)}
\begin{enumerate}
    \item Choose parameters $c,t \in \Z_{\geq 1}$ with $t \geq c$ 
    (see Remark~\ref{rmk:parameters}).
    \item If $p \nmid 2d_{K}$ compute $M:=\ker(\overline{\psi_{K,p}\vert_{U}})$;
    otherwise compute $M:=\ker(\overline{\eta_{K,p}\vert_{U}})$.
    \item If $\dim_{\F_{p}}(M)=0$ return that $U$ is $p$-saturated in $\mathcal{O}_{K}^{\times}$.
    \item \textup{[Optional]} Choose any non-trivial element $\overline{\beta} \in M$,
    choose a representative $\beta \in U$ and test whether $\beta$ is a $p$-th power
    in $\mathcal{O}_{K}^{\times}$. If there exists $\alpha$ with $\alpha^{p} = \beta$, then return $\alpha$.
    \item Compute
    $T_{c} := \{ \text{prime ideals }\mathfrak{q} \text{ of }\mathcal{O}_{K} \text{ with }
    \mathfrak{q} \notin S_{p}(K) \text{ and }
    \Norm(\mathfrak{q}) \leq c\}$. 
    \item Compute $N := M \cap \bigcap_{\mathfrak{q} \in T_{c}} 
    \ker(\overline{\chi_{\mathfrak{q}}\vert_{U}})$.
    \item If $\dim_{\F_{p}}(N)=0$ return that $U$ is $p$-saturated in $\mathcal{O}_{K}^{\times}$.
    \item If $c < t$ then replace $c$ by $2c$ and go to Step \textup{(5)}.
    \item Choose any non-trivial element $\overline{\beta} \in N$,
    choose a representative $\beta \in U$ and test whether $\beta$ is a $p$-th power
    in $\mathcal{O}_{K}^{\times}$. If there exists $\alpha$ with $\alpha^{p} = \beta$, then return $\alpha$.
    \item Replace $c$ by $2c$ and go to Step \textup{(5)}.
\end{enumerate}
\end{algorithm}

\begin{proof}[Proof of correctness]
The correctness of the outputs in Steps (3) and (7) follows from Proposition~\ref{prop:psi-eta-give-p-sat} and Corollary~\ref{cor:p-sat}(c), respectively.
The correctness of the outputs in Steps (4) and (9) follows from Corollary~\ref{cor:p-sat}(e).
The algorithm must terminate by \cite[Proposition~4.8]{MR4440537}; 
this is ultimately a consequence of the Grunwald--Wang theorem (see \cite[Theorem~4.4]{MR4440537}, for example).
\end{proof}

\begin{proof}[Further details on each step]
The computation of $M$ in Step~(2) is described in Corollary~\ref{cor:alg-schiro-kernel} and \S\ref{subsec:compute-eta}.
For Steps~(4) and~(9), determining whether an element of $K$ is a $p$-th power, and computing a $p$-th root when one exists, reduces to root computations or polynomial factorisation over number fields, for which efficient algorithms exist; see~\cite{MR774816, MR2168610, MR2537701, zbMATH07854947}.
The remaining steps involve computations with residue fields, for which the requisite algorithms are described in~\cite{MR1228206}.
\end{proof}

\begin{remark}\label{rmk:parameters}
The parameters of Algorithm~\ref{alg:check-saturated-or-extract-roots} can 
be altered depending on the use case. 
If it is highly likely that $U$ is $p$-saturated in $\mathcal{O}_{K}^{\times}$
(for example, if a GRH-conditional computation predicts that $U=\mathcal{O}_{K}^{\times}$)
then Step (4) should be omitted and $t$ should be chosen to be `large'.
This is because the computation of $p$-th roots is much more expensive
than the other steps. 
If it is plausible that $U$ is not $p$-saturated in $\mathcal{O}_{K}^{\times}$,
then Step (4) should be executed and $t$ should be chosen to be `small'.
\end{remark}

\begin{remark}\label{rem:satalg-complexity}
If $K$ is quasi-$p$-rational and 
$U$ is $p$-saturated in $\mathcal{O}_{K}^{\times}$, 
then Algorithm~\ref{alg:check-saturated-or-extract-roots} terminates at Step (3)
by Proposition~\ref{prop:quasi-p-rat-subgroup-of-units-eta} and Corollary~\ref{cor:quasi-p-rat-subgroup-of-units-schirokauer}.
Moreover, if $K$ is quasi-$p$-rational, but 
$U$ is not $p$-saturated in $\mathcal{O}_{K}^{\times}$, then Proposition~\ref{prop:ker-subgroup-of-units} guarantees that any representative in $U$ of any non-trivial element of $M$
must be a $p$-th power in $\mathcal{O}_{K}^{\times}$, and so 
Algorithm~\ref{alg:check-saturated-or-extract-roots} terminates at Step~(4), assuming
this optional step is performed.
Furthermore, under the assumption of Gras's Conjecture~\ref{conj:Gras}, $K$ is quasi-$p$-rational for all but finitely many primes $p$
(see Remark~\ref{rmk:quasi-p-rat-Gras-conj}). 
These observations constitute the key new contribution of this section, 
and they explain the finding in 
\cite[\S 6.2.3]{zbMATH07854947} that using the Schirokauer map $\psi_{K,p}$
alone is often enough to detect $p$-th powers in practice.
\end{remark}

\begin{remark}\label{rem:satalg-complexity-worst-case}
The worst-case runtime of Algorithm~\ref{alg:check-saturated-or-extract-roots} is the 
same as that of \cite[Algorithm~4.9]{MR4440537}.
Assuming GRH, this can be shown to be in $O(\poly(d, \log\lvert d_K\rvert, l(U), p))$.
Now suppose that we are in the important case that 
$K$ is quasi-$p$-rational, $U$ is $p$-saturated in $\mathcal{O}_{K}^{\times}$
and $p \nmid 2d_{K}$, and do \emph{not} assume GRH.
Then
Corollary~\ref{cor:alg-schiro-kernel} shows that
the runtime of Algorithm~\ref{alg:check-saturated-or-extract-roots}
is in $O(\poly(\log(p)))$ (for fixed $K$ and $U$).
By contrast, \cite[Algorithm~4.9]{MR4440537} will,
via the computation of $\ker(\overline{\chi_{\mathfrak{q}}\vert_{U}})$,
require solving discrete logarithms in a group of order $p$, for which generic algorithms have runtime in  $O(\sqrt{p})$ (see \cite{MR1745660} for an overview). This latter analysis
does not even take into account the number of prime ideals $\mathfrak{q}$ for which 
$\ker(\overline{\chi_{\mathfrak{q}}\vert_{U}})$ will need to be computed.
\end{remark}

\section{Unconditional unit group computations}\label{sec:unconditional-unit-group-comp}

\subsection{An algorithm for unconditional verification of a unit group}
We now apply the results of \S \ref{sec:algs-p-sat-units} to the 
unconditional computation of unit groups of number fields. 
Let $K$ be a number field with ring of integers $\mathcal{O}_{K}$ whose unit group should be determined.
A common strategy for computing a system of fundamental units, both in early algorithms of Pohst--Zassenhaus~\cite{MR498486} and modern variants of Buchmann's algorithm~\cite{MR1104698}, is to determine a tentative unit group 
$U \leq \mathcal{O}_{K}^{\times}$ of finite index as a first step.
Assuming additional hypotheses, based on ad hoc arguments or the truth of standard number theoretic conjectures such as GRH,
one can often conclude that $U$ is indeed the full unit group.
In order to remove these hypotheses and obtain an unconditional result, the main task then becomes the verification that $U = \mathcal{O}_{K}^{\times}$, 
which is equivalent to verifying that $U$ is 
$p$-saturated in $\mathcal{O}_{K}^{\times}$ for all prime numbers $p$.
Using a lower bound for the regulator $\mathrm{Reg}_{K}$, 
computed using~\cite{MR604833} or~\cite{MR2441075}, for example,
one can compute an explicit bound $B$ such that the equality 
$U = \mathcal{O}_{K}^{\times}$ is equivalent to the 
assertion that $U$ is $p$-saturated in $\mathcal{O}_{K}^{\times}$
for all prime numbers $p \leq B$. This describes the idea behind the following algorithm.

\begin{algorithm}\label{alg:unit-group-verification}
Let $K$ be a number field and let $U$ be a subgroup of
$\mathcal{O}_{K}^{\times}$ of finite index. 
Given $\mathcal{O}_{K}$, $U$, and a lower bound $b>0$ for the regulator
$\mathrm{Reg}_{K}$, the following algorithm returns 
\ensuremath{\mathsf{true}} if $U=\mathcal{O}_{K}^{\times}$,
and returns \ensuremath{\mathsf{false}} otherwise. 
\renewcommand{\theenumi}{\arabic{enumi}}
\renewcommand{\labelenumi}{(\theenumi)}
\begin{enumerate}
\item Determine whether $\mu(K) \subset U$; if not return \ensuremath{\mathsf{false}}.
\item Compute $B := \lfloor \mathrm{Reg}(U)/b \rfloor$.
\item For each prime number $p \leq B$, 
use Algorithm~\ref{alg:check-saturated-or-extract-roots}
to test whether $p$ divides $[\mathcal{O}_{K}^{\times} : U]$, and if so, 
return \ensuremath{\mathsf{false}}.
\item Return \ensuremath{\mathsf{true}}. 
\end{enumerate}
\end{algorithm}

\begin{proof}[Proof of correctness]
The correctness of the outputs in Steps (1) and (3) is clear.
It remains to verify the correctness of the output of Step (4).
Let $\pi : \mathcal{O}_{K}^{\times} \longrightarrow \mathcal{O}_{K}^{\times}/\mu(K)$
be the canonical map. 
Then
\[
[\mathcal{O}_{K}^{\times} : U] 
= [\pi(\mathcal{O}_{K}^{\times}) : \pi(U)]
= \mathrm{Reg}(U) / \mathrm{Reg}_{K}
\leq B,
\]
where the first equality holds since $\mu(K) \subset U$ as verified in Step (1),
the second equality follows from \cite[Lemma~4.15]{MR1421575}, and the inequality
follows from the definition of $B$ and the assumption that $0<b\leq \mathrm{Reg}_{K}$.
Hence $p \nmid [\mathcal{O}_{K}^{\times}:U]$ for every prime number $p > B$.
But the completion of Step (3) without returning \ensuremath{\mathsf{false}}
ensures that $p \nmid [\mathcal{O}_{K}^{\times}:U]$ for every prime number $p \leq B$.
Therefore $\mathcal{O}_{K}^{\times}=U$.
\end{proof}

\begin{proof}[Further details on each step]
For Step (1), first compute the order of $\mu(K)$ using \cite[Algorithm~4.9.9]{MR1228206}
and then use the algorithm of~\cite{ge}; see also~\cite[Theorem~1.11]{MR3749417}.
Step (2) can be performed by using the definitions. 
\end{proof}

\begin{remark}\label{rmk:previous-unconditional-unit-verification}
Previous algorithms for unconditional unit group verification, as described in 
\cite[\S 4]{MR1273458} or
\cite[\S 6]{10.1145/3815436.3815477}, for example, follow the structure of 
Algorithm~\ref{alg:unit-group-verification}, but carry out Step (3) using a 
version of \cite[Algorithm~4.9]{MR4440537}.
The latter is essentially Algorithm~\ref{alg:check-saturated-or-extract-roots} without
Steps (2)--(4).
\end{remark}

\begin{remark}
In the case that $K/\Q$ is Galois, the computation of the image of $U$
under the Schirokauer map $\psi_{K,p}$ for multiple primes $p$ can be sped up as outlined
in \S \ref{subsec:Schirokauer-improvement-Galois}. 
This results in a significant speed-up of
Algorithm~\ref{alg:check-saturated-or-extract-roots}
and thus in Step (3) of Algorithm~\ref{alg:unit-group-verification}.    
\end{remark}

\begin{remark}
Recently, Yee \cite{yee2025unconditional} has given a new 
approach to unconditional unit group verification, part of 
which involves checking that $p \nmid [\mathcal{O}_{K}^{\times} : U]$ for all prime numbers $p$
up to a certain bound; the method for this part is similar to
\cite[Algorithm~4.9]{MR4440537} and thus can again be improved significantly by using Algorithm~\ref{alg:check-saturated-or-extract-roots}.
\end{remark}

\begin{theorem}
Let $K$ be a number field and let $U$ be a subgroup of $\mathcal{O}_{K}^{\times}$
of finite index. 
Let $B$ be as in Algorithm~\ref{alg:unit-group-verification} and 
let $p_{1},\dotsc,p_{k}$ be the prime numbers $\leq B$ such that either $p_{i} \mid 2d_{K}$ or $K$ is not quasi-$p_{i}$-rational. 
Then assuming GRH and that in fact $U=\mathcal{O}_{K}^{\times}$,
the runtime of Step (3) of Algorithm~\ref{alg:unit-group-verification} is in
  \[ 
  O\left(B\poly(d, \log\lvert d_K\rvert, l(U), \log(B)) + \sum_{i=1}^{k} \poly(d, \log\lvert d_K\rvert, l(U), p_i )\right).
  \]
  In particular, if $p_i \leq \log(B)$ for $1 \leq i \leq k$, then assuming GRH 
  and that in fact $U=\mathcal{O}_{K}^{\times}$, the runtime is in 
  \[ 
  O(B\poly(d, \log\lvert d_K\rvert, l(U), \log(B))).
  \]
\end{theorem}

\begin{proof}
Since $U=\mathcal{O}_{K}^{\times}$, we have that $U$ is $p$-saturated in $\mathcal{O}_{K}^{\times}$ for every prime number $p$.
In particular, for the prime numbers $p \leq B$ for which $K$ is quasi-$p$-rational, Algorithm~\ref{alg:check-saturated-or-extract-roots} terminates at Step (3) 
by Remark~\ref{rem:satalg-complexity},
and its runtime is in
$O(\poly(d, \log\lvert  d_K\rvert, l(U), \log(B)))$ by Corollary~\ref{cor:alg-schiro-kernel}.
For each of the exceptional primes $p_{i}$, Remark~\ref{rem:satalg-complexity-worst-case}
gives a runtime in $O(\poly(d, \log |d_{K}|, l(U), p_{i}))$.
Summing these contributions gives the desired result. 
\end{proof}

\subsection{Examples}\label{subsec:example:uncond}
We now give some explicit examples to illustrate the improvements 
of Algorithm~\ref{alg:unit-group-verification} 
over previous methods, which are described in Remark~\ref{rmk:previous-unconditional-unit-verification}.
All of these algorithms have been implemented in
\textsc{Hecke}~\cite{MR3703682} (also available through~\textsc{Oscar}~\cite{OSCAR-book}).
We consider the number fields defined by the polynomials in Table~\ref{tab:polys-for-fields} in Appendix~\ref{appendix:number-fields}\footnote{The fields are also available at \url{github.com/thofma/p-rationality/}.}.
Each of these number fields is a Galois extension of $\Q$.

In Table~\ref{tab:unit-stats}, we present numerical invariants of the fields 
in Table~\ref{tab:polys-for-fields}
and runtimes of the different algorithms. More precisely, for a number field $K$ we record:
\begin{description}
\item[$G$ ] the isomorphism type of the Galois group of $K/\Q$, where $C_{n}$ denotes the cyclic group of order $n$ and $Q_{4n}$ denotes the generalised quaternion group of order $4n$.
\item[$\lceil \mathrm{Reg}(U) \rceil$ ] the ceiling of the regulator of a tentative unit group $U$, where $U=\mathcal{O}_{K}^{\times}$ if GRH holds.
\item[$b$ ] a lower bound for $\mathrm{Reg}_K$ from the algorithm of 
Fieker--Pohst~\cite{MR2441075}, as implemented in Magma~\cite{MR1484478},
\item[$B$ ] the quantity $\lfloor \mathrm{Reg}(U)/b\rfloor$,
\item[$t_{1}$ ] the runtime of an earlier algorithm described in 
Remark~\ref{rmk:previous-unconditional-unit-verification},
\item[$t_{2}$ ] the runtime of Algorithm~\ref{alg:unit-group-verification}.
\end{description}

All timings are for computations performed using 
a single core of  a server with an Intel Xeon 6226R 2.90GHz CPU and 384GB RAM.
The software used was \textsc{Julia} version 1.12.6 and \textsc{Hecke} version 0.39.21.

\begin{table}[ht!]
\caption{Runtimes for unconditional unit-group verification.}
\label{tab:unit-stats}
\begin{tabular}{crrrrccc}
\toprule %
Field & $G$ & $\lceil \mathrm{Reg}(U) \rceil$ & $b$ & $B$ & $t_{1}$ & $t_{2}$ & $t_{1}/t_{2}$ \\
\midrule
1 & $C_{5}$ & $46436465006$  & $1373.3$  & $33812575$  & 1h\,6m\,4s & 2m\,49s & $23.4$ \\
2 & $C_{5}$ & $572711843539$  & $3597.2$  & $159208684$  & 8h\,54m\,2s & 12m\,53s & $41.4$ \\
3 & $C_{5}$ & $361611954907$  & $1663.5$  & $217377582$  & 13h\,15m\,18s & 18m\,27s & $43.1$ \\
\midrule
4 & $C_{6}$ & $15224167250$  & $11578.7$  & $1314838$  & 1m\,44s & 9s & $11.5$ \\
5 & $C_{6}$ & $87928183170$  & $14958.4$  & $5878197$  & 9m\,53s & 35s & $16.9$ \\
6 & $C_{6}$ & $113173213389$  & $11110.1$  & $10186526$  & 19m\,23s & 1m\,1s & $19.0$ \\
\midrule
7 & $Q_8$ & $96677308322$  & $19126.4$  & $5054651$  & 12m\,56s & 37s & $20.9$ \\
8 & $Q_8$ & $614071514334$  & $23485.6$  & $26146692$  & 1h\,38m\,3s & 2m\,59s & $32.8$ \\
9 & $Q_8$ & $1278699105843$  & $23779.3$  & $53773668$  & 4h\,10m\,55s & 6m\,2s & $41.5$ \\
\midrule
10 & $C_{11}$ & $155801246971$  & $69693.5$  & $2235520$  & 8m\,6s & 57s & $8.5$ \\
11 & $C_{11}$ & $2083946923677$  & $254423.4$  & $8190861$  & 42m\,14s & 3m\,31s & $12.0$ \\
12 & $C_{11}$ & $2276941413144$  & $41736.4$  & $54555262$  & 7h\,17m\,13s & 22m\,12s & $19.6$ \\
\midrule
13 & $Q_{12}$ & $82138349503$  & $135762.7$  & $605015$  & 2m\,17s & 11s & $12.4$ \\
14 & $Q_{12}$ & $7197258073734$  & $153764.8$  & $46806942$  & 7h\,7m\,10s & 10m\,40s & $40.0$ \\
15 & $Q_{12}$ & $8117033381733$  & $77492.3$  & $104746309$  & 21h\,49m\,12s & 22m\,32s & $58.1$ \\
\midrule
16 & $C_{13}$ & $2733056591$  & $15421.3$  & $177226$  & 42s & 7s & $6.0$ \\
17 & $C_{13}$ & $1689788670897$  & $178321.3$  & $9476090$  & 1h\,6m\,13s & 5m\,37s & $11.7$ \\
18 & $C_{13}$ & $927977436617$  & $54513.1$  & $17023018$  & 2h\,13m\,25s & 10m\,12s & $13.0$ \\
\midrule
19 & $Q_{16}$ & $1725366588$  & $37136.2$  & $46461$  & 17s & 3s & $5.6$ \\
20 & $Q_{16}$ & $11079174313222$  & $2167892.5$  & $5110574$  & 44m\,34s & 2m\,17s & $19.5$ \\
21 & $Q_{16}$ & $12626699814824$  & $380019.7$  & $33226439$  & 7h\,22m\,9s & 11m\,46s & $37.5$ \\
\midrule
22 & $C_{17}$ & $30768108228$  & $26599.1$  & $1156737$  & 8m\,23s & 1m\,18s & $6.4$ \\
23 & $C_{17}$ & $559957546560$  & $66487.7$  & $8421970$  & 1h\,32m\,59s & 9m\,32s & $9.7$ \\
24 & $C_{17}$ & $24055588816326$  & $333320.3$  & $72169582$  & 27h\,58m\,30s & 1h\,18m\,28s & $21.3$ \\
\midrule
25 & $Q_{20}$ & $547793455064$  & $1247669.3$  & $439054$  & 4m\,9s & 19s & $13.1$ \\
26 & $Q_{20}$ & $17715817733239$  & $2165711.5$  & $8180138$  & 2h\,2m & 5m\,9s & $23.6$ \\
27 & $Q_{20}$ & $10186030781008$  & $843731.4$  & $12072600$  & 3h\,6m\,15s & 7m\,39s & $24.3$ \\
\bottomrule 
\end{tabular}
\end{table}

\begin{remark}
For specific Galois groups, Brauer relations can be used to determine the regulator of $K$ 
and thus verify the unit group unconditionally, using unit groups of proper 
subfields that have been computed unconditionally 
(see~\S\ref{subsec:unconditional-invs}). This approach does not apply to our 
examples, which are either cyclic and hence have no non-trivial Brauer relations, 
or lack so-called useful Brauer relations (see \cite[Theorem~2.9]{MR4440537}).
\end{remark}

\subsection{$S$-unit group computations}\label{subsec:extend-to-S-units}
The (unconditional) computation of unit groups also has applications to determining 
$S$-unit groups.
For a set of prime ideals $S$ of $\mathcal{O}_{K}$, 
the computation of a minimal generating set of the $S$-unit group
$\mathcal{O}_{K, S}^{\times}$ can be split up into the computation of a basis of $\mathcal{O}_{K, S}^{\times}/\mathcal{O}_{K}^{\times}$ 
and information regarding the class group $\Cl(K)$; 
see \cite[Algorithm~6.1]{MR1898758}, for example.
Hence any improvements in unit group computations directly benefit $S$-unit group computations.

\section{(Quasi)-$p$-rationality algorithms for arbitrary number fields}
\label{sec:quasi-p-rat-algs}

Let $K$ be a number field and let $p$ be a prime number.
Pitoun and Varescon \cite{MR3266966} have given an algorithm that
checks whether $K$ satisfies 
Leopoldt's conjecture at $p$ and then computes $T_{S_{p}}(K)$ (see 
\S \ref{subsec:p-rational-fields}). In particular, this can be used
to verify whether $K$ is $p$-rational.
Gras \cite{MR4491965} has given an algorithm to test $K$ for $p$-rationality 
based on Theorem~\ref{thm:Gras-ray-p-rat} along with an implementation in 
PARI/GP \cite{PARI2}.
Both of these approaches require computation of certain
ray class groups $\Cl_{K}(p^{k})$, or at least their Sylow-$p$ subgroups, and 
hence in particular require information about the Sylow-$p$ subgroup of the 
class group $\Cl_{K}$.
Barbulescu and Ray \cite{MR4158582} have proposed improvements to 
the method of Pitoun and Varescon for certain families of number fields.

The following algorithm determines whether $K$ is quasi-$p$-rational. 
In contrast to the aforementioned algorithms that verify whether $K$ is $p$-rational,
this does not require any computation of class groups or ray class groups.
Moreover, it does not require 
the full group of units $\mathcal{O}_{K}^{\times}$ as an input.
By Lemma~\ref{lem:rat-implies-quasi-rat-implies-leopoldt}(b),
it can also be used to verify that Leopoldt's conjecture holds for $K$ at $p$
in many cases. Indeed, assuming Gras's Conjecture~\ref{conj:Gras}, 
for fixed $K$ the algorithm verifies Leopoldt's conjecture
except for at most finitely many primes $p$ (see Remark~\ref{rmk:quasi-p-rat-Gras-conj}).
Other algorithms for verifying Leopoldt's conjecture include that 
of Buchmann and Sands \cite{MR904010}, 
the aforementioned algorithm of Pitoun and Varescon \cite[\S 3]{MR3266966}, 
and the algorithm of Nickel and the present authors \cite[Appendix~A.1]{MR4111943}.

\begin{algorithm}\label{alg:quasi-p-rat}
Let $K$ be a number field and let $U$ be a subgroup of
$\mathcal{O}_{K}^{\times}$ of finite index containing $\mu(K)$.
Given $\mathcal{O}_{K}$, $U$, $d_{K}$
and a prime number $p$, the following algorithm returns 
\ensuremath{\mathsf{true}} if $K$ is quasi-$p$-rational, 
and returns \ensuremath{\mathsf{false}} otherwise. 
\renewcommand{\labelenumi}{(\arabic{enumi})}
\begin{enumerate}
\item If $p \nmid 2d_{K}$ then 
\begin{enumerate}
    \item If $\dim_{\mathbb{F}_{p}}(\psi_{K,p}(U)) = \rank_{\Z}(\mathcal{O}_{K}^{\times})$ then return \ensuremath{\mathsf{true}}.
    \item Compute a $p$-saturated subgroup $V$ of $\mathcal{O}_{K}^{\times}$ 
such that $U \leq V \leq \mathcal{O}_{K}^{\times}$.
    \item If $\dim_{\mathbb{F}_{p}}(\psi_{K,p}(V)) = \rank_{\Z}(\mathcal{O}_{K}^{\times})$ then return \ensuremath{\mathsf{true}};
    otherwise return \ensuremath{\mathsf{false}}.
\end{enumerate}
\item If $p=2$ then 
\begin{enumerate}
    \item If $\lvert S_{2}(K) \rvert >1$ then return \ensuremath{\mathsf{false}}.
    \item If
    $\dim_{\mathbb{F}_{2}}(\eta_{K,2}(U)) = \rank_{\Z}(\mathcal{O}_{K}^{\times})+1$ then return \ensuremath{\mathsf{true}}.
    \item Compute a $2$-saturated subgroup $V$ of $\mathcal{O}_{K}^{\times}$ 
such that $U \leq V \leq \mathcal{O}_{K}^{\times}$.
    \item If $\dim_{\mathbb{F}_{2}}(\eta_{K,2}(V)) = \rank_{\Z}(\mathcal{O}_{K}^{\times})+1$ 
    then return \ensuremath{\mathsf{true}}; otherwise return \ensuremath{\mathsf{false}}.
\end{enumerate}
\item If $p>2$ and $p \mid d_{K}$ then 
\begin{enumerate}
    \item If $p \leq [K:\Q]+1$ then 
    \begin{enumerate}
    \item If $\mu_{p} \subset K$ and $\lvert S_{p}(K) \rvert >1$
    then return \ensuremath{\mathsf{false}}.
    \item If $\mu_{p} \not \subset K$ and there exists some $\mathfrak{p} \in S_{p}(K)$ 
    such that $(p-1) \mid e_{\mathfrak{p}}$ and $\mathfrak{p}$ splits completely in 
    $K(\zeta_{p})/K$ then return \ensuremath{\mathsf{false}}.
    \end{enumerate} 
    \item If $\dim_{\mathbb{F}_{p}}(\eta_{K,p}(U)) = \rank_{p}(\mathcal{O}_{K}^{\times})$ then return \ensuremath{\mathsf{true}}.
    \item Compute a $p$-saturated subgroup $V$ of $\mathcal{O}_{K}^{\times}$ 
such that $U \leq V \leq \mathcal{O}_{K}^{\times}$.
    \item If $\dim_{\mathbb{F}_{p}}(\eta_{K,p}(V)) = \rank_{p}(\mathcal{O}_{K}^{\times})$ then return \ensuremath{\mathsf{true}};
    otherwise return \ensuremath{\mathsf{false}}.
\end{enumerate}
\end{enumerate}
\end{algorithm}

\begin{proof}[Proof of correctness]
In Step (1), the correctness of the outputs follows from 
Proposition~\ref{prop:psi-eta-give-p-sat}(a) and Corollary~\ref{cor:quasi-p-rat-subgroup-of-units-schirokauer}.
The correctness of the outputs in Steps (2) and (3) 
follows from Propositions~\ref{prop:psi-eta-give-p-sat}(b) and \ref{prop:quasi-p-rat-subgroup-of-units-eta}, and Remark~\ref{rmk:roots-of-unity-condition}.
\end{proof}

\begin{proof}[Further details on each step]
In Step (1), $\psi_{K,p}(U)$ and $\psi_{K,p}(V)$ can be computed 
as described in \S \ref{subsec:compute-psi}.
In Steps (2) and (3), $\eta_{K,p}(U)$ and $\eta_{K,p}(V)$ can be computed 
as described in \S \ref{subsec:compute-eta}.  
Computation of the subgroup $V$ as required in Steps (1)(b), (2)(c) and (3)(c)
can be performed by iterated calls to Algorithm~\ref{alg:check-saturated-or-extract-roots}.
Computation of prime decomposition as required in Steps (2)(a) and (3)(a)
can be performed using \cite[Algorithm~6.2.9]{MR1228206}.
In Step (3)(a), $\mu_{p} \subset K$ if and only if $p$ divides $|\mu(K)|$,
which can be computed using \cite[Algorithm~4.9.9]{MR1228206}.
\end{proof}

\begin{remark}\label{rmk:skip-steps-when-unit-group-known}
If the full unit group $\mathcal{O}_{K}^{\times}$ is supplied as input then one may 
skip Steps (1)(a)(b), (2)(b)(c) and (3)(b)(c) and take $V=\mathcal{O}_{K}^{\times}$ in Steps (1)(c), (2)(d) and (3)(d).
\end{remark}

\begin{remark}\label{rmk:Galois-Shirokauer-many-primes}
If $K/\Q$ is Galois and one wishes to determine whether $K$ is quasi-$p$-rational 
for many primes $p$, then Step (1)(a) can be sped up as described in \S \ref{subsec:Schirokauer-improvement-Galois}.
\end{remark}

\begin{remark}
In principle, Step (3)(a)(ii) can be implemented using \cite[Algorithm~2]{MR5022766},
which decides whether $\mu_{p} \subset K_{\mathfrak{p}}$ holds for a given prime ideal $\mathfrak{p} \in S_p(K)$ (see Remark~\ref{rmk:roots-of-unity-condition}).
\end{remark}

The following algorithm determines whether $K$ is $p$-rational. 
In a case covered in the last two steps, it reverts to the 
aforementioned algorithm of Gras \cite{MR4491965}; 
when $K$ is totally real this can only occur when $p$ divides $h_{K}$
and some prime of $K$ above $p$ is wildly ramified in $K/\Q$.
Aside from this case, it is generally significantly faster than the 
method of Gras, especially when the field is fixed and one 
wishes to test $p$-rationality for many primes.
Note that a much more efficient algorithm for real cyclotomic fields
is given in \S \ref{sec:alg-real-cyclo}; this can be extended
to cyclotomic fields in many cases (see Remark~\ref{rmk:p-rat-full-cyclo}).

\begin{algorithm}\label{alg:p-rat}
Let $K$ be a number field and let $U$ be a subgroup of
$\mathcal{O}_{K}^{\times}$ of finite index containing $\mu(K)$.
Given $\mathcal{O}_{K}$, $U$, $h_{K}$, $d_{K}$
and a prime number $p$, the following algorithm returns 
\ensuremath{\mathsf{true}} if $K$ is $p$-rational, 
and returns \ensuremath{\mathsf{false}} otherwise. 
\renewcommand{\labelenumi}{(\arabic{enumi})}
\begin{enumerate} 
\item Determine whether $K$ is quasi-$p$-rational; if not, return \ensuremath{\mathsf{false}}.
\item If $p \nmid h_{K}$ then return \ensuremath{\mathsf{true}}.
\item If $c_{K}=0$ and every prime in $S_{p}(K)$ is at most tamely ramified in $K/\Q$
then return \ensuremath{\mathsf{false}}.
\item Compute $\rank_{p}(\Cl_{K}(p^{k}))$, where $k=3$ if $p=2$ and $k=2$ if $p>2$.
\item If $\rank_{p}(\Cl_{K}(p^{k})) = c_{K}+1$ then return \ensuremath{\mathsf{true}}; else return \ensuremath{\mathsf{false}}.
\end{enumerate}
\end{algorithm}

\begin{proof}[Proof of correctness]
The correctness of the outputs in Steps (1) and (2) follows from 
Lemma~\ref{lem:rat-implies-quasi-rat-implies-leopoldt}. 
In Step (3), it has already been established that $p \mid h_{K}$, and
so the correctness of the output follows from Corollary~\ref{cor:tot-real-p-rat-tame}. The correctness of the output of Step (5) follows from Theorem~\ref{thm:Gras-ray-p-rat}.
\end{proof}

\begin{proof}[Further details on each step]
For Step (1), use Algorithm~\ref{alg:quasi-p-rat}.
In Step (3), the computation of $c_{K}$ can be performed using
\cite[Algorithm~4.1.11]{MR1228206} and the
computation of ramification indices can be performed
using \cite[Algorithm~6.2.9]{MR1228206}.
In Step (4), the computation of ray class groups can be performed using
\cite[Algorithm~4.3.1]{MR1728313}. 
In fact, as we only need to compute  $\rank_{p}(\Cl_{K}(p^{k}))$, it suffices to determine $\Cl_{K}(p^{k})/\Cl_{K}(p^{k})^{p}$, for which improved algorithms are described in~\cite[\S 3]{MR3952015}.
\end{proof}

\begin{remark}
In certain situations, the computational effort required to determine whether 
a number field $K$ is $p$-rational may be reduced by applying 
Algorithm~\ref{alg:p-rat} to certain proper subfields of $K$ rather than
to $K$ directly. 
\begin{enumerate}
    \item To show that $K$ is not $p$-rational, it suffices to show that there exists a proper subfield of $K$ that is not $p$-rational (see Lemma~\ref{lem:p-rat-subfields}).
    \item If $K$ is a CM-field then, under certain hypotheses on $p$, one can apply
    either Proposition~\ref{prop:CM-tot-real-p-rat} or Corollary~\ref{cor:CM-tot-real-p-rat} to reduce the problem to determining the $p$-rationality of the maximal totally real subfield $K^{+}$.
    \item If $K/\Q$ is Galois and $p$ does not divide $[K:\Q]$, then 
    it may be possible to deduce that $K$ is $p$-rational from the $p$-rationality of proper subfields (see \cite[\S 3.3]{MR3512524}). 
    \item It may be possible to use a `going-up' result as described in Remark~\ref{rmk:going-up}.
\end{enumerate}
\end{remark}

\section{Galois representations with open image}\label{sec:Galois-reps}

Greenberg \cite{MR3512524} proved a number of results in which he
constructed certain continuous representations 
$\Gal(\overline{\Q}/\Q) \longrightarrow \GL_{n}(\Z_{p})$
with open image from totally complex $p$-rational number fields.
Here we recall a selection of his results and show
that the $p$-rationality hypothesis can be verified using our algorithms.

\subsection{Galois representations from totally complex $A_{5}$ and $S_{5}$-extensions}
Let $A_{5}$ and $S_{5}$ denote 
the alternating and symmetric groups on five letters, respectively.

\begin{prop}\label{prop:Galois-reps-A5-S5}
Let $K/\Q$ be a finite Galois extension and let $G = \Gal(K/\Q)$.
Suppose that $K$ is totally complex, that $G \simeq A_{5}$ or $S_{5}$,
and that $K$ is $p$-rational for some $p \geq 7$.
\begin{enumerate}
    \item If $G \simeq A_{5}$ then let $n \in \{ 3,4,5 \}$, and if $n=3$, further suppose that $p$ splits in $\Q(\sqrt{5})$, that is, that $p \equiv \pm 1 \bmod{5}$.
    \item If $G \simeq S_{5}$ then let $n \in \{ 4,5,6 \}$, and if $n=4$, further 
    suppose that the unique quadratic subfield of $K$ is imaginary.
\end{enumerate}
Let $K_{S_{p}}(p)$ be the 
maximal pro-$p$-extension of $K$ unramified outside primes above $p$
and let $\omega$ be an irreducible representation of $G$ over $\Q_{p}$ of degree $n$. 
Then there exists a continuous representation
\[
\rho : \Gal(K_{S_{p}}(p)/\Q) \longrightarrow \mathrm{GL}_{n}(\Z_{p})
\]
with open image such that the residual representations $\overline{\rho}$ 
and $\overline{\omega}$ over $\F_{p}$ are isomorphic.
\end{prop}

\begin{proof}
This is an application of \cite[Proposition~7.1]{MR3512524}.
For cases (a) and (b) the details are given in \cite[7.4.2 and 7.4.1]{MR3512524}, respectively.
\end{proof}

Greenberg \cite[\S 1]{MR3512524} noted the following obstacle to applying this result:
\begin{quote}
    \textit{The main result is Proposition 7.1. 
    Unfortunately, part of the hypothesis is that some number field $K$
    of large degree is $p$-rational, 
    something which would be difficult to verify in practice.}
\end{quote}
We remedy this situation as follows.
Let $K/\Q$ be one of two totally complex Galois extensions, with
$\Gal(K/\Q) \simeq A_{5}$ in one case and $\Gal(K/\Q) \simeq S_{5}$ in the other.
Since the unconditional computation of $h_{K}$ and $\mathcal{O}_{K}^{\times}$
for such fields is impractical using classical algorithmic methods,
we begin by describing an alternative approach in \S \ref{subsec:unconditional-invs}.
We next apply Algorithm~\ref{alg:quasi-p-rat}
to determine the (quasi)-$p$-rationality of $K$
for all primes $7 \le p < 10^{8}$, and then invoke
Proposition~\ref{prop:Galois-reps-A5-S5} to obtain the desired Galois representations.

\subsection{Unconditional computations using Brauer relations}\label{subsec:unconditional-invs}
Applications of Brauer relations to computational problems in number fields have been discussed
in~\cite{MR4440537}.
Note that in contrast to loc.\ cit., our main focus is on unconditional results, that is, results that do not depend on GRH.
Assume that $K/\Q$ is a finite Galois extension and that there exists a
Brauer relation of the form
\[ 
a \operatorname{Ind}_{1}^G(\mathbf{1}_1) = \sum_{1 \lneq H \leq G} a_H \operatorname{Ind}_H^G(\mathbf{1}_H), 
\]
where $a \in \Z_{>0}$, $a_H \in \Z$, and $G=\Gal(K/\Q)$. 
The Artin formalism for $L$-functions (see \cite[(10.4)~Proposition]{MR1697859}) and 
the behaviour of Dedekind zeta functions at $s=0$ 
(see~\cite[Chapitre~2, 2.2~Corollaire]{MR782485}) imply that
\begin{align}\label{eq:brauer}
\left(\frac{h_K R_K}{w_K}\right)^{a} 
= 
\prod_{1 \lneq H \leq G} \left(\frac{h_{K^H} R_{K^H}}{w_{K^H}}\right)^{a_H}, 
\end{align}
where $w_{K^H}$ denotes the number of roots of unity of $K^{H}$.
This reduces the computation of $h_{K}$ to that of $R_{K}$ as well as $h_{K^H}$ and $R_{K^H}$ for $1 \lneq H \leq G$.
This can be improved further by also taking the 
relations between the unit groups into account.
We summarise the approach in the following steps:
\begin{enumerate}
    \item 
    For each $H$ with $a_H \neq 0$ use classical algorithms (see~\cite[Section~6.5]{MR1228206}) to unconditionally determine $h_{K^H}$ and a system of fundamental units of $\mathcal{O}_{K^H}^\times$.
    For the verification of the unit group, use the improvements outlined
    in \S \ref{sec:unconditional-unit-group-comp}.
    \item 
    Consider the subgroup 
    $V = \langle \mathcal{O}_{K^{H^g}}^\times \mid H \leq G, g \in G, a_H \neq 0 \rangle$ of $\mathcal{O}_{K}^{\times}$.
    Determine a supergroup $V \leq U \leq \mathcal{O}_{K}^{\times}$ 
    such that $U$ is $p$-saturated in $\mathcal{O}_{K}^{\times}$ 
    for all $p$ dividing $\lvert G \rvert$. 
    Then $U = \mathcal{O}_{K}^{\times}$ by~\cite[Corollary~3.4]{MR4440537} and from this we can determine the regulator $R_{K}$.
    \item
    Use \eqref{eq:brauer} to determine $h_{K}$.
\end{enumerate}

\subsection{A totally complex $A_5$-extension}\label{subsec:A5}
We consider the polynomial 
\[
x^5 - x^4 + 2x^2 - 2x + 2 \in \Q[x]    
\]
from~\cite{MR1901356}, whose splitting field $K$ is an $A_{5}$-extension of $\Q$ with discriminant $2^{90} 17^{40}$.
First note that $A_{5}$ admits the Brauer relation  
\[ 
\Ind_{1}^{A_{5}}(\mathbf{1}_1) 
= 3 \Ind_{H_5}^{A_{5}}(\mathbf{1}_{H_5}) 
+ 2 \Ind_{H_{10}}^{A_{5}}(\mathbf{1}_{H_{10}})
+ 4 \Ind_{H_{12}}^{A_{5}}(\mathbf{1}_{H_{12}}) 
- 8 \Ind_{A_{5}}^{A_{5}}(\mathbf{1}_{A_{5}}), 
\]
where each $H_i$ is a certain subgroup of $A_{5}$ of order $i$ (see Table~\ref{tab:subfieldsa5}).
We fix an isomorphism $\Gal(K/\Q) \simeq A_{5}$ and denote by $K_{i} = K^{H_{i}}$ the subfield of $K$ fixed by $H_{i}$ under this identification.
For each subfield $K_{i}$, we determine unconditionally the class number, a set of fundamental units and the regulator.
The results are presented in Table~\ref{tab:subfieldsa5}.

\begin{table}[h!]
\caption{Arithmetic invariants of the fixed subfields of $K$.}
\label{tab:subfieldsa5}
\begin{tabular}{cccccrc}
\toprule
$i$ & $H_{i}\simeq$ & $[K_{i} : \Q]$ & $d_{K_{i}}$ & $h_{K_{i}}$ & $R_{K_{i}}$ & $w_{K_{i}}$ \\\midrule
12 & $A_{4}$          & $5$ & $2^6 \cdot 17^2 $ & $1$   &$3.0253...$  & $2$ \\\addlinespace[0.5em]
10 & $D_{10}$ & $6$ & $2^6 \cdot 17^4$ & $1$ & $15.6591...$ & $2$ \\\addlinespace[0.5em]
5 & $C_{5}$ & $12$ & $2^{18} \cdot 17^8$ & $5$   & $291.2664...$ & $2$ \\
\bottomrule
\end{tabular}
\end{table}
\noindent
We determine the subgroup $V \leq \mathcal{O}_{K}^{\times}$ as described above, 
saturate at the primes $2, 3, 5$ and obtain
\[ 
R_K = 63444609174857.6599... 
\]
Finally, we determine the class number of $K$ to be 
\[ 
h_K = \frac{2}{R_K} \left(\frac{h_{K_5}R_{K_5}}{2}\right)^3 \left(\frac{h_{K_{10}}R_{K_{10}}}{2}\right)^2\left(\frac{h_{K_{12}}R_{K_{12}}}{2}\right)^4 \cdot \left(\frac{1}{2}\right)^{-8} = 1. 
\]

\begin{theorem}\label{thm:a5}
Let $K$ be the splitting field of $x^5 - x^4 + 2x^2 - 2x + 2 \in \Q[x]$, an 
$A_{5}$-extension of $\Q$.
Then $K$ is $p$-rational for all $5 \leq p < 10^{8}$ and is not
quasi-$p$-rational for $p=2,3$.
\end{theorem}

\begin{proof}
Observe that for every prime number $p$, the field
$K$ is $p$-rational if and only if it is quasi-$p$-rational by
Lemma~\ref{lem:rat-implies-quasi-rat-implies-leopoldt}(c) and
the fact that $h_{K}=1$.
Hence the desired result follows from an application of Algorithm~\ref{alg:quasi-p-rat},
with the full unit group $\mathcal{O}_{K}^{\times}$ given as an input (see Remark~\ref{rmk:skip-steps-when-unit-group-known}).
Step (3) is only executed in the case $p=17$.
\end{proof}

\begin{corollary}
Let $K$ be as above and let $G=\Gal(K/\Q)$.
Let $7 \leq p < 10^{8}$ be a prime number. 
Let $n \in \{ 3,4,5 \}$, and if $n=3$, further suppose that $p \equiv \pm 1 \bmod 5$.
Let $\omega$ be an irreducible representation of $G$ over 
$\Q_{p}$ of degree $n$. 
Then there exists a continuous representation
\[
\rho : \Gal(K_{S_{p}}(p)/\Q) \longrightarrow \mathrm{GL}_{n}(\Z_{p})
\]
with open image such that the residual representations $\overline{\rho}$ 
and $\overline{\omega}$ over $\F_{p}$ are isomorphic.
\end{corollary}

\begin{proof}
This is the combination of 
Theorem~\ref{thm:a5} and Proposition~\ref{prop:Galois-reps-A5-S5}(a).
\end{proof}

\subsection{A totally complex $S_{5}$-extension}\label{subsec:S5}
We consider the splitting field $K$ of the polynomial 
\[
x^5 - x^3 - x^2 - x + 1 \in \Q[x],
\]
taken from~\cite{MR1901356}. The field $K$ is an $S_5$-extension with discriminant $(79 \cdot 89)^{60}$.
To determine the class number and fundamental units unconditionally, we make use of the following Brauer relation:
\begin{align*} 
\Ind_{1}^{S_{5}}(\mathbf{1}_1) 
= 6 \Ind_{H_6}^{S_{5}}(\mathbf{1}_{H_6}) 
&+ 4 \Ind_{H_{12}^{(1)}}^{S_{5}}(\mathbf{1}_{H_{12}^{(1)}}) 
- \Ind_{H_{12}^{(2)}}^{S_{5}}(\mathbf{1}_{H_{12}^{(2)}}) 
+ 5 \Ind_{H_{20}}^{S_{5}}(\mathbf{1}_{H_{20}}) \\ 
&- 11 \Ind_{H_{24}}^{S_{5}}(\mathbf{1}_{H_{24}})
- 3\Ind_{H_{60}}^{S_{5}}(\mathbf{1}_{A_5}) 
+ \Ind_{S_{5}}^{S_{5}}(\mathbf{1}_{S_{5}}), 
\end{align*}
where the $H_i$ and $H_i^{(j)}$, respectively, are certain subgroups of $S_{5}$ of order $i$ (see  Table~\ref{tab:subfieldss5}). 
We fix an isomorphism $\Gal(K/\Q) \simeq S_{5}$ and for $H \leq S_{5}$
denote by $K^{H}$ the subfield of $K$ fixed by $H$ under this identification.
The invariants of the fixed fields are given in Table~\ref{tab:subfieldss5}.

\begin{table}[h!]
\caption{Arithmetic invariants of the fixed subfields of $K$.}
\label{tab:subfieldss5}
\begin{tabular}{ccccrc}
\toprule
$H$ & $[K^H : \Q]$ & $d_{K^H}$ & $h_{K^H}$ & $R_{K^H}$ & $w_{K^H}$ \\\midrule
$H_{6} \simeq S_3$          & $20$ & $-79^7 \cdot 89^7$ & $1$   & $544287.7624...$      & $2$ \\\addlinespace[0.5em]
$H_{12}^{(1)} \simeq A_4$ & $10$ & $-79^5 \cdot 89^5$ & $2^{2} \cdot 3^{4}$ & $882.8378...$         & $2$ \\\addlinespace[0.5em]
$H_{12}^{(2)} \simeq D_{12}$ & $10$ & $-79^3 \cdot 89^3$ & $1$   & $89.1767...$          & $2$ \\\addlinespace[0.5em]
$H_{20} \simeq F_{5}$       & $6$  & $-79^3 \cdot 89^3$ & $2^{4}$  & $470.0499...$         & $2$ \\\addlinespace[0.5em]
$H_{24} \simeq S_{4}$       & $5$  & $-79^1 \cdot 89^1$ & $1$   & $0.8883...$           & $2$ \\\addlinespace[0.5em]
$H_{60} = A_{5}$       & $2$  & $-79^1 \cdot 89^1$ & $2^{2} \cdot 3^{3}$ & $1\phantom{.0000...}$ & $2$ \\
\bottomrule
\end{tabular}
\end{table}

Determining $V$ and saturating at $2$, $3$ and $5$ yields
\[ 
R_{K} = 5075956383097420432006255018096635804231732104073502336905869521.05... 
\]
so that using the Brauer relation we obtain $h_K = 27000 = 2^3 \cdot 3^3 \cdot 5^3$.

\begin{theorem}\label{thm:s5}
    Let $K$ be the splitting field of $x^5 - x^3 - x^2 - x + 1 \in \Q[x]$, 
    an $S_{5}$-extension of $\Q$.
    Then $K$ is $p$-rational for all $7 \leq p < 10^{8}$.
\end{theorem}

\begin{proof}
Observe that for every prime $p \geq 7$, the field
$K$ is $p$-rational if and only if it is quasi-$p$-rational by
Lemma~\ref{lem:rat-implies-quasi-rat-implies-leopoldt}(c) and
the fact that $p \nmid h_{K}$.
Hence the desired result follows from an application of Algorithm~\ref{alg:quasi-p-rat},
with the full unit group $\mathcal{O}_{K}^{\times}$ given as an input (see Remark~\ref{rmk:skip-steps-when-unit-group-known}).
Step (2) is not executed since $p$ is odd.
Step (3) is only executed in the cases $p=79,89$.
\end{proof}

\begin{corollary}
Let $K$ be as above and let $G=\Gal(K/\Q)$.
Let $7 \leq p < 10^{8}$ be a prime number. 
Let $n \in \{4,5,6\}$ and let $\omega$ be an irreducible representation of $G$ over 
$\Q_{p}$ of degree $n$.
Then there exists a continuous representation
\[
\rho : \Gal(K_{S_{p}}(p)/\Q) \longrightarrow \mathrm{GL}_{n}(\Z_{p})
\]
with open image such that the residual representations $\overline{\rho}$ 
and $\overline{\omega}$ over $\F_{p}$ are isomorphic.
\end{corollary}

\begin{proof}
The unique quadratic subfield of $K$ is $\Q(\sqrt{-79 \cdot 89})$, which is imaginary.
Therefore the result follows from Theorem~\ref{thm:s5} and 
Proposition~\ref{prop:Galois-reps-A5-S5}(b).
\end{proof}

\subsection{Galois representations from totally complex cyclic extensions}
We first set up some notation. Let $p$ be an odd prime number. 
Let $\chi_{\mathrm{cyc}} : \Gal(\overline{\Q}/\Q) \longrightarrow \Z_{p}^{\times}$ be the $p$-adic cyclotomic character and let $\kappa$ denote the composition of $\chi_{\mathrm{cyc}}$ with the projection onto the second factor of the canonical decomposition $\Z_{p}^{\times} = \mu_{p-1} \times (1+p\Z_{p})$.
If $K/\Q$ is a finite Galois extension then $K_{S_{p}}(p)/\Q$ is Galois and 
$\kappa$ factors through $\Gal(K_{S_{p}}(p)/\Q)$.
For $n \geq 2$, let $S_{n}^{0}(\Z_{p})$ be the explicit Sylow pro-$p$ subgroup
of $\mathrm{SL}_{n}(\Z_{p})$ defined in \cite[\S 5.4]{MR3512524}.

\begin{prop}\label{prop:tot-complex-cyclic}
Let $K/\Q$ be a totally complex finite cyclic extension of $\Q$.
Let $n \geq 2$ be such that $[K:\Q] \geq 4 \lfloor n/2 \rfloor$.
Let $p$ be an odd prime number such that $[K:\Q]$ divides $p-1$ and $K$ is $p$-rational. 
Then there exists a continuous representation
\[
\rho_{0} : \Gal(K_{S_{p}}(p)/\Q) \longrightarrow \mathrm{GL}_{n}(\Z_{p})
\]
such that $\rho_{0}(\Gal(K_{S_{p}}(p)/K)) = S_{n}^{0}(\Z_{p})$.
Furthermore, $\rho = \rho_{0} \otimes \kappa$ is a continuous representation from
$\Gal(K_{S_{p}}(p)/\Q)$ to $\mathrm{GL}_{n}(\Z_{p})$ with open image. 
\end{prop}

\begin{proof}
This is \cite[Proposition~6.6]{MR3512524}.    
\end{proof}

\begin{example}\label{ex:cyclic-p-rat}
By using the computational results for real cyclotomic fields 
described in \S \ref{subsec:comp-results-real-cyclo} 
and applying Remark~\ref{rmk:p-rat-full-cyclo}, 
it is easy to give many triples $(K,n,p)$ that satisfy the hypotheses of
Proposition~\ref{prop:tot-complex-cyclic}.
For instance, $K=\Q(\zeta_{101})$ is cyclic and of degree $100$ over $\Q$, 
so $p$ must satisfy $p \equiv 1 \bmod{100}$.
By the results of \S \ref{subsec:comp-results-real-cyclo}, the
field $\Q(\zeta_{101})^{+}$ is $p$-rational for all such $p$ with $p<10^{7}$, except
for $p \in \{ 101, 401, 5501, 19301\}$.
The entry for conductor $101$ in \cite[Tables, \S 3]{MR1421575} shows that
we may need to exclude $p \in \{ 601, 18701 \}$ as well.
In other words, $K=\Q(\zeta_{101})$ is $p$-rational for all primes
$p$ satisfying $p \equiv 1 \bmod{100}$ and $p <10^{7}$, except for 
$p \in \{ 101, 401, 5501, 19301\}$, and possibly also $p \in \{ 601, 18701 \}$.
Moreover, we may choose any $n$ in the range $2 \leq n \leq 51$. 
\end{example}

\appendix

\section{Number fields for unit group computation}\label{appendix:number-fields}

\renewcommand{\arraystretch}{1.2}
\footnotesize
\begin{xltabular}{\linewidth}{c X}
\caption{Defining polynomials for the number fields considered in \S \ref{subsec:example:uncond}}
\label{tab:polys-for-fields} \\
\toprule 
Field & Defining polynomial \\
\midrule
\endfirsthead
\toprule
Field & Defining polynomial \\
\midrule
\endhead
\endfoot
\endlastfoot
1 & $x^{5} - 68410 x^{3} - 136820 x^{2} + 1120453185 x - 876769924$ \\
2 & $x^{5} - 51632410 x^{3} - 48818443655 x^{2} + 439631951438910 x + 74547141886471529$ \\
3 & $x^{5} - x^{4} - 126480 x^{3} - 7563528 x^{2} + 2211380784 x + 52157866032$ \\
\midrule
4 & $x^{6} - 113004 x^{4} - 200896 x^{3} + 3044676189 x^{2} - 17026538688 x - 15145401762606$ \\
5 & $x^{6} - 383534 x^{4} + 10522638824 x^{2} - 72174779693816$ \\
6 & $x^{6} - 88680 x^{4} - 4324628 x^{3} + 1780900581 x^{2} + 162990904692 x + 2233922149646$ \\
\midrule
7 & $x^{8} - 9668 x^{6} + 14526170 x^{4} - 4580040976 x^{2} + 224422007824$ \\
8 & $x^{8} - 13964 x^{6} + 26908628 x^{4} - 5898547204 x^{2} + 178431052921$ \\
9 & $x^{8} - 14284 x^{6} + 58093028 x^{4} - 77583417444 x^{2} + 29501094483081$ \\
\midrule
10 & $x^{11} - x^{10} - 450 x^{9} + 393 x^{8} + 65785 x^{7} - 76121 x^{6} - 3728249 x^{5} + 6483322 x^{4} + 89323866 x^{3} - 205338716 x^{2} - 745717454 x + 2032282741$ \\
11 & $x^{11} - 1265 x^{9} + 10373 x^{8} + 461472 x^{7} - 7026316 x^{6} - 19353235 x^{5} + 853505367 x^{4} - 3771314393 x^{3} - 11665833289 x^{2} + 103006408615 x - 152668956353$ \\
12 & $x^{11} - x^{10} - 310 x^{9} + 395 x^{8} + 26441 x^{7} - 57583 x^{6} - 744575 x^{5} + 2220564 x^{4} + 2568726 x^{3} - 5088204 x^{2} - 1360638 x + 996543$ \\
\midrule
13 & $x^{12} + 4 x^{11} - 379 x^{10} - 1534 x^{9} + 42138 x^{8} + 119516 x^{7} - 1762831 x^{6} - 1608216 x^{5} + 30789468 x^{4} - 19429774 x^{3} - 182930861 x^{2} + 325707704 x - 111088243$ \\
14 & $x^{12} - 3 x^{11} - 187 x^{10} + 485 x^{9} + 11475 x^{8} - 27733 x^{7} - 255001 x^{6} + 623111 x^{5} + 1442720 x^{4} - 3990420 x^{3} + 44128 x^{2} + 2671616 x + 262144$ \\
15 & $x^{12} + 5 x^{11} - 257 x^{10} - 1535 x^{9} + 14705 x^{8} + 73035 x^{7} - 302075 x^{6} - 855725 x^{5} + 1060290 x^{4} + 2216060 x^{3} - 711992 x^{2} - 402720 x + 91904$ \\
\midrule
16 & $x^{13} - 78 x^{11} + 65 x^{10} + 2080 x^{9} - 2457 x^{8} - 24128 x^{7} + 27027 x^{6} + 137683 x^{5} - 110214 x^{4} - 376064 x^{3} + 128206 x^{2} + 363883 x + 12167$ \\
17 & $x^{13} - x^{12} - 276 x^{11} + 1967 x^{10} + 8169 x^{9} - 109375 x^{8} + 114077 x^{7} + 1684091 x^{6} - 4924742 x^{5} - 5465967 x^{4} + 34969245 x^{3} - 20502539 x^{2} - 55304818 x + 57031547$ \\
18 & $x^{13} + x^{12} - 144 x^{11} - 161 x^{10} + 6530 x^{9} + 9620 x^{8} - 109398 x^{7} - 196143 x^{6} + 512628 x^{5} + 917970 x^{4} - 650724 x^{3} - 1134730 x^{2} + 253950 x + 409375$ \\
\midrule
19 & $x^{16} + 4 x^{15} - 58 x^{14} - 168 x^{13} + 1484 x^{12} + 2324 x^{11} - 20300 x^{10} - 6912 x^{9} + 141871 x^{8} - 83952 x^{7} - 400064 x^{6} + 489244 x^{5} + 197204 x^{4} - 441588 x^{3} + 64118 x^{2} + 61592 x - 14171$ \\
20 & $x^{16} - 3 x^{15} - 288 x^{14} + 351 x^{13} + 32164 x^{12} + 6818 x^{11} - 1712220 x^{10} - 2050498 x^{9} + 43855038 x^{8} + 73975139 x^{7} - 504661128 x^{6} - 805577043 x^{5} + 2530948457 x^{4} + 2475352300 x^{3} - 4676812808 x^{2} - 174707192 x + 1104680656$ \\
21 & $x^{16} - 164 x^{14} + 8282 x^{12} - 180400 x^{10} + 1846640 x^{8} - 8606720 x^{6} + 16998272 x^{4} - 12049408 x^{2} + 1721344$ \\
\midrule
22 & $x^{17} - x^{16} - 48 x^{15} + 105 x^{14} + 763 x^{13} - 2579 x^{12} - 3653 x^{11} + 23311 x^{10} - 11031 x^{9} - 74838 x^{8} + 107759 x^{7} + 50288 x^{6} - 198615 x^{5} + 102976 x^{4} + 58507 x^{3} - 75722 x^{2} + 25763 x - 2837$ \\
23 & $x^{17} - x^{16} - 64 x^{15} + 43 x^{14} + 1478 x^{13} - 932 x^{12} - 16008 x^{11} + 12183 x^{10} + 86347 x^{9} - 84507 x^{8} - 213223 x^{7} + 271237 x^{6} + 152800 x^{5} - 314540 x^{4} + 100605 x^{3} + 20132 x^{2} - 13981 x + 1681$ \\
24 & $x^{17} - x^{16} - 112 x^{15} + 47 x^{14} + 3976 x^{13} - 4314 x^{12} - 64388 x^{11} + 136247 x^{10} + 422013 x^{9} - 1631073 x^{8} + 411840 x^{7} + 5840196 x^{6} - 11894369 x^{5} + 10635750 x^{4} - 4739804 x^{3} + 938485 x^{2} - 54850 x + 619$ \\
\midrule
25 & $x^{20} - 90 x^{18} - 165 x^{17} + 2655 x^{16} + 8844 x^{15} - 23405 x^{14} - 124135 x^{13} + 25635 x^{12} + 705485 x^{11} + 462661 x^{10} - 1893595 x^{9} - 2174625 x^{8} + 2297075 x^{7} + 3851885 x^{6} - 700216 x^{5} - 2796615 x^{4} - 556105 x^{3} + 551050 x^{2} + 127380 x - 33329$ \\
26 & $x^{20} - 5 x^{19} - 95 x^{18} + 465 x^{17} + 3575 x^{16} - 16976 x^{15} - 69870 x^{14} + 318920 x^{13} + 774865 x^{12} - 3377945 x^{11} - 4973199 x^{10} + 20746175 x^{9} + 18042200 x^{8} - 73300815 x^{7} - 35081690 x^{6} + 141973034 x^{5} + 35360815 x^{4} - 134710590 x^{3} - 22696440 x^{2} + 46079055 x + 13060431$ \\
27 & $x^{20} - 3 x^{19} - 119 x^{18} + 177 x^{17} + 5889 x^{16} - 1381 x^{15} - 152812 x^{14} - 117981 x^{13} + 2164979 x^{12} + 3491841 x^{11} - 15663089 x^{10} - 38883595 x^{9} + 42546574 x^{8} + 187263905 x^{7} + 52059719 x^{6} - 314665312 x^{5} - 318011368 x^{4} + 4965810 x^{3} + 106222904 x^{2} + 21505161 x - 6072119$ \\
\bottomrule
\end{xltabular}
\renewcommand{\arraystretch}{1.0}

\bibliography{SchirokauerBib}

\newcommand{\etalchar}[1]{$^{#1}$}
\providecommand{\bysame}{\leavevmode\hbox to3em{\hrulefill}\thinspace}
\providecommand{\MR}{\relax\ifhmode\unskip\space\fi MR }
\providecommand{\MRhref}[2]{%
  \href{http://www.ams.org/mathscinet-getitem?mr=#1}{#2}
}
\providecommand{\href}[2]{#2}
\begin{thebibliography}{BvHKS09}

\bibitem[Adl91]{10.1145/103418.103432}
L.~M. Adleman, \emph{Factoring numbers using singular integers}, Proceedings of
  the Twenty-Third Annual ACM Symposium on Theory of Computing (New York, NY,
  USA), STOC '91, Association for Computing Machinery, 1991, pp.~64--71.

\bibitem[Are91]{MR1151862}
B.~Arenz, \emph{Computing fundamental units from independent units},
  Computational number theory ({D}ebrecen, 1989), de Gruyter, Berlin, 1991,
  pp.~163--171. \MR{1151862}

\bibitem[Ass95]{MR1324685}
J.~Assim, \emph{Codescente en {$K$}-th\'eorie \'etale et corps de nombres},
  Manuscripta Math. \textbf{86} (1995), no.~4, 499--518. \MR{1324685}

\bibitem[BCP97]{MR1484478}
W.~Bosma, J.~Cannon, and C.~Playoust, \emph{The {M}agma algebra system. {I}.
  {T}he user language}, J. Symbolic Comput. \textbf{24} (1997), no.~3-4,
  235--265, Computational algebra and number theory (London, 1993).
  \MR{1484478}

\bibitem[BF13]{MR3207410}
J.-F. Biasse and C.~Fieker, \emph{Improved techniques for computing the ideal
  class group and a system of fundamental units in number fields}, A{NTS}
  {X}---{P}roceedings of the {T}enth {A}lgorithmic {N}umber {T}heory
  {S}ymposium, Open Book Ser., vol.~1, Math. Sci. Publ., Berkeley, CA, 2013,
  pp.~113--133. \MR{3207410}

\bibitem[BFHP22]{MR4440537}
J.-F. Biasse, C.~Fieker, T.~Hofmann, and A.~Page, \emph{Norm relations and
  computational problems in number fields}, J. Lond. Math. Soc. (2)
  \textbf{105} (2022), no.~4, 2373--2414. \MR{4440537}

\bibitem[BFL24]{zbMATH07854947}
O.~Bernard, P.-A. Fouque, and A.~Lesavourey, \emph{Computing {{\(e\)}}-th roots
  in number fields}, 2024 Proceedings of the Symposium on Algorithm Engineering
  and Experiments (ALENEX), Society for Industrial {and} Applied Mathematics,
  2024, pp.~207--219.

\bibitem[BL25]{MR4973192}
Z.~Bouazzaoui and D.~Lim, \emph{On the {G}alois structure of units in totally
  real {$p$}-rational number fields}, New York J. Math. \textbf{31} (2025),
  1439--1481. \MR{4973192}

\bibitem[BM21]{MR4192837}
Y.~Benmerieme and A.~Movahhedi, \emph{Multi-quadratic {$p$}-rational number
  fields}, J. Pure Appl. Algebra \textbf{225} (2021), no.~9, Paper No. 106657,
  17. \MR{4192837}

\bibitem[Bou20]{MR4129146}
Z.~Bouazzaoui, \emph{Fibonacci numbers and real quadratic {$p$}-rational
  fields}, Period. Math. Hungar. \textbf{81} (2020), no.~1, 123--133.
  \MR{4129146}

\bibitem[BR20]{MR4158582}
R.~Barbulescu and J.~Ray, \emph{Numerical verification of the
  {C}ohen-{L}enstra-{M}artinet heuristics and of {G}reenberg's
  {$p$}-rationality conjecture}, J. Th\'eor. Nombres Bordeaux \textbf{32}
  (2020), no.~1, 159--177. \MR{4158582}

\bibitem[BS87]{MR904010}
J.~Buchmann and J.~W. Sands, \emph{An algorithm for testing {L}eopoldt's
  conjecture}, J. Number Theory \textbf{27} (1987), no.~1, 92--105. \MR{904010}

\bibitem[Buc90]{MR1104698}
J.~Buchmann, \emph{A subexponential algorithm for the determination of class
  groups and regulators of algebraic number fields}, S\'{e}minaire de
  {T}h\'{e}orie des {N}ombres, {P}aris 1988--1989, Progr. Math., vol.~91,
  Birkh\"{a}user Boston, Boston, MA, 1990, pp.~27--41. \MR{1104698}

\bibitem[BvHKS09]{MR2537701}
K.~Belabas, M.~van Hoeij, J.~Kl\"uners, and A.~Steel, \emph{Factoring
  polynomials over global fields}, J. Th\'eor. Nombres Bordeaux \textbf{21}
  (2009), no.~1, 15--39. \MR{2537701}

\bibitem[Coa77]{MR460282}
J.~Coates, \emph{{$p$}-adic {$L$}-functions and {I}wasawa's theory}, Algebraic
  number fields: {$L$}-functions and {G}alois properties ({P}roc. {S}ympos.,
  {U}niv. {D}urham, {D}urham, 1975), Academic Press, London-New York, 1977,
  pp.~269--353. \MR{460282}

\bibitem[Coh93]{MR1228206}
H.~Cohen, \emph{A course in computational algebraic number theory}, Graduate
  Texts in Mathematics, vol. 138, Springer-Verlag, Berlin, 1993.

\bibitem[Coh00]{MR1728313}
\bysame, \emph{Advanced topics in computational number theory}, Graduate Texts
  in Mathematics, vol. 193, Springer-Verlag, New York, 2000. \MR{1728313}

\bibitem[DEF{\etalchar{+}}25]{OSCAR-book}
W.~Decker, C.~Eder, C.~Fieker, M.~Horn, and M.~Joswig (eds.), \emph{The
  {C}omputer {A}lgebra {S}ystem {OSCAR}: {A}lgorithms and {E}xamples},
  Algorithms and {C}omputation in {M}athematics, vol.~32, Springer, 2025.

\bibitem[DP20]{MR4083092}
R.~Davis and R.~Pries, \emph{Cohomology groups of {F}ermat curves via ray class
  fields of cyclotomic fields}, J. Algebra \textbf{554} (2020), 78--105.
  \MR{4083092}

\bibitem[EV26]{10.1145/3815436.3815477}
A.-S. Elsenhans and J.~Voight, \emph{Computing class groups and unit groups in
  {M}agma}, Proceedings of the 2026 International Symposium on Symbolic and
  Algebraic Computation (New York, NY, USA), ISSAC '26, Association for
  Computing Machinery, 2026, pp.~162--171.

\bibitem[FHHJ17]{MR3703682}
C.~Fieker, W.~Hart, T.~Hofmann, and F.~Johansson, \emph{Nemo/{H}ecke: computer
  algebra and number theory packages for the {J}ulia programming language},
  I{SSAC}'17---{P}roceedings of the 2017 {ACM} {I}nternational {S}ymposium on
  {S}ymbolic and {A}lgebraic {C}omputation, ACM, New York, 2017, pp.~157--164.
  \MR{3703682}

\bibitem[FHS19]{MR3952015}
C.~Fieker, T.~Hofmann, and C.~Sircana, \emph{On the construction of class
  fields}, Proceedings of the {T}hirteenth {A}lgorithmic {N}umber {T}heory
  {S}ymposium, Open Book Ser., vol.~2, Math. Sci. Publ., Berkeley, CA, 2019,
  pp.~239--255. \MR{3952015}

\bibitem[FP08]{MR2441075}
C.~Fieker and M.~E. Pohst, \emph{A lower regulator bound for number fields}, J.
  Number Theory \textbf{128} (2008), no.~10, 2767--2775. \MR{2441075}

\bibitem[Ge93]{ge}
G.~Ge, \emph{Algorithms related to multiplicative representations of algebraic
  numbers}, Ph.D. thesis, University of California, Berkeley, 1993.

\bibitem[GJ89]{MR1017575}
G.~Gras and J.-F. Jaulent, \emph{Sur les corps de nombres r\'eguliers}, Math.
  Z. \textbf{202} (1989), no.~3, 343--365. \MR{1017575}

\bibitem[GK89]{MR1008802}
R.~Gold and J.~M. Kim, \emph{Bases for cyclotomic units}, Compositio Math.
  \textbf{71} (1989), no.~1, 13--27. \MR{1008802}

\bibitem[Gor01]{MR1826497}
E.~Z. Goren, \emph{Hasse invariants for {H}ilbert modular varieties}, Israel J.
  Math. \textbf{122} (2001), 157--174. \MR{1826497}

\bibitem[Gra03]{MR1941965}
G.~Gras, \emph{Class field theory}, Springer Monographs in Mathematics,
  Springer-Verlag, Berlin, 2003, From theory to practice, Translated from the
  French manuscript by Henri Cohen. \MR{1941965}

\bibitem[Gra14]{MR3320496}
\bysame, \emph{On the structure of the {G}alois group of the {A}belian closure
  of a number field}, J. Th\'eor. Nombres Bordeaux \textbf{26} (2014), no.~3,
  635--655. \MR{3320496}

\bibitem[Gra16]{MR3492629}
\bysame, \emph{Les {$\theta$}-r\'{e}gulateurs locaux d'un nombre
  alg\'{e}brique: conjectures {$p$}-adiques}, Canad. J. Math. \textbf{68}
  (2016), no.~3, 571--624. \MR{3492629}

\bibitem[Gra19]{MR4491965}
\bysame, \emph{On {$p$}-rationality of number fields.
  {A}pplications---{PARI}/{GP} programs}, Publications math\'ematiques de
  {B}esan\c con. {A}lg\`ebre et th\'eorie des nombres. 2019/2, Publ. Math.
  Besan\c con Alg\`ebre Th\'eorie Nr., vol. 2019/2, Presses Univ.
  Franche-Comt\'e, Besan\c con, 2019, pp.~29--51. \MR{4491965}

\bibitem[Gre16]{MR3512524}
R.~Greenberg, \emph{Galois representations with open image}, Ann. Math. Qu\'e.
  \textbf{40} (2016), no.~1, 83--119. \MR{3512524}

\bibitem[HM19]{MR3933912}
F.~Hajir and C.~Maire, \emph{Prime decomposition and the {I}wasawa
  {MU}-invariant}, Math. Proc. Cambridge Philos. Soc. \textbf{166} (2019),
  no.~3, 599--617. \MR{3933912}

\bibitem[HMR21]{MR4308183}
F.~Hajir, C.~Maire, and R.~Ramakrishna, \emph{Cutting towers of number fields},
  Ann. Math. Qu\'e. \textbf{45} (2021), no.~2, 321--345. \MR{4308183}

\bibitem[Ich25]{MR5078208}
H.~Ichimura, \emph{On divisibility of the class number of a real abelian field
  and {F}ermat quotients}, Kyushu J. Math. \textbf{79} (2025), no.~2, 257--267.
  \MR{5078208}

\bibitem[JN20]{MR4111943}
H.~Johnston and A.~Nickel, \emph{On the {$p$}-adic {S}tark conjecture at {$s =
  1$} and applications}, J. Lond. Math. Soc. (2) \textbf{101} (2020), no.~3,
  1320--1354, With an appendix by T. Hofmann, H. Johnston and A. Nickel.
  \MR{4111943}

\bibitem[JNQD93]{MR1265910}
J.-F. Jaulent and T.~Nguyen Quang~Do, \emph{Corps {$p$}-rationnels, corps
  {$p$}-r\'eguliers, et ramification restreinte}, J. Th\'eor. Nombres Bordeaux
  \textbf{5} (1993), no.~2, 343--363. \MR{1265910}

\bibitem[KM01]{MR1901356}
J.~Kl\"uners and G.~Malle, \emph{A database for field extensions of the
  rationals}, LMS J. Comput. Math. \textbf{4} (2001), 182--196. \MR{1901356}

\bibitem[KN25]{MR5022766}
P.~Koprowski and J.~Novacoski, \emph{Computing the local group of prime-power
  classes}, I{SSAC}'25---{P}roceedings of the 2025 {I}nternational {S}ymposium
  on {S}ymbolic and {A}lgebraic {C}omputation, ACM, New York, 2025, pp.~34--41.
  \MR{5022766}

\bibitem[Ku{\v{c}}04]{MR2183100}
R.~Ku{\v{c}}era, \emph{Circular units and class groups of abelian fields}, Ann.
  Sci. Math. Qu\'ebec \textbf{28} (2004), no.~1-2, 121--136. \MR{2183100}

\bibitem[Len83]{MR774816}
A.~K. Lenstra, \emph{Factoring polynomials over algebraic number fields},
  Computer algebra ({L}ondon, 1983), Lecture Notes in Comput. Sci., vol. 162,
  Springer, Berlin, 1983, pp.~245--254. \MR{774816}

\bibitem[Len92]{MR1129315}
H.~W. Lenstra, Jr., \emph{Algorithms in algebraic number theory}, Bull. Amer.
  Math. Soc. (N.S.) \textbf{26} (1992), no.~2, 211--244. \MR{1129315}

\bibitem[Let90]{MR1057325}
G.~Lettl, \emph{A note on {T}haine's circular units}, J. Number Theory
  \textbf{35} (1990), no.~2, 224--226. \MR{1057325}

\bibitem[Lim22]{MR4330938}
D.~Lim, \emph{On {$p$}-rationality of {$\Bbb Q(\zeta_{2l+1})^+$} for {S}ophie
  {G}ermain primes {$l$}}, J. Number Theory \textbf{231} (2022), 378--400.
  \MR{4330938}

\bibitem[LS18]{MR3749417}
H.~W. Lenstra, Jr. and A.~Silverberg, \emph{Algorithms for commutative algebras
  over the rational numbers}, Found. Comput. Math. \textbf{18} (2018), no.~1,
  159--180. \MR{3749417}

\bibitem[MNQD90]{MR1042770}
A.~Movahhedi and T.~Nguyen Quang~Do, \emph{Sur l'arithm\'etique des corps de
  nombres {$p$}-rationnels}, S\'eminaire de {T}h\'eorie des {N}ombres, {P}aris
  1987--88, Progr. Math., vol.~81, Birkh\"auser Boston, Boston, MA, 1990,
  pp.~155--200. \MR{1042770}

\bibitem[Mov88]{MovahhediPhD}
A.~Movahhedi, \emph{Sur les $p$-extensions des corps $p$-rationnels}, Ph.D.
  thesis, Paris VII, 1988.

\bibitem[Mov90]{MR1124802}
\bysame, \emph{Sur les {$p$}-extensions des corps {$p$}-rationnels}, Math.
  Nachr. \textbf{149} (1990), 163--176. \MR{1124802}

\bibitem[Neu99]{MR1697859}
J.~Neukirch, \emph{Algebraic number theory}, Grundlehren der mathematischen
  Wissenschaften, vol. 322, Springer-Verlag, Berlin, 1999, Translated from the
  1992 German original and with a note by Norbert Schappacher, With a foreword
  by G. Harder. \MR{1697859}

\bibitem[NQD25]{do2025greenbergsgeneralizedconjecturefamilies}
T.~Nguyen Quang~Do, \emph{On {G}reenberg's generalized conjecture for families
  of number fields}, 2025, arXiv:2505.07529.

\bibitem[NSW08]{MR2392026}
J.~Neukirch, A.~Schmidt, and K.~Wingberg, \emph{Cohomology of number fields},
  second ed., Grundlehren der mathematischen Wissenschaften, vol. 323,
  Springer-Verlag, Berlin, 2008. \MR{2392026}

\bibitem[Poh93]{MR1243639}
M.~Pohst, \emph{Computational algebraic number theory}, DMV Seminar, vol.~21,
  Birkh\"{a}user Verlag, Basel, 1993. \MR{1243639}

\bibitem[Poh94]{MR1273458}
\bysame, \emph{On computing fundamental units}, J. Number Theory \textbf{47}
  (1994), no.~1, 93--105. \MR{1273458}

\bibitem[Poh05]{MR2168610}
\bysame, \emph{Factoring polynomials over global fields. {I}}, J. Symbolic
  Comput. \textbf{39} (2005), no.~6, 617--630. \MR{2168610}

\bibitem[PV15]{MR3266966}
F.~Pitoun and F.~Varescon, \emph{Computing the torsion of the {$p$}-ramified
  module of a number field}, Math. Comp. \textbf{84} (2015), no.~291, 371--383.
  \MR{3266966}

\bibitem[PZ77]{MR498486}
M.~Pohst and H.~Zassenhaus, \emph{An effective number geometric method of
  computing the fundamental units of an algebraic number field}, Math. Comp.
  \textbf{31} (1977), no.~139, 754--770. \MR{498486}

\bibitem[PZ97]{MR1483321}
\bysame, \emph{Algorithmic algebraic number theory}, Encyclopedia of
  Mathematics and its Applications, vol.~30, Cambridge University Press,
  Cambridge, 1997, Revised reprint of the 1989 original. \MR{1483321}

\bibitem[Sch93]{MR1253502}
O.~Schirokauer, \emph{Discrete logarithms and local units}, Philos. Trans. Roy.
  Soc. London Ser. A \textbf{345} (1993), no.~1676, 409--423. \MR{1253502}

\bibitem[Ser79]{MR554237}
J.-P. Serre, \emph{Local fields}, Graduate Texts in Mathematics, vol.~67,
  Springer-Verlag, New York-Berlin, 1979, Translated from the French by Marvin
  Jay Greenberg. \MR{554237}

\bibitem[Shp99]{MR1745660}
I.~E. Shparlinski, \emph{Finite fields: theory and computation}, Mathematics
  and its Applications, vol. 477, Kluwer Academic Publishers, Dordrecht, 1999,
  The meeting point of number theory, computer science, coding theory and
  cryptography. \MR{1745660}

\bibitem[Sim02]{MR1898758}
D.~Simon, \emph{Solving norm equations in relative number fields using
  {$S$}-units}, Math. Comp. \textbf{71} (2002), no.~239, 1287--1305.
  \MR{1898758}

\bibitem[Sin78]{MR485778}
W.~Sinnott, \emph{On the {S}tickelberger ideal and the circular units of a
  cyclotomic field}, Ann. of Math. (2) \textbf{108} (1978), no.~1, 107--134.
  \MR{485778}

\bibitem[Sin81]{MR595586}
\bysame, \emph{On the {S}tickelberger ideal and the circular units of an
  abelian field}, Invent. Math. \textbf{62} (1980/81), no.~2, 181--234.
  \MR{595586}

\bibitem[Sto00]{storjohann}
A.~Storjohann, \emph{{Algorithms for matrix canonical forms}}, Ph.{D}. thesis,
  ETH Z\"urich, 2000.

\bibitem[Tat84]{MR782485}
J.~Tate, \emph{Les conjectures de {S}tark sur les fonctions {$L$} d'{A}rtin en
  {$s=0$}}, Progress in Mathematics, vol.~47, Birkh\"auser Boston, Inc.,
  Boston, MA, 1984. \MR{782485}

\bibitem[{The}26]{PARI2}
{The PARI~Group}, Univ. Bordeaux, \emph{{PARI/GP version \texttt{2.17.4}}},
  2026, available from \url{https://pari.math.u-bordeaux.fr/}.

\bibitem[Was97]{MR1421575}
L.~C. Washington, \emph{Introduction to cyclotomic fields}, second ed.,
  Graduate Texts in Mathematics, vol.~83, Springer-Verlag, New York, 1997.
  \MR{1421575}

\bibitem[Yee25]{yee2025unconditional}
R.~K.-M. Yee, \emph{Unconditional verification of unit groups of number
  fields}, Ph.D. thesis, University of Calgary, 2025.

\bibitem[Zim81]{MR604833}
R.~Zimmert, \emph{Ideale kleiner {N}orm in {I}dealklassen und eine
  {R}egulatorabsch\"{a}tzung}, Invent. Math. \textbf{62} (1981), no.~3,
  367--380. \MR{604833}

\end{thebibliography}
\bibliographystyle{amsalpha}

\end{document}